\documentclass{amsart}

\usepackage[T1]{fontenc}
\usepackage{enumerate, amsmath, amsfonts, amssymb, amsthm, mathrsfs, wasysym, graphics, graphicx, xcolor, url, hyperref, hypcap,  shuffle, xargs, multicol, overpic, pdflscape, multirow, hvfloat, minibox, accents, array, multido, xifthen, a4wide, ae, aecompl, blkarray, pifont, mathtools, etoolbox, dsfont}
\usepackage{marginnote}
\hypersetup{colorlinks=true, citecolor=darkblue, linkcolor=darkblue}
\usepackage[all]{xy}
\usepackage[bottom]{footmisc}
\usepackage{tikz}
\usepackage{tikz-qtree}
\usetikzlibrary{trees, decorations, decorations.markings, shapes, arrows, matrix, calc, fit, intersections, patterns, angles}
\usepackage[external]{forest}
\graphicspath{{figures/}{figures/nodes/}}
\makeatletter\def\input@path{{figures/}}\makeatother
\usepackage{caption}
\usepackage[noabbrev,capitalise]{cleveref}
\usepackage[export]{adjustbox}
\usepackage{ulem}
\usepackage{picins/picins}

\newtheorem{theorem}{Theorem}
\newtheorem{corollary}[theorem]{Corollary}
\newtheorem{proposition}[theorem]{Proposition}

\newtheorem*{theorem*}{Theorem}

\theoremstyle{definition}
\newtheorem{definition}[theorem]{Definition}
\newtheorem{example}[theorem]{Example}
\newtheorem{remark}[theorem]{Remark}

\newtheorem{notation}[theorem]{Notation}

\crefname{notation}{Notation}{Notations}
 
\newcommand{\R}{\mathbb{R}} 
\newcommand{\N}{\mathbb{N}} 
\renewcommand{\c}[1]{\mathcal{#1}} 
\renewcommand{\b}[1]{\boldsymbol{#1}} 
\newcommand{\f}[1]{\mathfrak{#1}} 

\newcommand{\set}[2]{\left\{ #1 \;\middle|\; #2 \right\}} 
\newcommand{\bigset}[2]{\big\{ #1 \;\big|\; #2 \big\}} 
\newcommand{\ssm}{\smallsetminus} 
\newcommand{\dotprod}[2]{\left\langle \, #1 \; \middle| \; #2 \, \right\rangle} 
\newcommand{\one}{\b{1}} 
\newcommand{\eqdef}{\mbox{\,\raisebox{0.2ex}{\scriptsize\ensuremath{\mathrm:}}\ensuremath{=}\,}} 
\newcommand{\simplex}{\triangle} 

\DeclareMathOperator{\conv}{conv} 
\DeclareMathOperator{\rev}{rev} 

\newcommand{\ie}{\textit{i.e.}~} 
\newcommand{\aka}{\textit{a.k.a.}~} 
\definecolor{darkblue}{rgb}{0,0,0.7} 
\definecolor{green}{RGB}{57,181,74} 
\definecolor{violet}{RGB}{147,39,143} 
\newcommand{\darkblue}{\color{darkblue}} 
\newcommand{\defn}[1]{\textsl{\darkblue #1}} 
\newcommand{\OEIS}[1]{{\rm \href{http://oeis.org/#1}{\texttt{#1}}}}

\usepackage{todonotes}

\makeatletter
\def\part{\@startsection{part}{1}%
\z@{.7\linespacing\@plus\linespacing}{.8\linespacing}%
{\LARGE\sffamily\centering}}
\@addtoreset{section}{part}
\makeatother

\makeatletter
\def\l@part{\@tocline{1}{8pt}{0pc}{}{}}
\def\l@section{\@tocline{1}{3pt}{0pc}{}{}}
\makeatother
\let\oldtocpart=\tocpart
\renewcommand{\tocpart}[2]{\sc\large\oldtocpart{#1}{#2}}
\let\oldtocsection=\tocsection
\renewcommand{\tocsection}[2]{\bf\oldtocsection{#1}{#2}}
\let\oldtocsubsubsection=\tocsubsubsection
\renewcommand{\tocsubsubsection}[2]{\quad\oldtocsubsubsection{#1}{#2}}

\makeatletter
\newcommand{\ostar}{\mathbin{\mathpalette\make@circled\star}}
\newcommand{\make@circled}[2]{%
  \ooalign{$\m@th#1\smallbigcirc{#1}$\cr\hidewidth$\m@th#1#2$\hidewidth\cr}%
}
\newcommand{\smallbigcirc}[1]{%
  \vcenter{\hbox{\scalebox{0.77778}{$\m@th#1\bigcirc$}}}%
}
\makeatother
\newcommand{\joinGraph}{\ostar} 

\newcommand{\shuffleDP}{\star} 

\newcommandx{\poset}[1][1=P]{\mathbb{#1}} 
\newcommand{\less}{\vartriangleleft} 
\newcommand{\more}{\vartriangleright} 

\newcommand{\meet}{\wedge} 
\newcommand{\join}{\vee} 
\newcommand{\projDown}{\pi_\downarrow} 
\newcommand{\projUp}{\pi^\uparrow} 

\newcommandx{\Fan}[1][1=n]{\mathcal{F}(#1)} 
\newcommand{\polytope}[1]{\mathds{#1}} 
\newcommandx{\DefoPerms}[1][1=n]{\polytope{DP}(#1)} 
\newcommandx{\Perm}[1][1=n]{\polytope{P}\mathrm{erm}(#1)} 
\newcommandx{\Asso}[1][1=n]{\polytope{A}\mathrm{sso}(#1)} 
\newcommandx{\Ossa}[1][1=n]{\polytope{A}\overline{\mathrm{sso}}(#1)} 
\newcommandx{\Para}[1][1=n]{\polytope{P}\mathrm{ara}(#1)} 
\newcommandx{\Zono}[1][1=G]{\polytope{Z}\mathrm{ono}(#1)} 
\newcommandx{\Point}[1][1=n]{\polytope{P}\mathrm{oint}(#1)} 
\newcommandx{\Multiplihedron}[2][1=m, 2=n]{\polytope{M}\mathrm{ul}(#1, #2)} 
\newcommandx{\Constrainahedron}[2][1=m, 2=n]{\polytope{C}\mathrm{onstr}(#1, #2)} 
\newcommandx{\Biassociahedron}[2][1=m, 2=n]{\polytope{B}\mathrm{ias}(#1, #2)} 

\newcommand{\surjections}[2]{\mathsf{S}(#1,#2)} 
\newcommand{\nc}[1]{\mathsf{NC}(#1)} 
\newcommand{\pp}[2]{\mathsf{PP}_{#1}(#2)} 

\newcommand{\children}{\mathsf{C}} 
\newcommand{\node}[1]{\mathsf{#1}} 
\newcommand*\circled[1]{\tikz[baseline=(char.base)]{
            \node[shape=circle,draw,inner sep=2pt] (char) {#1};}}
\newcommand{\sylv}{\mathrm{sylv}} 
\DeclareMathOperator{\rank}{rk} 
\newcommand{\stump}[1]{\overline{#1}} 
\newcommand{\cut}[1]{\underline{#1}} 
\newcommand{\imagetop}[1]{\vtop{\null\hbox{#1}}} 

\newcommandx{\tree}[3][2=15pt, 3=22pt]{
		\begin{tikzpicture}[level 1/.style={level distance=#2}, level distance=#3]
		\Tree #1
		\end{tikzpicture}
}

\makeatletter
\long\def\ifnodedefined#1#2#3{%
    \@ifundefined{pgf@sh@ns@#1}{#3}{#2}%
}
\makeatother
\newcommand{\paintedTree}[3][1]{
	\imagetop{
		\scalebox{#1}{
		\begin{forest}
			for tree={l=2pt, l sep=0pt, s sep=1pt, parent anchor=center, child anchor=center, edge={color=black, thick}, if n children=0{tier=bottom}{}}, 
			delay={
				for tree={
					tier/.process={_O= ? p}{}{content}{O{tier}}{O{content}},
					name'/.process={_O= ? p}{}{content}{_{}}{O{content}},
					content={},
				}
			}
			[#2]
			\coordinate[xshift=-2pt] (bbl) at (current bounding box.north west);
			\coordinate[xshift=2pt] (bbr) at (current bounding box.south east);
			\foreach \i in {#3} {
				\ifnodedefined{\i}{
					\node[label={[label distance=-10pt]180:{\color{red}\scriptsize\i}}] (l) at (bbl|-\i.center) {};
					\node[label={[label distance=-10pt]0:{\color{red}\scriptsize\i}}] (r) at (bbr|-\i.center) {};
					\draw[red] (l) -- (r);
				}{}
			}
		\end{forest}
		}
	}
}
\newcommandx{\PT}[1][1=T]{\mathbb{#1}} 

\newcommandx{\CT}[1][1=T]{\mathbb{#1}} 
\newcommand{\cotree}[4][1]{
	\imagetop{
		\scalebox{#1}{
		\begin{forest}
			for tree={l=2pt, l sep=0pt, s sep=1pt, parent anchor=center, child anchor=center, edge={color=black, thick}, if n children=0{tier=top}{}},
			delay={
				for tree={
					tier/.process={_O= ? p}{}{content}{O{tier}}{O{content}},
					name'/.process={_O= ? p}{}{content}{_{}}{O{content}},
					content={},
				}
			},
			[#2]
			\coordinate[xshift=-2pt] (bbl) at (current bounding box.north west);
			\coordinate[xshift=5pt] (bbr) at (current bounding box.south east);
			\foreach \i in {#4} {
				\ifnodedefined{\i}{
					\node (l) at (bbl|-\i.center) {};
					\node (r) at (bbr|-\i.center) {};
					\draw[red] (l) -- (r);
				}{}
			}
		\end{forest}
		\hspace{-.73cm}
		\begin{forest}
			for tree={l=2pt, l sep=0pt, s sep=1pt, parent anchor=center, child anchor=center, edge={color=black, thick}, if n children=0{tier=bottom}{}}, 
			delay={
				for tree={
					tier/.process={_O= ? p}{}{content}{O{tier}}{O{content}},
					name'/.process={_O= ? p}{}{content}{_{}}{O{content}},
					content={},
				}
			},
			[#3]
			\coordinate[xshift=-5pt] (bbl) at (current bounding box.north west);
			\coordinate[xshift=2pt] (bbr) at (current bounding box.south east);
			\foreach \i in {#4} {
				\ifnodedefined{\i}{
					\node (l) at (bbl|-\i.center) {};
					\node (r) at (bbr|-\i.center) {};
					\draw[red] (l) -- (r);
				}{}
			}
		\end{forest}
		}
	}
}

\newcommandx{\BT}[1][1=T]{\mathbb{#1}} 
\newcommand{\bitree}[4][1]{
	\imagetop{
		\scalebox{#1}{
		\begin{forest}
			for tree={grow'=90, l=2pt, l sep=0pt, s sep=1pt, parent anchor=center, child anchor=center, edge={color=black, thick}, if n children=0{tier=top}{}},
			delay={
				for tree={
					tier/.process={_O= ? p}{}{content}{O{tier}}{O{content}},
					name'/.process={_O= ? p}{}{content}{_{}}{O{content}},
					content={},
				}
			},
			[#2]
			\coordinate[xshift=-2pt] (bbl) at (current bounding box.north west);
			\coordinate[xshift=5pt] (bbr) at (current bounding box.south east);
			\foreach \i in {#4} {
				\ifnodedefined{\i}{
					\node (l) at (bbl|-\i.center) {};
					\node (r) at (bbr|-\i.center) {};
					\draw[red] (l) -- (r);
				}{}
			}
		\end{forest}
		\hspace{-.73cm}
		\begin{forest}
			for tree={l=2pt, l sep=0pt, s sep=1pt, parent anchor=center, child anchor=center, edge={color=black, thick}, if n children=0{tier=bottom}{}}, 
			delay={
				for tree={
					tier/.process={_O= ? p}{}{content}{O{tier}}{O{content}},
					name'/.process={_O= ? p}{}{content}{_{}}{O{content}},
					content={},
				}
			},
			[#3]
			\coordinate[xshift=-5pt] (bbl) at (current bounding box.north west);
			\coordinate[xshift=2pt] (bbr) at (current bounding box.south east);
			\foreach \i in {#4} {
				\ifnodedefined{\i}{
					\node (l) at (bbl|-\i.center) {};
					\node (r) at (bbr|-\i.center) {};
					\draw[red] (l) -- (r);
				}{}
			}
		\end{forest}
		}
	}
}

\newcommand{\PTGF}{\mathcal{PT}} 
\newcommand{\BTGF}{\mathcal{BT}} 
\newcommand{\CTGF}{\mathcal{BT}} 
\newcommand{\CGF}{\mathcal{C}} 
\newcommand{\tCGF}{\mathcal{C}_\ast} 
\newcommand{\SGF}{\mathcal{S}} 
\newcommand{\tSGF}{\mathcal{S}_\ast} 

\title[Shuffles of deformed permutahedra]{Shuffles of deformed permutahedra, \\ multiplihedra, constrainahedra, and biassociahedra}

\thanks{Partially supported by the French ANR grants CAPPS~17\,CE40\,0018 and CHARMS~19\,CE40\,0017, by the French--Austrian project PAGCAP (ANR-21-CE48-0020 \& FWF I 5788), by the Spanish grant PID2022-137283NB-C21 of MCIN/AEI/10.13039/501100011033 / FEDER, UE, and by Departament de Recerca i Universitats de la Generalitat de Catalunya (2021 SGR 00697).}

\author{Fr\'ed\'eric Chapoton}
\address[FC]{CNRS \& IRMA, Universit\'e de Strasbourg}
\email{chapoton@math.unistra.fr}
\urladdr{\url{https://irma.math.unistra.fr/~chapoton/}}

\author{Vincent Pilaud}
\address{Universitat de Barcelona}
\email{vincent.pilaud@ub.edu}
\urladdr{\url{https://www.ub.edu/comb/vincentpilaud/}}

\begin{document}

\begin{abstract}
We introduce the shuffle of deformed permutahedra (\aka generalized permutahedra), a simple associative operation obtained as the Cartesian product followed by the Minkowski sum with the graphical zonotope of a complete bipartite graph.
Besides preserving the class of graphical zonotopes (the shuffle of two graphical zonotopes is the graphical zonotope of the join of the graphs), this operation is particularly relevant when applied to the classical permutahedra and associahedra.
First, the shuffle of an $m$-permutahedron with an $n$-associahedron gives the $(m,n)$-multiplihedron, whose face structure is encoded by $m$-painted $n$-trees, generalizing the classical multiplihedron.
We show in particular that the graph of the $(m,n)$-multiplihedron is the Hasse diagram of a lattice generalizing the weak order on permutations and the Tamari lattice on binary trees.
Second, the shuffle of an $m$-associahedron with an $n$-associahedron gives the $(m,n)$-constrainahedron, whose face structure is encoded by $(m,n)$-cotrees, and reflects collisions of particles constrained on a grid.
Third, the shuffle of an $m$-anti-associahedron with an $n$-associahedron gives the $(m,n)$-biassociahedron, whose face structure is encoded by \mbox{$(m,n)$-bitrees}, with relevant connections to bialgebras up to homotopy.
We provide explicit vertex, facet, and Minkowski sum descriptions of these polytopes, as well as summation formulas for their $f$-polynomials based on generating functionology of decorated trees.

\medskip
\noindent
\textsc{msc classes.} 52B11, 52B12, 05A15, 05E99, 06B99
\end{abstract}

\vspace*{-1.7cm} 

\maketitle

\vspace*{-.8cm} 

\begin{figure}[h]
	\centerline{\input{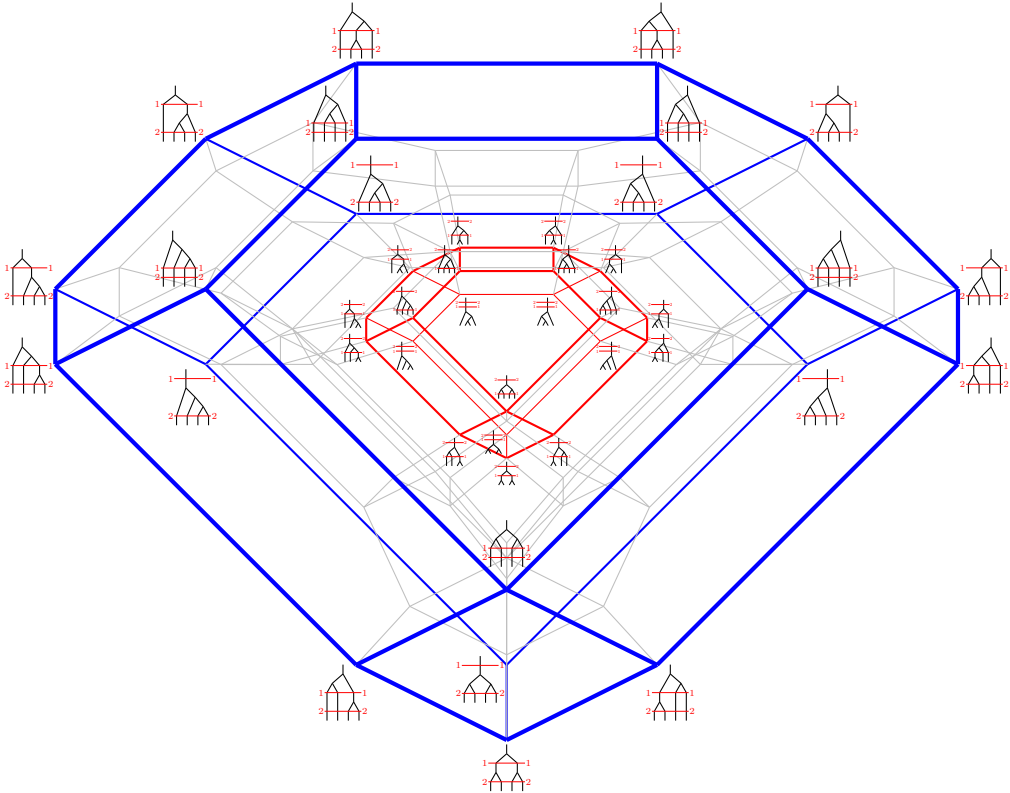}}
	\caption{The $(2,3)$-multiplihedron.}
	\label{fig:multiplihedron23}
\end{figure}

\newpage
\tableofcontents
\vspace{-.5cm}


\section*{Introduction}
\label{sec:intro}

The \defn{associahedra}, now very classical objects, have their origin in algebraic topology \cite{Stasheff}, where they are used to define \defn{associative spaces up-to-homotopy} and \defn{associative algebras up-to-homotopy}.
They were first defined as topological cell complexes and later realized as convex polytopes, as explained in \cite{CeballosZiegler}. For any integer~$n$, the associahedron~$\Asso$ is a polytope of dimension~$n-1$ whose vertices are labeled by binary trees with $n$ nodes.
The axioms of $A_\infty$ algebras encode the fact that each facet of an associahedron can be identified with a product of two smaller associahedra.

One important and natural question is to search for a similar clean description of the axioms for \defn{bialgebras up-to-homotopy}.
This has been studied by several people~\cite{Markl, SaneblidzeUmble-matrads, MerkulovWillwacher} who have found that one meets new difficulties.
The expected picture is the existence, for any pair of integers~$(m,n)$, of a \defn{$(m,n)$-biassociahedron} $\Biassociahedron$, a polytope of dimension~$m+n-1$ whose vertices are labeled by \defn{binary $(m,n)$-bitrees} (which are pairs of binary trees, growing in opposite directions, with $n$ and $m$ nodes respectively, and that are somehow shuffled).
These polytopes are called \defn{step-one biassociahedra} in~\cite{Markl}.
The new difficulty is that the facets of these polytopes are no longer products of two smaller biassociahedra, but rather fiber products with respect to natural projection maps to associahedra.
This implies that, in order to associate an algebraic homotopy to each facet of a biassociahedra, one needs to decompose each problematic facet into several cells.
In the smallest concrete case, the $(2,1)$-biassociahedron $\polytope{B}_{2,1}$ is an hexagon, but one of its edges appears as the diagonal of a square, and must be replaced by half the boundary of this square.
For more details on all this, the reader may consult the given references.
In this article, we give the first complete description of all biassociahedra as convex polytopes, with detailed vertex and facet descriptions.
As far as we know, these objects were previously only known as topological cell complexes, except in small dimensions.
As an historical and futile remark, the first author had the idea of the corresponding fans more than twenty years ago, and asked at least twice the second author whether these fans could be normal fans of convex polytopes.
While we do not consider here the question of finding good axioms for the bialgebras up-to-homotopy, we hope that our simple setting could be helpful to make progress on this subtle question, whose current status is not really satisfactory.

The \defn{constrainahedra} are another family of polytopes closely related to the associahedra and arising in algebraic topology.
They should describe the up-to-homotopy version of double semigroups, namely structures endowed with two associative products (horizontal $\bullet$ and vertical $\circ$) that satisfy the compatibility axiom $(a \bullet b) \circ (c \bullet d) = (a \circ b) \bullet (c \circ d)$.
To our knowledge, this has not appeared in the literature, possibly because it would involve a variant of operads with two-dimensional inputs. 
Instead, constrainahedra were introduced as constrained versions of the $2$-associahedra of~\cite{Bottman}.
A first sketch-definition of the constrainahedra appeared in~\cite{Tierney}, which was based on a private communication from N.~Bottman.
The complete rigorous definition appeared in~\cite{Poliakova, BottmanPoliakova}, where the constrainahedra are also realized as convex polytopes.
For any pair of integers~$(m,n)$, the \defn{$(m,n)$-constrainahedron} $\Constrainahedron$ is a polytope of dimension~$m+n-1$ whose vertices are labeled by good rectangular orders on~$[m+n]$ or equivalently by maximal rectangular bracketings in the $(m \times n)$-grid (see~\cite{Poliakova, BottmanPoliakova} for details).
Here, we prefer to interpret the vertices as \defn{binary $(m,n)$-cotrees} (which are pairs of binary trees, growing in the same direction, with $n$ and $m$ nodes respectively, and that are somehow shuffled).
We provide alternative polytopal realizations of these polytopes, with detailed vertex, facet, and Minkowski sum descriptions.
As a side note, let us mention that arbitrary $2$-associahedra do not fit in the framework of this paper.

The \defn{multiplihedra} are yet other close relatives of the associahedra, with a very similar story in algebraic topology~\cite{Stasheff-HSpaces}.
Their original source is the study of maps between~$A_\infty$ algebras, but they appear under various other guises~\cite{SaneblidzeUmble-diagonals, ForceyLauveSottile, MauWoodward}.
For instance, because the coproduct in Hopf algebras is a morphism of algebras, the multiplihedra belong to the family of biassociahedra.
The interested reader can find a detailed historical exposition in the introduction of~\cite{Forcey-multiplihedra}.
As the associahedra, the multiplihedra were originally built as topological cell complexes (or polytopes with subdivided faces), until a polytopal realization was provided in \cite{Forcey-multiplihedra, ArdilaDoker}.
For any integer~$n$, the multiplihedron~$\polytope{M}\mathrm{ul}(n)$ is a polytope of dimension~$n-1$ whose vertices are labeled by painted binary trees with~$n$ nodes.
In this paper, we include the multiplihedra in a larger family of polytopes.
Namely, for any pair of integers~$(m,n)$, we construct a \defn{$(m,n)$-multiplihedron}~$\Multiplihedron$, a polytope of dimension~$m+n-1$ whose vertices are labeled by \defn{$m$-painted binary $n$-trees} (which are binary trees with $n$ nodes painted with $m$ colors).
Again, we give detailed vertex, facet, and Minkowski sum descriptions of these polytopes.
The classical multiplihedra are obtained when~$m = 1$, but our polytopes provide a different realization from that of~\cite{Forcey-multiplihedra, ArdilaDoker}.
For general $m$, there is no clear interpretation of the $(m,n)$-multiplihedra in terms of algebraic topology.

Our main result is that these three constructions are actually instances of a common natural \defn{shuffle operation} on the family of \defn{deformed permutahedra}.
These polytopes are those obtained from the classical permutahedron by moving facets parallely without passing a vertex, or equivalently those whose normal fans coarsen the braid fan.
They were studied under the name \defn{polymatroids} by J.~Edmonds~\cite{Edmonds} and rediscovered under the name \defn{generalized permutahedra} by A.~Postnikov~\cite{Postnikov}.
Relevant examples of deformed permutahedra include the permutahedra themselves, the graphical zonotopes, the matroid polytopes, the associahedra of~\cite{ShniderSternberg, Loday, HohlwegLange}, the permutreehedra of~\cite{PilaudPons-permutrees}, the quotientopes of~\cite{PilaudSantos-quotientopes, PadrolPilaudRitter}, etc.
This paper focuses on the following simple operation on deformed permutahedra.

\begin{definition}
The \defn{shuffle} of two deformed permutahedra~$\polytope{P} \subseteq \R^m$ and~$\polytope{Q} \subseteq \R^n$ is the Minkowski sum of the Cartesian product~$\polytope{P} \times \polytope{Q}$ with the segments~$[\b{e}_i, \b{e}_{m+j}]$ for all~$i \in [m]$ and~$j \in [n]$.
\end{definition}

This shuffle operation preserves deformed permutahedra (since the Cartesian product and the Minkowski sum do).
It also preserves the family of graphical zonotopes: the shuffle of two graphical zonotopes is the graphical zonotope of the join of the graphs.
But more importantly, it turns out that shuffles of permutahedra and associahedra provide polytopal realizations of the above-mentioned algebraic structures, whose combinatorics is described in details in \mbox{\cref{sec:multiplihedra,sec:constrainahedra,sec:biassociahedra}}.

\begin{proposition}
Let~$m$ and~$n$ be two positive integers.
\begin{enumerate}
\item The shuffle of an $m$-permutahedron by an $n$-associahedron is an \defn{$(m,n)$-multiplihedron}, whose faces are encoded by \defn{$m$-painted $n$-trees},
\item The shuffle of an $m$-associahedron by an $n$-associahedron is an \defn{$(m,n)$-constrainahedron}, whose faces are encoded by \defn{$(m,n)$-cotrees},
\item The shuffle of an $m$-anti-associahedron by an $n$-associahedron is an \defn{$(m,n)$-biassociahedron}, whose faces are encoded by \defn{$(m,n)$-bitrees}.
\end{enumerate}
\end{proposition}

This enables us to give precise integer vertex and facet descriptions of polytopal realizations of the $(m,n)$-multiplihedron, the $(m,n)$-constrainahedron, and the $(m,n)$-biassociahedron.
Along the way, we also provide summation formulas for their $f$-polynomials based on generating functionology of decorated~trees.
As a side note, observe that the shuffle of an $m$-permutahedron with a graph associahedron also generalizes the graph multiplihedra of~\cite{DevadossForcey}.

Finally, we study the behavior of the shuffle operation with respect to lattice properties of the deformed permutahedra.
Our motivation is the classical fact that, when oriented in the direction~$\b{\omega} \eqdef (n,\dots,1) - (1,\dots,n)$, the graph of the permutahedron is the Hasse diagram of the weak order on permutations, and the graph of the associahedron is the Hasse diagram of the Tamari lattice on binary trees.
In view of these examples, we say that a deformed permutahedron has the \defn{lattice property} when its graph oriented in the direction~$\b{\omega}$ is the Hasse diagram of a lattice.
Unfortunately, the shuffle operation does not preserve the lattice property: for instance, the \mbox{$(3,3)$-constrainahedron} and the $(3,3)$-biassociahedron do not have the lattice property.
However, the shuffle with a permutahedron preserves the lattice property.

\begin{proposition}
If a deformed permutahedron~$\polytope{P}$ has the lattice property, then the shuffle~${\polytope{P} \shuffleDP \Perm}$ has the lattice property for any integer~$n \ge 1$.
In particular, the graph of the $(m,n)$-multiplihedron oriented by~$\b{\omega}$ defines a lattice structure on the $m$-painted $n$-trees.
\end{proposition}

In fact, it is well-known that the Tamari lattice is the quotient of the weak order by the sylvester congruence (where two permutations are equivalent when the corresponding cones of the braid fan belong to the same cone of the normal fan of the associahedron).
This implies in particular that the classes of the sylvester congruence are intervals of the weak order.
We say that a deformed permutahedron has the \defn{congruence property} (resp.~the  \defn{interval property}) when the corresponding equivalence relation on permutations is a lattice congruence of the weak order (resp.~admits only intervals as equivalence classes).
We observe that the shuffle operation preserves the interval property but not the congruence property.

\medskip
The paper is organized as follows.
In \cref{sec:preliminaries}, we recall classical definitions and properties concerning permutahedra, associahedra, graphical zonotopes, deformed permutahedra and lattice congruences.
In \cref{sec:shuffles}, we define the shuffle of two deformed permutahedra, provide a combinatorial description of its faces, and discuss the shuffle with a point and the shuffle of graphical zonotopes.
Finally, using shuffles of permutahedra and associahedra, we construct the $(m,n)$-multiplihedron in \cref{sec:multiplihedra}, the $(m,n)$-constrainahedron in \cref{sec:constrainahedra}, and the $(m,n)$-associahedron in \cref{sec:biassociahedra}, provide their vertex and facet descriptions, describe their face lattices, and compute their $f$-polynomials.


\section*{Acknowledgements}
\label{sec:acknowledgments}

We are deeply grateful to Spencer Backman, Nathaniel Bottman and Daria Poliakova for pointing us to the constrainahedra~\cite{Poliakova, BottmanPoliakova} and asking for their connection to the biassociahedra of~\cite{Markl}.
As explained in \cref{sec:constrainahedra,sec:biassociahedra}, the former are shuffles of associahedra with associahedra while the latter are shuffles of anti-associahedra with associahedra.
This motivated us to include both in the present version (while the former was previously omitted due to our lack of \mbox{algebraic motivation}).
We also thank two anonymous referees for corrections.


\section{Preliminaries}
\label{sec:preliminaries}

This section recalls classical definitions and properties concerning polyhedral geometry (\cref{subsec:fansPolytopes}), permutahedra (\cref{subsec:permutahedra}), associahedra (\cref{subsec:associahedra}), graphical zonotopes (\cref{subsec:graphicalZonotopes}), deformed permutahedra (\cref{subsec:deformedPermutahedra}), and lattice congruences (\cref{subsec:latticeProperties}).
The reader familiar with these notions is invited to jump directly to \cref{sec:shuffles} and to refer to this section only for conventions and notations.
We omit the proofs of all results of this section as they are either well-known or immediate.
In fact, we try to attribute properly the results of this section, but consider some of them as folklore, and do not claim anything new in this section.


\subsection{Fans and polytopes}
\label{subsec:fansPolytopes}

We refer to~\cite{Ziegler-polytopes} for a standard reference on polyhedral geometry.
We denote by~$(\b{e}_i)_{i \in [n]}$ the standard basis of~$\R^n$.

\begin{definition}
\label{def:cone}
A (polyhedral) \defn{cone} is defined equivalently as 
\begin{itemize}
\item the cone~$\R_{\ge0} \b{R} \eqdef \set{\sum_{\b{r} \in \b{R}}\lambda_{\b{r}} \b{r}}{\lambda_{\b{r}} \ge 0 \text{ for all } \b{r} \in \b{R}}$ generated by a finite set~$\b{R} \subset \R^n$,
\item the cone~$\set{\b{x} \in \R^n}{\dotprod{\b{n}}{\b{x}} \ge 0 \text{ for all } \b{n} \in \b{N}}$ defined by a finite set~$\b{N} \subset \R^n$.
\end{itemize}
A \defn{face} of a cone~$\polytope{C}$ is the intersection of~$\polytope{C}$ with a supporting hyperplane of~$\polytope{C}$.
In this paper, we also consider $\polytope{C}$ itself as a face, but ignore the empty face.
\end{definition}

\begin{definition}
\label{def:fan}
A (polyhedral) \defn{fan} is a collection~$\c{F}$ of cones of~$\R^n$ such that
\begin{itemize}
\item any face of a cone in~$\c{F}$ is also in~$\c{F}$,
\item the intersection of any two cones of~$\c{F}$ is a face of both.
\end{itemize}
The \defn{rays} (resp.~\defn{walls}, resp.~\defn{chambers}) of~$\c{F}$ are its $1$-dimensional (resp.~codimension~$1$, resp.~full-dimensional) cones.
\end{definition}

\begin{definition}
\label{def:polytope}
A \defn{polytope} is defined equivalently as
\begin{itemize}
\item the convex hull~$\set{\sum_{\b{v} \in \b{V}} \lambda_{\b{v}} \b{v}}{\lambda_{\b{v}} \ge 0 \text{ for all } \b{v} \in \b{V} \text{ and } \sum_{\b{v} \in \b{V}} \lambda_{\b{v}} = 1}$ of a finite set~${\b{V} \in \R^n}$,
\item a bounded intersection of a finite number of affine half-spaces of~$\R^n$.
\end{itemize}
A \defn{face} of a polytope~$\polytope{P}$ is the intersection of~$\polytope{P}$ with a supporting hyperplane of~$\polytope{P}$.
The \defn{vertices} (resp.~\defn{edges}, resp.~\defn{facets}) are the $0$-dimensional (resp.~$1$-dimensional, resp.~codimension~$1$) faces.
In this paper, we also consider $\polytope{P}$ itself as a face, but ignore the empty face.
\end{definition}

Any polytope defines a fan as follows (in contrast, not all fans come from polytopes).

\begin{definition}
\label{def:normalFan}
Let~$\polytope{P}$ be a polytope and~$\polytope{F}$ be a face of~$\polytope{P}$.
The \defn{normal cone} of~$\polytope{F}$ is the cone $\c{N}(\polytope{F}) \eqdef \set{\b{v} \in \R^n}{\dotprod{\b{v}}{\b{f}} \ge \dotprod{\b{v}}{\b{p}} \text{ for all } \b{f} \in \polytope{F} \text{ and } \b{p} \in \polytope{P}}$ of linear functions maximized over~$\polytope{P}$ by all the face~$\polytope{F}$.
The \defn{normal fan} of~$\polytope{P}$ is the fan~$\c{N}(\polytope{P}) \eqdef \set{\c{N}(\polytope{F})}{\polytope{F} \text{ face of } \polytope{P}}$ containing the normal cones of all faces of~$\polytope{P}$.
\end{definition}

In this paper, we will use the following standard operations on fans and polytopes.

\begin{definition}
\label{def:directSumCommonRefinement}
Let~$\c{F} \subset \R^m$ and~$\c{G} \subset \R^n$ be two fans. Then
\begin{itemize}
\item the \defn{direct sum} of~$\c{F}$ and~$\c{G}$ is the fan~$\c{F} \oplus \c{G} \eqdef \set{C \times D}{C \in \c{F} \text{ and } D \in \c{G}}$,
\item if~$m = n$, the \defn{common refinement} of~$\c{F}$ and~$\c{G}$ is the fan~$\c{F} \wedge \c{G} \eqdef \set{C \cap D}{C \in \c{F} \text{ and } D \in \c{G}}$.
\end{itemize}
\end{definition}

\begin{definition}
\label{def:CartesianProductMinkowskiSum}
Let~$\polytope{P} \subset \R^m$ and~$\polytope{Q} \subset \R^n$ be two polytopes. Then
\begin{itemize}
\item the \defn{Cartesian product} of~$\polytope{P}$ and~$\polytope{Q}$ is the polytope~$\polytope{P} \times \polytope{Q} \eqdef \set{(\b{p},\b{q})}{\b{p} \in \polytope{P} \text{ and } \b{q} \in \polytope{Q}}$,
\item if~$m = n$, the \defn{Minkowski sum} of~$\polytope{P}$ and~$\polytope{Q}$ is the polytope~$\polytope{P} + \polytope{Q} \eqdef \set{\b{p}+\b{q}}{\b{p} \in \polytope{P} \text{ and } \b{q} \in \polytope{Q}}$.
\end{itemize}
\end{definition}

The following connection between \cref{def:directSumCommonRefinement,def:CartesianProductMinkowskiSum} is classical, see \cite[Lems.~7.7 \& 7.12]{Ziegler-polytopes}.

\pagebreak

\begin{proposition}
\label{prop:CartesianProductMinkowskiSum}
Let~$\polytope{P} \subset \R^m$ and~$\polytope{Q} \subset \R^n$ be two polytopes. Then
\begin{itemize}
\item the normal fan of the Cartesian product~$\polytope{P} \times \polytope{Q}$ is the direct sum of the normal fans of~$\polytope{P}$~and~$\polytope{Q}$, that is~$\c{N}(\polytope{P} \times \polytope{Q}) = \c{N}(\polytope{P}) \oplus \c{N}(\polytope{Q})$,
\item if~$m = n$, the normal fan of the Minkowski sum~$\polytope{P} + \polytope{Q}$ is the common refinement of the normal fans of~$\polytope{P}$ and~$\polytope{Q}$, that is~$\c{N}(\polytope{P} + \polytope{Q}) = \c{N}(\polytope{P}) \wedge \c{N}(\polytope{Q})$.
\end{itemize}
\end{proposition}


\subsection{Permutahedra}
\label{subsec:permutahedra}

Let~$\f{S}_n$ denote the symmetric group of permutations of~$[n] \eqdef \{1, \dots, n\}$.

\begin{figure}[b]
	\centerline{\includegraphics[scale=.6]{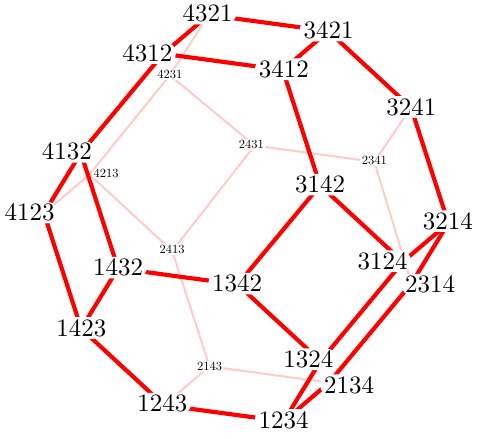} \quad \raisebox{-.1cm}{\includegraphics[scale=.6]{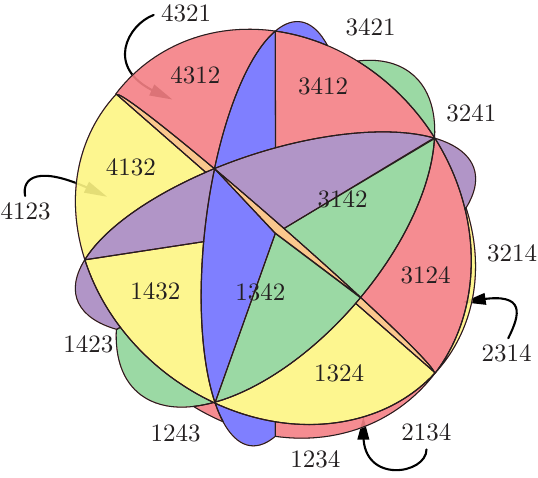}} \quad \includegraphics[scale=.55]{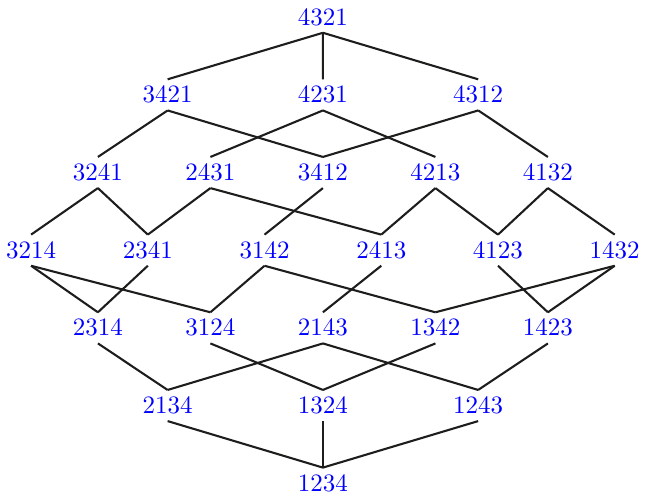}}
	\caption{The permutahedron, the braid fan, and the weak order.}
	\label{fig:permutahedron}
\end{figure}

\begin{definition}
\label{def:permutahedron}
The \defn{permutahedron}~$\Perm$ is the polytope in~$\R^n$ equivalently defined as:
\begin{itemize}
\item the convex hull of the points~$\sum_{i \in [n]} i \, \b{e}_{\sigma(i)}$ for all permutations~$\sigma \in \f{S}_n$,
\item or the intersection of the hyperplane~$\bigset{\b{x} \in \R^n}{\sum_{i \in [n]} x_i = \binom{n+1}{2}}$ with the affine half-spaces~$\bigset{\b{x} \in \R^n}{\sum_{i \in I} x_i \ge \binom{|I|+1}{2}}$ for all~${\varnothing \ne I \subsetneq [n]}$,
\item or (a translate of) the Minkowski sum of all segments~$[\b{e}_i, \b{e}_j]$ for all~$1 \le i < j \le n$.
\end{itemize}
See \cref{fig:permutahedron}\,(left).
\end{definition}

The permutahedron~$\Perm$ has dimension~$n-1$ but it will be convenient to consider it embedded in~$\R^n$.
Note that the point corresponding to a permutation~$\sigma$ is the point of coordinates $(\sigma^{-1}(1), \dots, \sigma^{-1}(n))$.
The face structure of the permutahedron is encoded by ordered~partitions.

\begin{definition}
\label{def:orderedPartitions}
An \defn{ordered partition} of~$[n]$ is a partition~$\mu \eqdef \mu_1 | \dots | \mu_p$ of~$[n]$ into non-empty parts, with a total order on the parts (but each part is unordered).
It defines a preposet (\ie a reflexive and transitive binary relation) $\preccurlyeq_\mu$ on~$[n]$ where~$i \preccurlyeq_\mu j$ if the part of~$\mu$ containing~$i$ is before or equal to the part of~$\mu$ containing~$j$.
The resulting preposets are all total preposets, that is, where any two elements of~$[m]$ are comparable (\ie~$i \preccurlyeq j$ or $i \succcurlyeq j$ or both).
For two ordered partitions~$\mu$ and~$\nu$, we say that~$\mu$ \defn{refines}~$\nu$ (and~$\nu$ \defn{coarsens} $\mu$) when~$i \preccurlyeq_\mu j$ implies $i \preccurlyeq_\nu j$ for any~$i,j \in [n]$.
We denote by~$\f{P}_n$ the set of ordered partitions of~$[n]$.
\end{definition}

\begin{proposition}
\label{prop:permutahedron}
The face lattice of the permutahedron~$\Perm$ is isomorphic to the refinement poset on $\f{P}_n$ (augmented with a minimal element).
\end{proposition}

\begin{proposition}
\label{prop:braidFan}
The normal fan of the permutahedron~$\Perm$ is the \defn{braid fan} with one cone $\polytope{C}(\mu) \eqdef \set{\b{x} \in \R^n}{x_i \le x_j \text{ if } i \preccurlyeq_\mu j}$ for each~$\mu \in \f{P}_n$.
Its walls are given by the arrangement of the hyperplanes~$\set{\b{x} \in \R^n}{x_i = x_j}$ for all~$1 \le i < j \le n$.
See \cref{fig:permutahedron}\,(middle).
\end{proposition}

\enlargethispage{.3cm}
We now recall the connection between the graph of the permutahedron~$\Perm$ and the weak order on permutations of~$\f{S}_n$.

\begin{definition}
\label{def:weakOrder}
An \defn{inversion} of a permutation~$\sigma$ is a pair~$(\sigma_i, \sigma_j)$ such that~${i < j}$ but~${\sigma_i > \sigma_j}$.
The \defn{weak order} is the lattice on permutations of~$[n]$ defined by inclusion of their inversion sets.
See \cref{fig:permutahedron}\,(right).
\end{definition}

\begin{proposition}
\label{prop:weakOrder}
When oriented in the direction~$\b{\omega} \eqdef (n,\dots,1) - (1,\dots,n) = \sum_{i \in [n]} (n+1-2i) \, \b{e}_i$, the graph of the permutahedron~$\Perm$ is the Hasse diagram of the weak order on~$\f{S}_n$.
\end{proposition}


\pagebreak

\subsection{Associahedra}
\label{subsec:associahedra}

We now recall the classical associahedra, whose vertices (resp.~faces) correspond to binary trees (resp.~Schr\"oder trees).
We start with some formal definitions on rooted plane trees used later in \cref{subsec:paintedTrees,subsec:cotrees,subsec:bitrees}.
\cref{def:tree,def:inorder,def:treeDeletion,def:binaryTreeSchroderTree} are illustrated in \cref{fig:trees}.

\begin{definition}
\label{def:tree}
A (rooted plane) \defn{tree} is either a \defn{leaf}~$|$ or a \defn{node}~$\node{n}$ with an ordered non-empty list~$\children(\node{n})$ of (rooted plane) trees.
The node~$\node{n}$ is the \defn{parent} of the nodes in~$\children(\node{n})$, which are the \defn{children} of~$\node{n}$.
The \defn{degree} of~$\node{n}$ is its number of children.
The \defn{root} is the unique node with no parent.
An \defn{$n$-tree} is a tree with $n+1$ leaves.
\end{definition}

\begin{definition}
\label{def:inorder}
A $n$-tree~$T$ is labeled in \defn{inorder} when each degree~$\ell$ node is labeled by an $(\ell-1)$-subset~$\{x_1, \dots, x_{\ell-1}\}$ of~$[n]$ such that all labels in its $i$-th subtree are larger than~$x_{i-1}$ and smaller than~$x_i$ (where by convention~$x_0 = 0$ and~$x_\ell = n+1$).
It defines a preposet~$\preccurlyeq_T$ on~$[n]$ where~$i \preccurlyeq_T j$ if there is a (possibly empty) path from the node containing~$i$ to the node containing~$j$ in the tree~$T$ oriented towards the root.
The resulting preposets are all preposets~$\preccurlyeq$ such that any~$1 \le i < k \le n$ are comparable (\ie~$i \preccurlyeq k$ or $i \succcurlyeq k$ or both) if and only if there is no~$i < j < k$ such that~$i \prec j \succ k$.
\end{definition}

\begin{definition}
\label{def:treeDeletion}
The \defn{deletion} of a node~$\node{n}$ with parent~$\node{p}$ consists in replacing~$\node{n}$ by the list~$\children(\node{n})$ in the list~$\children(\node{p})$.
Intuitively, this operation contracts the edge from~$\node{n}$ to~$\node{p}$ in the tree.
\begin{figure}[b]
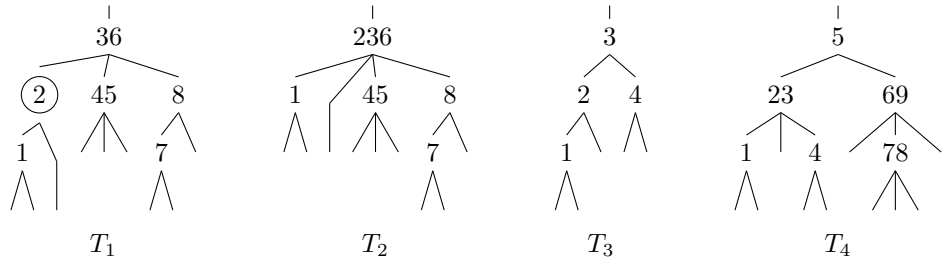

	\centerline{
	\begin{tabular}{c@{\qquad}c@{\qquad}c@{\qquad}c}
		\tree{[ [.36 [.\circled{2} [.1 {} {} ] [ {} ] ] [.45 {} {} {} ] [.8 [.7 {} {} ] {} ] ] ]}
		&
		\tree{[ [.236 [.1 {} {} ] [ {} ] [.45 {} {} {} ] [.8 [.7 {} {} ] {} ] ] ]}
		&
		\tree{[ [.3 [.2 [.1 {} {} ] {} ] [.4 {} {} ] ] ]}
		&
		\tree{[ [.5 [.23 [.1 {} {} ] {} [.4 {} {} ] ] [.69 {} [.78 {} {} {} ] {} ] ] ]}
		\\[-.1cm]
		$T_1$
		&
		$T_2$
		&
		$T_3$
		&
		$T_4$
	\end{tabular}
	}
	\caption{A (plane rooted) tree~$T_1$ with a circled node, the tree~$T_2$ obtained by deletion of the circled node in~$T_1$, a binary tree~$T_3$, and a Schr\"oder tree~$T_4$. All trees are labeled in inorder (at each node, we simply write the word~$x_1 \dots x_{\ell-1}$ for the set~$\{x_1, \dots, x_{\ell-1}\}$).}
	\label{fig:trees}
\end{figure}
\end{definition}

\begin{definition}
\label{def:binaryTreeSchroderTree}
A \defn{binary} (resp.~\defn{Schr\"oder}) \defn{tree} is a rooted plane tree whose internal nodes have degree exactly (resp.~at least)~$2$.
We denote by~$\f{B}_n$ (resp.~$\f{T}_n$) the set of \mbox{binary (resp.~Schr\"oder)~$n$-trees.}
\end{definition}

\begin{proposition}
\label{prop:SchroderTreeDeletion}
For any integer~$n \ge 0$, the set~$\f{T}_n$ is stable by deletion, and the deletion graph is the Hasse diagram of a poset ranked by ${\rank(T) = \sum_{\node{n} \in T} \big(\!\deg(\node{n})-2 \big)} = n - |T|$.
In this poset, $S$ is smaller than~$T$ if and only if~$\preccurlyeq_S$ refines~$\preccurlyeq_T$.
The binary trees are the minimal elements, and the corolla is the unique maximal element of this poset.
\end{proposition}

\begin{definition}
\label{def:SchroderTreeDeletionPoset}
The \defn{$n$-Schr\"oder tree deletion poset} is the poset on $\f{T}_n$ where a Schr\"oder tree is covered by all Schr\"oder trees that can be obtained by a deletion.
\end{definition}

We now recall a classical geometric realization of this poset, tracing back to~\cite{ShniderSternberg, Loday, Postnikov}.
Generalizations of this construction were explored in~\cite{HohlwegLange, HohlwegLangeThomas, HohlwegPilaudStella, PilaudPons-permutrees, PadrolPaluPilaudPlamondon, Pilaud-acyclicReorientationLattices} among others. 
See \cite{PilaudSantosZiegler} for a recent survey.

\begin{figure}
	\centerline{\includegraphics[scale=.6]{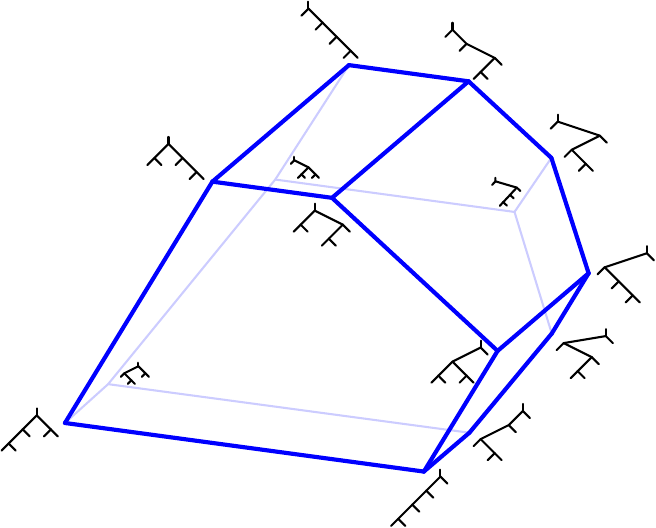} \quad \raisebox{-.1cm}{\includegraphics[scale=.6]{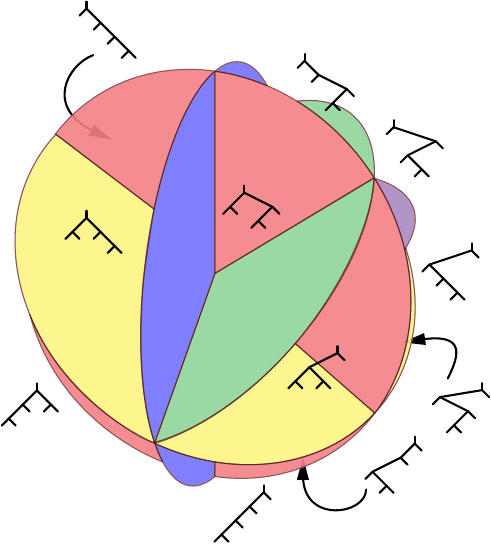}} \quad \includegraphics[scale=.5]{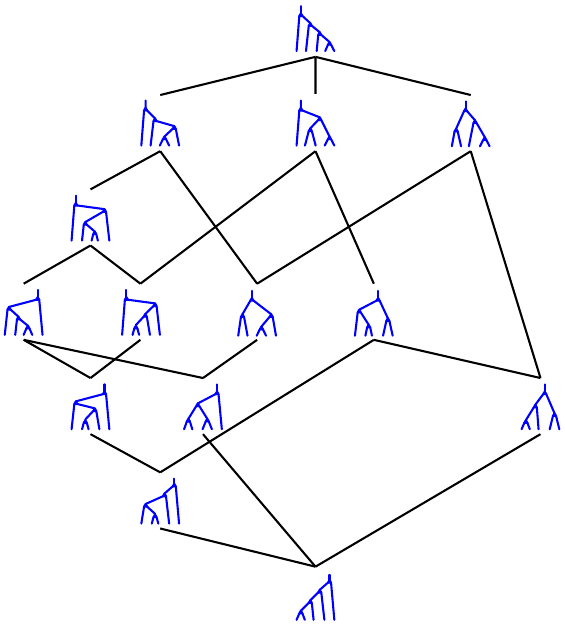}}
	\caption{The associahedron, the sylvester fan, and the Tamari lattice.}
	\label{fig:associahedron}
\end{figure}

\begin{definition}
\label{def:associahedron}
The \defn{associahedron}~$\Asso$ is the polytope in~$\R^n$ equivalently defined as:
\begin{itemize}
\item the convex hull of the points~$\sum_{i \in [n]} \ell(T,i) \, r(T,i) \, \b{e}_i$ for all binary trees~$T \in \f{B}_n$, where $\ell(T,i)$ and~$r(T,i)$ respectively denote the numbers of leaves in the left and right subtrees of the $i$-th node of~$T$ in inorder (see~\cite{Loday}),
\item or the intersection of the hyperplane~$\bigset{\b{x} \in \R^n}{\sum_{i \in [n]} x_i = \binom{n+1}{2}}$ with the affine half-spaces~$\bigset{\b{x} \in \R^n}{\sum_{i \le \ell \le j} x_\ell \ge \binom{j-i+2}{2}}$ for all~$1 \le i \le j \le n$ (see~\cite{ShniderSternberg}),
\item or (a translate of) the Minkowski sum of the faces~$\simplex_{[i,j]}$ of the standard simplex~$\simplex_{[n]}$ for all ${1 \le i \le j \le n}$, where~$\simplex_X \eqdef \conv\set{\b{e}_x}{x \in X}$ for~$X \subseteq [n]$ (see~\cite{Postnikov}).
\end{itemize}
See \cref{fig:associahedron}\,(left).
\end{definition}

The associahedron~$\Asso$ has dimension~$n-1$, although it is convenient to consider it embedded in~$\R^n$.
Note that any facet defining inequality for~$\Asso$ is also a facet defining inequality for~$\Perm$.
In other words, the associahedron~$\Asso$ is a \defn{removahedron}: it can be obtained by deleting some inequalities in the facet description of the permutahedron~$\Perm$.

\begin{proposition}[\cite{Loday}]
\label{prop:associahedron}
The face lattice of the associahedron~$\Asso$ is isomorphic to the deletion poset on $\f{T}_n$ (augmented with a minimal element).
\end{proposition}

\begin{proposition}
\label{prop:sylvesterFan}
The normal fan of the associahedron~$\Asso$ is the \defn{sylvester fan} with one cone $\polytope{C}(S) \eqdef \set{\b{x} \in \R^n}{x_i \le x_j \text{ if } i \preccurlyeq_S j}$ for each~$S \in \f{T}_n$.
See \cref{fig:associahedron}\,(middle).
\end{proposition}

Note that here and throughout, sylvester is an ancient adjective for woody, and has nothing to do with the mathematician James Joseph Sylvester.

It turns out that the sylvester fan of \cref{prop:sylvesterFan} coarsens the braid fan of \cref{prop:braidFan}.

\begin{proposition}
\label{prop:fanRefinementSylvester}
The braid fan refines the sylvester fan.
More precisely, for any Schr\"oder tree~$S$, the sylvester cone~$\polytope{C}(S)$ is the union of the braid cones~$\polytope{C}(\mu)$ for the ordered partitions~$\mu$ such that~$\preccurlyeq_\mu$ extends~$\preccurlyeq_S$ (meaning that~$i \preccurlyeq_S j$ implies $i \preccurlyeq_\mu j$ for any~$i,j \in [n]$).
\end{proposition}

\enlargethispage{.3cm}
This can be interpreted as equivalence relations on permutations and on ordered partitions.

\begin{definition}
\label{def:sylvesterRelation}
The \defn{sylvester relation} on ordered partitions of~$[n]$ is the equivalence relation~$\equiv_\sylv$ defined by~$\mu \equiv_\sylv \nu$ if and only if the cones~$\polytope{C}(\mu)$ and~$\polytope{C}(\nu)$ of the braid fan belong precisely to the same cones of the sylvester fan.
It also restricts to an equivalence relation on permutations.
\end{definition}

\begin{remark}
\label{rem:sylvesterRelation}
The sylvester relation on permutations admits several equivalent definitions. Namely, two permutations~$\sigma, \tau \in \f{S}_n$ are equivalent when:
\begin{itemize}
\item $\sigma$ and~$\tau$ are linear extensions of the poset~$\preccurlyeq_T$ for the same binary tree~$T$ of~$\f{B}_n$,
\item $\sigma$ and~$\tau$ are sent to the same binary tree~$T$ via right-to-left binary search tree insertions,
\item the braid cones~$\polytope{C}(\sigma)$ and~$\polytope{C}(\tau)$ of the braid fan belong to the same sylvester cone~$\polytope{C}(T)$,
\item $\sigma$ and~$\tau$ are connected via a sequence of rewritings of the form~$UacVbW \equiv UcaVbW$ where~$1 \le a < b < c \le n$ and~$U,V,W$ are words on~$[n]$.
\end{itemize}
\end{remark}

We now recall the connection between the graph of the associahedron~$\Asso$ and the Tamari lattice on binary trees of~$\f{B}_n$ illustrated in \cref{fig:associahedron}\,(right).

\parpic(4.5cm,1.5cm)(0pt, 55pt)[r][b]{
	\tree{[[[ {A} {B} ] {C} ]]}[5pt][18pt] \raisebox{.5cm}{$\longrightarrow$} \tree{[[ {A} [ {B} {C} ]]]}[5pt][18pt]
}{
\begin{definition}
\label{def:TamariLattice}
A \defn{right rotation} is the operation on binary trees illustrated on the right (this operation can be applied locally anywhere in the tree).
The \defn{Tamari lattice} is the lattice on~$\f{B}_n$ whose Hasse diagram is the graph of right rotations.
See \cref{fig:associahedron}\,(right).
\end{definition}
}

\begin{proposition}
\label{prop:TamariLattice}
When oriented in the direction~$\b{\omega} \eqdef (n,\dots,1) - (1,\dots,n) = \sum_{i \in [n]} (n+1-2i) \, \b{e}_i$, the graph of the associahedron~$\Asso$ is the Hasse diagram of the Tamari lattice on~$\f{B}_n$.
\end{proposition}

\begin{remark}
\label{rem:latticeQuotient}
In fact, the sylvester relation is a lattice congruence of the weak order (meaning that it respects meets and joins), and the Tamari lattice is the quotient of the weak order by the sylvester congruence.
This perspective has been largely explored by N.~Reading in his study of lattice congruences of the weak order~\cite{Reading-latticeCongruences}.
In this paper, we do not consider this property as it is not stable by the shuffle operation we focus on.
See \cref{subsec:latticeProperties}.
\end{remark}

We conclude with some classical numerology on binary and Schr\"oder trees that will be generalized to multiplihedra, constrainahedra, and biassociahedra in \cref{subsec:numerologyMultiplihedra,subsec:numerologyConstrainahedra,subsec:numerologyBiassociahedra}.

\begin{notation}
\label{not:CatalanSchroderGF}
Let~$C(n) = \frac{1}{n+1} \binom{2n}{n}$ denote the \defn{Catalan number} of binary trees with~$n+1$ leaves and let~$S(n,p)$ denote the \defn{Schr\"oder number} of Schr\"oder trees with $n+1$ leaves and $n-p$ internal nodes.
We denote the corresponding generating functions by
\[
\CGF(y) \eqdef \sum_{n \ge 1} C(n) \, y^n
\qquad\text{and}\qquad
\SGF(y,z) \eqdef \sum_{n \ge 1, p \ge 0} S(n,p) \, y^n \, z^p.
\]
\end{notation}

\begin{proposition}
\label{prop:CatalanSchroderGF}
The generating functions~$\CGF(y)$ of binary trees (\ie vertices of associahedra) and~$\SGF(y,z)$ of Schr\"oder trees (\ie faces of associahedra) satisfy
\[
\CGF(y) = y + \CGF(y)^2
\qquad\text{and}\qquad
\SGF(y,z) = y + \frac{\SGF(y,z)^2}{1-z\SGF(y,z)}
\]
and are therefore given by
\[
\CGF(y) = \frac{1-\sqrt{1-4y}}{2}
\qquad\text{and}\qquad
\SGF(y,z) = \frac{1+yz-\sqrt{1-4y-2yz+y^2z^2}}{2(z+1)}.
\]
\end{proposition}


\subsection{Graphical zonotopes}
\label{subsec:graphicalZonotopes}

In this section, we consider a simple (no loop nor multiple edges) non-oriented graph~$G$, with vertex set~$V(G)$ and edge set~$E(G)$.
We say that $G$ is an integer graph when~$V(G) = [n]$, and we then represent the edges of~$G$ by ordered pairs~$(i,j)$ with~$1 \le i < j \le n$.

\begin{definition}
\label{def:graphicalZonotope}
The \defn{graphical zonotope}~$\Zono[G]$ of an integer graph~$G$ is the Minkowski sum of the segments~$[\b{e}_i, \b{e}_j]$ for~${(i,j) \in E(G)}$.
See \cref{fig:graphicalZonotopes}\,(left).
\end{definition}

\begin{figure}
	\centerline{
		\raisebox{.7cm}{

\begin{tikzpicture}%
	[x={(-0.366215cm, -0.789554cm)},
	y={(0.235950cm, -0.590693cm)},
	z={(0.900119cm, -0.166391cm)},
	scale=1.200000,
	back/.style={very thin, opacity=0.5},
	edge/.style={color=red, thick, decoration={markings, mark=at position 0.5 with {}}},
	facet/.style={fill=red,fill opacity=0},
	vertex/.style={}]
%
%

\coordinate (0.00000, 0.00000, 0.00000) at (0.00000, 0.00000, 0.00000);
\coordinate (0.00000, 0.00000, 1.63299) at (0.00000, 0.00000, 1.63299);
\coordinate (0.00000, 1.41421, -0.81650) at (0.00000, 1.41421, -0.81650);
\coordinate (0.00000, 1.41421, 0.81650) at (0.00000, 1.41421, 0.81650);
\coordinate (1.33333, -0.94281, 0.00000) at (1.33333, -0.94281, 0.00000);
\coordinate (1.33333, -0.94281, 1.63299) at (1.33333, -0.94281, 1.63299);
\coordinate (1.33333, 0.47140, -0.81650) at (1.33333, 0.47140, -0.81650);
\coordinate (2.66667, 0.94281, 1.63299) at (2.66667, 0.94281, 1.63299);
\coordinate (1.33333, 0.47140, 2.44949) at (1.33333, 0.47140, 2.44949);
\coordinate (1.33333, 1.88562, 0.00000) at (1.33333, 1.88562, 0.00000);
\coordinate (1.33333, 1.88562, 1.63299) at (1.33333, 1.88562, 1.63299);
\coordinate (2.66667, -0.47140, 0.81650) at (2.66667, -0.47140, 0.81650);
\coordinate (2.66667, -0.47140, 2.44949) at (2.66667, -0.47140, 2.44949);
\coordinate (2.66667, 0.94281, 0.00000) at (2.66667, 0.94281, 0.00000);
\draw[edge,back,postaction={decorate}] (0.00000, 1.41421, -0.81650) -- (0.00000, 0.00000, 0.00000);
\draw[edge,back,postaction={decorate}] (0.00000, 1.41421, 0.81650) -- (0.00000, 0.00000, 1.63299);
\draw[edge,back,postaction={decorate}] (0.00000, 1.41421, 0.81650) -- (0.00000, 1.41421, -0.81650);
\draw[edge,back,postaction={decorate}] (1.33333, 0.47140, -0.81650) -- (0.00000, 1.41421, -0.81650);
\draw[edge,back,postaction={decorate}] (1.33333, 1.88562, 0.00000) -- (0.00000, 1.41421, -0.81650);
\draw[edge,back,postaction={decorate}] (1.33333, 1.88562, 1.63299) -- (0.00000, 1.41421, 0.81650);
\draw[edge,back,postaction={decorate}] (1.33333, 1.88562, 1.63299) -- (1.33333, 1.88562, 0.00000);
\draw[edge,back,postaction={decorate}] (2.66667, 0.94281, 0.00000) -- (1.33333, 1.88562, 0.00000);
\node[vertex] at (0.00000, 1.41421, -0.81650)     {};
\node[vertex] at (1.33333, 1.88562, 0.00000)     {};
\node[vertex] at (0.00000, 1.41421, 0.81650)     {};
\fill[facet] (2.66667, 0.94281, 0.00000) -- (1.33333, 0.47140, -0.81650) -- (1.33333, -0.94281, 0.00000) -- (2.66667, -0.47140, 0.81650) -- cycle {};
\fill[facet] (1.33333, -0.94281, 1.63299) -- (0.00000, 0.00000, 1.63299) -- (0.00000, 0.00000, 0.00000) -- (1.33333, -0.94281, 0.00000) -- cycle {};
\fill[facet] (2.66667, -0.47140, 2.44949) -- (1.33333, -0.94281, 1.63299) -- (0.00000, 0.00000, 1.63299) -- (1.33333, 0.47140, 2.44949) -- cycle {};
\fill[facet] (2.66667, -0.47140, 2.44949) -- (1.33333, -0.94281, 1.63299) -- (1.33333, -0.94281, 0.00000) -- (2.66667, -0.47140, 0.81650) -- cycle {};
\fill[facet] (1.33333, 0.47140, 2.44949) -- (2.66667, -0.47140, 2.44949) -- (2.66667, 0.94281, 1.63299) -- (1.33333, 1.88562, 1.63299) -- cycle {};
\fill[facet] (2.66667, -0.47140, 0.81650) -- (2.66667, 0.94281, 0.00000) -- (2.66667, 0.94281, 1.63299) -- (2.66667, -0.47140, 2.44949) -- cycle {};
\draw[edge,postaction={decorate}] (0.00000, 0.00000, 1.63299) -- (0.00000, 0.00000, 0.00000);
\draw[edge,postaction={decorate}] (1.33333, -0.94281, 0.00000) -- (0.00000, 0.00000, 0.00000);
\draw[edge,postaction={decorate}] (1.33333, -0.94281, 1.63299) -- (0.00000, 0.00000, 1.63299);
\draw[edge,postaction={decorate}] (1.33333, 0.47140, 2.44949) -- (0.00000, 0.00000, 1.63299);
\draw[edge,postaction={decorate}] (1.33333, -0.94281, 1.63299) -- (1.33333, -0.94281, 0.00000);
\draw[edge,postaction={decorate}] (1.33333, 0.47140, -0.81650) -- (1.33333, -0.94281, 0.00000);
\draw[edge,postaction={decorate}] (2.66667, -0.47140, 0.81650) -- (1.33333, -0.94281, 0.00000);
\draw[edge,postaction={decorate}] (2.66667, -0.47140, 2.44949) -- (1.33333, -0.94281, 1.63299);
\draw[edge,postaction={decorate}] (2.66667, 0.94281, 0.00000) -- (1.33333, 0.47140, -0.81650);
\draw[edge,postaction={decorate}] (2.66667, 0.94281, 1.63299) -- (1.33333, 1.88562, 1.63299);
\draw[edge,postaction={decorate}] (2.66667, 0.94281, 1.63299) -- (2.66667, -0.47140, 2.44949);
\draw[edge,postaction={decorate}] (2.66667, 0.94281, 1.63299) -- (2.66667, 0.94281, 0.00000);
\draw[edge,postaction={decorate}] (1.33333, 1.88562, 1.63299) -- (1.33333, 0.47140, 2.44949);
\draw[edge,postaction={decorate}] (2.66667, -0.47140, 2.44949) -- (1.33333, 0.47140, 2.44949);
\draw[edge,postaction={decorate}] (2.66667, -0.47140, 2.44949) -- (2.66667, -0.47140, 0.81650);
\draw[edge,postaction={decorate}] (2.66667, 0.94281, 0.00000) -- (2.66667, -0.47140, 0.81650);
\node[vertex] at (0.00000, 0.00000, 0.00000)     {};
\node[vertex] at (0.00000, 0.00000, 1.63299)     {};
\node[vertex] at (1.33333, -0.94281, 0.00000)     {};
\node[vertex] at (1.33333, -0.94281, 1.63299)     {};
\node[vertex] at (1.33333, 0.47140, -0.81650)     {};
\node[vertex] at (2.66667, 0.94281, 1.63299)     {};
\node[vertex] at (1.33333, 0.47140, 2.44949)     {};
\node[vertex] at (1.33333, 1.88562, 1.63299)     {};
\node[vertex] at (2.66667, -0.47140, 0.81650)     {};
\node[vertex] at (2.66667, -0.47140, 2.44949)     {};
\node[vertex] at (2.66667, 0.94281, 0.00000)     {};
\end{tikzpicture}} \quad
		\includegraphics[scale=.6]{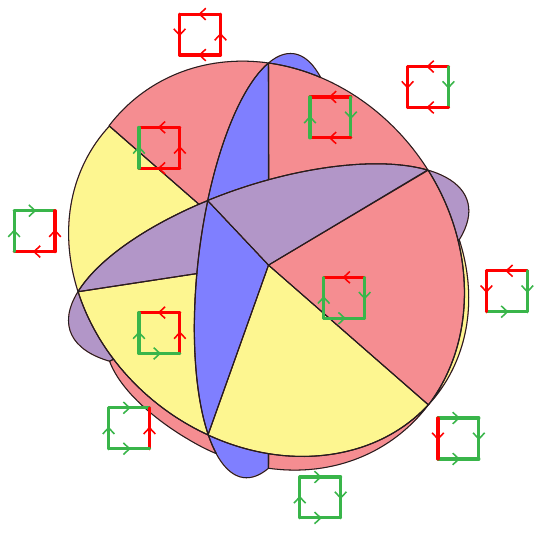} \quad
		\includegraphics[scale=.8]{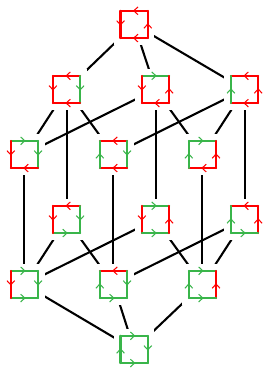}
	}
	\caption{A graphical zonotope, its graphical fan, and its acyclic orientation poset.} 
	\label{fig:graphicalZonotopes}
\end{figure}
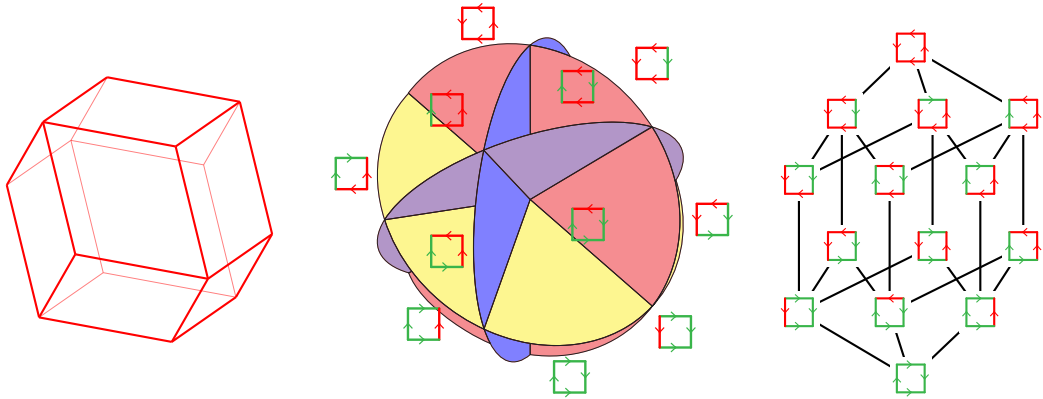

For instance, the graphical zonotope of the complete graph (resp.~path, resp.~empty graph) on~$[n]$ is the permutahedron~$\Perm$ (resp.~a parallelotope denoted~$\Para$, resp.~a point denoted $\Point$).
We need the following definition to describe the face structure of graphical zonotopes.
We refer to~\cite{Greene} and~\cite[Sect.~7]{GreeneZaslavsky} for the original references.

\begin{definition}
\label{def:Gpartition}
A \defn{$G$-ordered partition} is a pair~$\Pi = (\pi, \omega)$, where
\begin{itemize}
\item $\pi$ is a partition of~$[n]$ where each part induces a connected subgraph of~$G$,
\item $\omega$ is an acyclic orientation on the quotient graph~$G/\pi$.
\end{itemize}
It defines a preposet~$\preccurlyeq_\Pi$ on~$[n]$, where $i \preccurlyeq_\Pi j$ if and only if there is a (possibly empty) oriented path in~$\omega$ joining the part of~$\pi$ containing~$i$ to the part of~$\pi$ containing~$j$.
For two $G$-ordered partitions~$\Pi$ and~$\Theta$, we say that~$\Pi$ refines~$\Theta$ (and~$\Theta$ coarsens~$\Pi$) when~$i \preccurlyeq_\Pi j$ implies $i \preccurlyeq_\Theta j$ for any~$i,j \in [n]$.
\end{definition}

\begin{proposition}[{\cite[Sect.~7]{GreeneZaslavsky}}]
\label{prop:facesGraphicalZonotope}
The face lattice of the graphical zonotope~$\Zono[G]$ is isomorphic to the refinement poset on $G$-ordered partitions (augmented with a minimal element).
In particular, 
\begin{itemize}
\item the vertices of~$\Zono[G]$ are in bijection with \defn{acyclic orientations} of~$G$,
\item the facets of~$\Zono[G]$ are in bijection with \defn{biconnected subsets} of~$G$, \ie non-empty connected subset~$U \subset V$ whose complement~$\bar U$ in its connected component of~$G$ is also non-empty and connected.
\end{itemize}
\end{proposition}

For instance for the complete graph~$K_n$, the $K_n$-ordered partitions are all ordered partitions (in the classical sense), the acyclic orientations are given by permutations, and the biconnected subsets are all proper subsets.

\begin{proposition}
\label{prop:graphicalFan}
The normal fan of the graphical zonotope~$\Zono[G]$ is the \defn{graphical fan}~$\Fan[G]$ with one cone~$\polytope{C}(\pi) \eqdef \set{\b{x} \in \R^n}{x_i \le x_j \text{ if } i \preccurlyeq_\Pi j}$ for each~$G$-ordered partition~$\Pi$.
Its walls are given by the arrangement of the hyperplanes~$\set{\b{x} \in \R^n}{x_i = x_j}$ for all~$(i,j) \in E(G)$.
See \cref{fig:graphicalZonotopes}\,(middle).
\end{proposition}

As for the sylvester fan of \cref{prop:sylvesterFan}, the graphical fan of \cref{prop:graphicalFan} coarsens the braid fan of \cref{prop:braidFan}.

\begin{proposition}
\label{prop:fanRefinementGraphical}
The braid fan refines the graphical fan~$\Fan[G]$.
More precisely, for a $G$-ordered partition~$\Pi$, the cone~$\polytope{C}(\Pi)$ is the union of the braid cones~$\polytope{C}(\mu)$ for the ordered partitions~$\mu$ such that~$\preccurlyeq_\mu$ extends~$\preccurlyeq_\Pi$.
\end{proposition}

This can be interpreted as an equivalence relation on permutations and on ordered partitions, similar to the sylvester relation discussed in \cref{def:sylvesterRelation,rem:sylvesterRelation}.

\begin{definition}
\label{def:graphicalRelation}
The \defn{graphical relation~$\equiv_G$} on ordered partitions of~$[n]$ is defined by~$\mu \equiv_G \nu$ if and only if the cones~$\polytope{C}(\mu)$ and~$\polytope{C}(\nu)$ of the braid fan belong precisely to the same cones of the graphical fan~$\Fan[G]$. Equivalently, $\mu \equiv_G \nu$ if and only if $i \preccurlyeq_\mu j \iff i \preccurlyeq_\nu j$ for any edge~$(i,j)$ of~$G$.
It restricts to an equivalence relation on permutations, which can also be seen as the transitive closure of the rewriting rule~$UabV \equiv_G UbaV$ for all words~$U,V$ on~$[n]$ and elements~$a,b$ in~$[n]$ which do not form an edge of~$G$.
\end{definition}

We now orient the graph of the graphical zonotope~$\Zono[G]$ as in \cref{prop:weakOrder,prop:TamariLattice}.

\begin{definition}
\label{def:inversionAcyclicOrientation}
An \defn{inversion} of an acyclic orientation~$\omega$ of an integer graph~$G$ is an edge~$\{i, j\}$ of~$G$ such that~$i < j$ but the edge goes from~$j$ to~$i$ in the orientation~$\omega$.
The \defn{acyclic orientation poset} of~$G$ is the poset on acyclic orientations of~$G$ defined by inclusion of their inversion sets.
See \cref{fig:graphicalZonotopes}\,(right).
\end{definition}

\begin{proposition}
\label{prop:acyclicOrientationPoset}
When oriented in the direction~${\b{\omega} \eqdef (n,\dots,1) - (1,\dots,n) = \sum_{i \in [n]} (n+1-2i) \, \b{e}_i}$, the graph of the graphical zonotope~$\Zono[G]$ is the Hasse diagram of the acyclic orientation poset of~$G$.
\end{proposition}

\begin{remark}
\label{rem:acyclicOrientationPosetNotLattice}
In contrast to \cref{prop:weakOrder,prop:TamariLattice}, the acyclic orientation poset is not always a lattice, as will be discussed in more details in \cref{prop:latticePropertiesGraphicalZonotopes} (see also \cite{Pilaud-acyclicReorientationLattices}).
\end{remark}

In contrast to the permutahedra and braid fans of \cref{subsec:permutahedra} and as illustrated in \cref{fig:graphicalZonotopes}, the graphical zonotope~$\Zono[G]$ is not always simple and the graphical fan~$\Fan[G]$ is not always simplicial.
The following characterization was stated in \cite[Rem.~6.2]{Kim}, \cite[Prop.~5.2]{PostnikovReinerWilliams} and \cite[Prop.~52]{Pilaud-acyclicReorientationLattices} (the immediate proof is omitted in the first two).

\begin{proposition}
\label{prop:simpleGraphicalZonotope}
The graphical zonotope~$\Zono[G]$ is simple (or equivalently, the graphical fan~$\Fan[G]$ is simplicial) if and only if~$G$ is chordful, meaning that any cycle of~$G$ induces a clique of~$G$.
\end{proposition}

We now want to underline that the Cartesian products and Minkowski sums of \cref{def:CartesianProductMinkowskiSum} preserve the family of graphical zonotopes.

\begin{definition}
\label{def:disjointUnionEdgeUnionGraphs}
For two graphs~$G$ and~$H$,
\begin{itemize}
\item if~$V(G) \cap V(H) = \varnothing$, then the \defn{disjoint union}~$G \sqcup H$ is the graph with~$V(G \sqcup H) = V(G) \sqcup V(H)$ and $E(G \sqcup H) = E(G) \sqcup E(H)$,
\item if~$V(G) = V(H)$ and~$E(G) \cap E(H) = \varnothing$, then the \defn{superposition}~$G \oplus H$ is the graph with ${V(G \oplus H) = V(G) = V(H)}$ and $E(G \oplus H) = E(G) \sqcup E(H)$.
\end{itemize}
\end{definition}

\begin{definition}
\label{def:shiftedIntegerGraphs}
For two graphs~$G$ on~$[m]$ and~$H$ on~$[n]$, define
\begin{itemize}
\item the \defn{shifted graph}~$H^{+m}$ as the graph with vertices~$[n]^{+m} \eqdef \{m+1, \dots, m+n\}$ and edges~$E(H)^{+m} \eqdef \set{(m+i, m+j)}{(i,j) \in E(H)}$,
\item the \defn{shifted union} as~$G \otimes H \eqdef G \sqcup H^{+m}$.
\end{itemize}
\end{definition}

\begin{proposition}
\label{prop:operationsGraphsZonotopes}
For all graphs~$G$ on~$[m]$ and~$H$ on~$[n]$,
\begin{itemize}
\item $\Zono[G] \times \Zono[H] = \Zono[G \otimes H]$.
\item if~$m = n$ and~$E(G) \cap E(H) = \varnothing$, then~$\Zono[G] + \Zono[H] = \Zono[G \oplus H]$.
\end{itemize}
\end{proposition}

Note that if $E(G) \cap E(H) \ne \varnothing$, then~$\Zono[G \oplus H]$ has the same combinatorics, but not the same geometry as~$\Zono[G] + \Zono[H]$.
In this paper, we will anyway only need Minkowski sums of graphical zonotopes of graphs with disjoint edge sets.

\enlargethispage{.1cm}
Finally, we briefly describe the graphical zonotopes of complete multipartite graphs, that will play a crucial role in this paper.

\begin{definition}
\label{def:completeMultipartiteGraphicalZonotope}
We consider a $k$-tuple~$\b{n} = (n_1, \dots, n_k)$ of positive integers, and let~${n \eqdef n_1 + \dots + n_k}$ and~$V_1 \eqdef [n_1]$, $V_2 \eqdef [n_2]^{+n_1}$, \dots, $V_k \eqdef [n_k]^{+n_1 + \dots + n_{k-1}}$.
We denote by~$K_{\b{n}}$ the \defn{complete multipartite graph} with vertex set~$V(K_{\b{n}}) \eqdef [n] = V_1 \sqcup \dots \sqcup V_k$ and edge set~$E(K_{\b{n}}) \eqdef \bigcup_{1 \le i < j \le k} V_i \times V_j$.
We denote by~${\polytope{Z}_{\b{n}} = \Zono[K_{\b{n}}]}$ its graphical zonotope.
For~$m,n \in \N$, we often write~$K_{m,n}$ and~$\polytope{Z}_{m,n}$ instead of~$K_{(m,n)}$ and~$\polytope{Z}_{(m,n)}$.
\end{definition}

\begin{figure}
	\centerline{\includegraphics[scale=1]{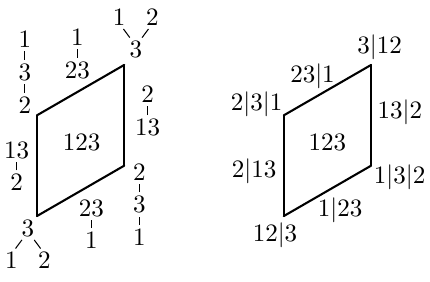}}
	\caption{The graphical zonotope~$\polytope{Z}_{(2,1)}$ with faces labeled by $K_{(2,1)}$-ordered partitions (left) and by ordered partitions of~$[3]$ with no two consecutive parts contained in $\{1,2\}$ (right).}
	\label{fig:labelingsGraphicalZonotope}
\end{figure}

As discussed in \cref{def:Gpartition,prop:facesGraphicalZonotope}, the combinatorial structure of~$\polytope{Z}_{\b{n}}$ is given by $K_{\b{n}}$-ordered partitions.
It turns out that these $K_{\b{n}}$-ordered partitions are almost ordered partitions of~$[n]$ in the classical sense.
Indeed, note that 
\begin{itemize}
\item a subset of~$[n]$ induces a connected subgraph of~$K_{\b{n}}$ if and only if either it is a singleton or it is not contained in one of the~$V_i$'s,
\item two such sets are connected by an edge except if they are two singletons in the same~$V_i$.
\end{itemize}
Therefore, given a $K_{\b{n}}$-ordered partition~$\Pi = (\pi,\omega)$, the preposet~$\preccurlyeq_\Pi^\star$ obtained from~$\preccurlyeq_\Pi$ by adding all relations between incomparable elements of~$\preccurlyeq_\Pi$ is the preposet of an ordered partition, with the property that no two consecutive parts are included in the same~$V_i$.
Conversely, given such an ordered partition~$\mu$, the preposet~$\preccurlyeq_\mu^{\b{n}}$ obtained from~$\preccurlyeq_\mu$ by deleting all relations inside each part of~$\mu$ completely contained in one of the~$V_i$'s is the preposet of a $K_{\b{n}}$-ordered partition.
The correspondence between these two combinatorial descriptions of the faces of~$\polytope{Z}_{\b{n}}$ is illustrated in \cref{fig:labelingsGraphicalZonotope}.
The following statement summarizes this observation.

\begin{proposition}
\label{prop:completeMultipartiteGraphicalZonotopeFaces}
The faces of the graphical zonotope~$\polytope{Z}_{\b{n}}$ are in bijection with the ordered partitions of~$[n]$ where no two consecutive parts are included in the same~$V_i$.
The vertices of~$\polytope{Z}_{\b{n}}$ then correspond to those ordered partitions where each part is included in some~$V_i$.
\end{proposition}

The poset of \cref{prop:acyclicOrientationPoset} can then be read on the partition model as follows.

\begin{proposition}
\label{prop:completeMultipartiteGraphicalZonotopeRotationPosets}
Consider two ordered partitions~$\mu$ and~$\mu'$ where each part is contained in some~$V_i$, and let~$\b{v}$ and~$\b{v}'$ denote the corresponding vertices of~$\polytope{Z}_{\b{n}}$.
There is a path from $\b{v}$ to~$\b{v}'$ in the graph of~$\polytope{Z}_{\b{n}}$ oriented in the direction~${\b{\omega} \eqdef (n,\dots,1) - (1,\dots,n) = \sum_{i \in [n]} (n+1-2i) \, \b{e}_i}$ if and only if $p \preccurlyeq_\mu q$ implies $p \preccurlyeq_{\mu'} q$ for any~$p \in V_i$ and~$q \in V_j$ with~$1 \le i < j \le k$.
\end{proposition}

One also easily derives from \cref{prop:completeMultipartiteGraphicalZonotopeFaces} the number of vertices of~$\polytope{Z}_{\b{n}}$ by encoding such an ordered partition into a word with no consecutive identical letters and some surjections.
For ${\b{n} = (m,n)}$, this yields poly-Bernoulli numbers, see \mbox{\cite{Kaneko-polyBernoulli, ArakawaKaneko-polyBernoulli, ArakawaKaneko-multipleZetaValues, CameronGlassSchumacher, BenyiHajnal-polyBernoulli1, BenyiHajnal-polyBernoulli2}}.

\begin{proposition}
\label{prop:completeMultipartiteGraphicalZonotopeNumberVertices}
Let~$\surjections{n}{k}$ denotes the number of surjections from~$[n]$ to~$[k]$ (see \OEIS{A019538} in~\cite{OEIS}).
The number of vertices of the graphical zonotope~$\polytope{Z}_{\b{n}}$ is given by the summation formula
\[
\sum_{w \in W_k} \prod_{i \in [k]} \surjections{n_i}{|w|_i},
\]
where $W_k$ is the set of words on the alphabet~$[k]$ containing at least one copy of each letter and no consecutive identical letters, and $|w|_i$ denotes the number of letters~$i$ in the word~$w$.
In particular, when~$\b{n} = (m,n)$, we obtain the poly-Bernoulli number (see \cite{Kaneko-polyBernoulli}, \OEIS{A099594} in~\cite{OEIS}, and \cite{CameronGlassSchumacher} for an explanation of the formula)
\[
B(-m,n) \eqdef \sum_{\ell \ge 0} \frac{\surjections{m+1}{\ell+1} \, \surjections{n+1}{\ell+1}}{(\ell+1)^2}.
\]
\end{proposition}


The number of facets of~$\polytope{Z}_{\b{n}}$ will appear later as a special case of \cref{prop:numberFacetsShuffleGraphicalZonotopes}.


\subsection{Deformed permutahedra}
\label{subsec:deformedPermutahedra}

We now consider deformations of the permutahedron of \cref{subsec:permutahedra}, introduced by A.~Postnikov~\cite{Postnikov, PostnikovReinerWilliams}.
They are usually called ``generalized permutahedra'' but we prefer the term ``deformed permutahedra'' which we find more explicit.

\begin{definition}
\label{def:generalizedPermutahedron}
A \defn{deformed permutahedron} is a polytope whose normal fan coarsens that of the permutahedron~$\Perm$.
We denote by~$\DefoPerms$ the set of deformed permutahedra in~$\R^n$.
\end{definition}

\begin{remark}
\label{rem:alternativeDefinitionsDeformedPermutahedra}
There are further equivalent definitions of deformed permutahedra, among others:
\begin{itemize}
\item they are all polytopes obtained by moving parallely the facet defining inequalities of the permutahedron~$\Perm$ without passing any vertex~\cite{Postnikov,PostnikovReinerWilliams},
\item their right-hand-sides are described by submodular functions~\cite{Edmonds,Postnikov,PostnikovReinerWilliams},
\item they are all weak Minkowski summands of the permutahedron~\cite{Meyer,McMullen-typeCone},
\item they are all polytopes obtained by Minkowski sums and differences of faces of the standard simplex~\cite{ArdilaBenedettiDoker}.
\end{itemize}
\end{remark}

\begin{example}
Examples of deformed permutahedra include permutahedra (see \cref{subsec:permutahedra}), associahedra (see \cref{subsec:associahedra}), graphical zonotopes (see \cref{subsec:graphicalZonotopes}), and all polytopes discussed in this paper in particular multiplihedra (see \cref{subsec:multiplihedra}), constrainahedra (see \cref{subsec:constrainahedra}), and biassociahedra (see \cref{subsec:biassociahedra}).
\end{example}

By \cref{def:generalizedPermutahedron}, the normal cones of the faces of a deformed permutahedron are defined by inequalities of the form~$x_i \le x_j$.
This justifies the following definition.

\begin{definition}
\label{def:facePreposet}
Each face~$\polytope{F}$ of a deformed permutahedron~$\polytope{P}$ defines a preposet~$\preccurlyeq_\polytope{F}$ on~$[n]$ such that the normal cone of~$\polytope{F}$ is given by~$\set{\b{x} \in \R^n}{x_i \le x_j \text{ if } i \preccurlyeq_\polytope{F} j}$.
This preposet is a poset when~$\polytope{F}$ is a vertex of~$\polytope{P}$.
We call them \defn{face preposets} of~$\polytope{P}$ or shortly \defn{$\polytope{P}$-preposets}, and \defn{vertex poset} of~$\polytope{P}$ or shortly \defn{$\polytope{P}$-posets}.
The face lattice of~$\polytope{P}$ is isomorphic to the refinement lattice on the $\polytope{P}$-preposets.
\end{definition}

\begin{remark}
\label{rem:simpleDeformedPermutahedra}
In contrast to the permutahedra of \cref{subsec:permutahedra} and the associahedra of \cref{subsec:associahedra}, not all deformed permutahedra are simple polytopes.
It is immediate that a deformed permutahedron~$\polytope{P}$ is simple if and only if the Hasse diagrams of its vertex posets are all forests.
\end{remark}

These preposets (resp.~posets) naturally define an equivalence relation on ordered partitions (resp.~on permutations), similar to the sylvester congruence presented in \cref{def:sylvesterRelation}.

\begin{definition}
\label{def:equivalenceRelationDeformedPermutahedron}
A deformed permutahedron~$\polytope{P}$ defines an equivalence relation~$\equiv_{\polytope{P}}$ on ordered partitions by~$\mu \equiv_{\polytope{P}} \nu$ if and only if the cones~$\polytope{C}(\mu)$ and~$\polytope{C}(\nu)$ of the braid fan belong precisely to the same cones of the normal fan of~$\polytope{P}$.
Said differently, each face~$\polytope{F}$ of~$\polytope{P}$ defines an equivalence class of~$\equiv_\polytope{P}$ consisting in all ordered partitions~$\mu$ such that~$i \preccurlyeq_\polytope{F} j$ implies $i \preccurlyeq_\mu j$ for all~$i,j \in [n]$.
This relation~$\equiv_\polytope{P}$ also restrict to an equivalence relation on permutations, with one equivalence class for each vertex of~$\polytope{P}$.
\end{definition}

\begin{remark}
\label{rem:equivalenceNotCongruence}
In contrast to the sylvester congruence~$\equiv_\sylv$ presented in \cref{def:sylvesterRelation}, the equivalence relation~$\equiv_{\polytope{P}}$ is not necessarily a lattice congruence of the weak order.
See \cref{subsec:latticeProperties}.
\end{remark}

We now orient the graphs of arbitrary deformed permutahedra as in \cref{prop:weakOrder,prop:TamariLattice,prop:acyclicOrientationPoset}.

\begin{definition}
\label{def:rotationPoset}
The \defn{rotation graph} of~$\polytope{P} \in \DefoPerms[n]$ is the directed graph on~$\polytope{P}$-posets obtained by orienting the graph of~$\polytope{P}$ in the direction~${\b{\omega} \eqdef (n,\dots,1) - (1,\dots,n) = \sum_{i \in [n]} (n+1-2i) \, \b{e}_i}$.
The \defn{rotation poset}~$\le_\polytope{P}$ is the transitive closure of the rotation graph.
\end{definition}

\begin{remark}
\label{rem:equivalenceNotLattice}
In contrast to the Tamari lattice presented in \cref{def:TamariLattice}, the rotation poset~$\le_\polytope{P}$ is not always a lattice.
See \cref{subsec:latticeProperties}.
\end{remark}

We finally want to underline that the Cartesian products and Minkowski sums of \cref{def:CartesianProductMinkowskiSum,prop:CartesianProductMinkowskiSum} preserve deformed permutahedra.

\begin{proposition}
\label{prop:CartesianProductMinkowskiSumDeformedPermutahedra}
Let~$\polytope{P} \in \DefoPerms[m]$ and~$\polytope{Q} \in \DefoPerms[n]$ be two deformed permutahedra. Then
\begin{itemize}
\item the Cartesian product~${\polytope{P} \times \polytope{Q}}$ is a deformed permutahedron in~$\DefoPerms[m+n]$,
\item if~$m = n$, then the Minkowski sum~$\polytope{P} + \polytope{Q}$ is a deformed permutahedra in~$\DefoPerms[m]$.
\end{itemize}
\end{proposition}

To describe the resulting face preposets, equivalence relations on ordered partitions (or on permutations), and rotation posets on vertex posets, we need the following standard notations.

\begin{definition}
\label{def:restrictionShift}
For a preposet~$\preccurlyeq$ on~$[n]$ and an integer~$m \in [n]$, we define
\begin{itemize}
\item by~$\preccurlyeq_{[m]}$ the \defn{restriction} of~$\preccurlyeq$ to~$[m]$,
\item by~$\preccurlyeq^{\pm m}$ the \defn{shift} of~$\preccurlyeq$ by~$\pm m$, defined by~$i \pm m \preccurlyeq^{\pm m} j \pm m \iff i \preccurlyeq j$.
\end{itemize}
We use similar notations for ordered partitions and permutations.
\end{definition}

Using the notations of \cref{def:restrictionShift}, we first describe the behavior of the Cartesian product and Minkowski sum on the face preposets of \cref{def:facePreposet}.

\begin{proposition}
\label{prop:CartesianProductMinkowskiSumFacePreposets}
For two deformed permutahedra~$\polytope{P} \in \DefoPerms[m]$ and~$\polytope{Q} \in \DefoPerms[n]$ and two faces~$\polytope{F}$ of~$\polytope{P}$ and~$\polytope{G}$ of~$\polytope{Q}$,
\begin{itemize}
\item $\preccurlyeq_\polytope{F} \sqcup \; {\preccurlyeq_\polytope{G}}^{+m}$ is a face preposet of~$\polytope{P} \times \polytope{Q}$,
\item when~$m = n$, if~$\preccurlyeq_\polytope{F}$ and~$\preccurlyeq_\polytope{G}$ have a common extension and any three of the relations~${i \preccurlyeq_\polytope{F} j}$, ${j \preccurlyeq_\polytope{F} i}$, ${i \preccurlyeq_\polytope{G} j}$, ${j \preccurlyeq_\polytope{G} i}$ imply the fourth, then the transitive closure of ${\preccurlyeq_\polytope{F} \cup \preccurlyeq_\polytope{G}}$ is a face preposet~of~$\polytope{P} + \polytope{Q}$.
\end{itemize}
Moreover, any face preposet of~$\polytope{P} \times \polytope{Q}$ and~$\polytope{P} + \polytope{Q}$ is uniquely obtained this way.
\end{proposition}

We next describe the behavior of the Cartesian product and Minkowski sum on the equivalence relations on ordered partitions of \cref{def:equivalenceRelationDeformedPermutahedron}.

\begin{proposition}
\label{prop:CartesianProductMinkowskiSumEquivalenceRelations}
For two deformed permutahedra~$\polytope{P} \in \DefoPerms[m]$ and~$\polytope{Q} \in \DefoPerms[n]$, and two ordered partitions~$\mu \in \f{P}_m$ and~$\nu \in \f{P}_n$,
\begin{itemize}
\item $\mu \equiv_{\polytope{P} \times \polytope{Q}} \nu$ if and only if~$\mu_{[m]} \equiv_{\polytope{P}} \nu_{[m]}$ and~${\mu^{-m}}_{[n]} \equiv_{\polytope{Q}} {\nu^{-m}}_{[n]}$,
\item if~$m = n$, then~${\mu \equiv_{\polytope{P} + \polytope{Q}} \nu}$ if and only if $\mu \equiv_{\polytope{P}} \nu$ and~$\mu \equiv_{\polytope{Q}} \nu$.
\end{itemize}
\end{proposition}

Finally, we describe the behavior of the Cartesian product and Minkowski sum on the rotation posets of \cref{def:rotationPoset}.

\begin{proposition}
\label{prop:CartesianProductMinkowskiSumRotationPosets}
For two deformed permutahedra~$\polytope{P} \in \DefoPerms[m]$ and~$\polytope{Q} \in \DefoPerms[n]$, and four vertices ${\b{v}, \b{v}' \in \polytope{P}}$ and $\b{w}, \b{w}' \in \polytope{Q}$, we have 
\begin{itemize}
\item ${\preccurlyeq_{\b{v}} \sqcup \; {\preccurlyeq_{\b{w}}}^{+m}} \le_{\polytope{P} \times \polytope{Q}} {\preccurlyeq_{\b{v}'} \sqcup \; {\preccurlyeq_{\b{w}'}}^{+m}}$ if and only if ${\preccurlyeq_{\b{v}}} \le_\polytope{P} {\preccurlyeq_{\b{v}'}}$ and~${\preccurlyeq_{\b{w}}} \le_\polytope{Q} {\preccurlyeq_{\b{w}'}}$,
\item if~$m = n$ and the transitive closure~$\preccurlyeq$ of $\preccurlyeq_{\b{v}} \cup \preccurlyeq_{\b{w}}$ (resp.~$\preccurlyeq'$ of $\preccurlyeq_{\b{v}'} \cup \preccurlyeq_{\b{w}'}$) is a vertex poset of~$\polytope{P} + \polytope{Q}$, then~${\preccurlyeq} \; \le_{\polytope{P} + \polytope{Q}} \; {\preccurlyeq}'$ if and only if ${\preccurlyeq_{\b{v}}} \le_\polytope{P} {\preccurlyeq_{\b{v}'}}$ and~${\preccurlyeq_{\b{w}}} \le_\polytope{Q} {\preccurlyeq_{\b{w}'}}$.
\end{itemize}
\end{proposition}


\subsection{Lattice properties of rotation posets}
\label{subsec:latticeProperties}

As mentioned in \cref{prop:weakOrder,prop:TamariLattice,prop:acyclicOrientationPoset}, the graphs of the permutahedron~$\Perm$, of the associahedron~$\Asso$, and of the graphical zonotope~$\Zono[G]$, oriented in the direction~$\b{\omega}$ are the Hasse diagrams of the weak order, of the Tamari lattice, and of the acyclic orientation poset of~$G$, respectively.
More generally, we have defined in \cref{def:rotationPoset} the rotation poset~$\le_\polytope{P}$ of a deformed permutahedron~$\polytope{P}$ by orienting its graph in the direction~$\b{\omega}$.
It turns out that the Tamari lattice is a lattice quotient of the weak order.
More generally, the graph of any quotientope of~\cite{PilaudSantos-quotientopes} oriented by~$\b{\omega}$ is the Hasse diagram of a lattice quotient of the weak order.
In contrast, not all acyclic orientation posets are lattices, even less lattice quotients of the weak order.
In this section, we discuss lattice properties of the rotation posets of deformed permutahedra.
We start by recalling some basic facts about lattice congruences.

\begin{definition}
\label{def:quotientRelation}
Given a binary relation~$R$ and an equivalence relation~$\equiv$ on the same set~$X$, the \defn{quotient relation}~$R/{\equiv}$ is the binary relation on~$X/{\equiv}$ defined by~$I \; R/{\equiv} \; J$ if and only if there is $i \in I$ and $j \in J$ such that $i \; R \; j$.
\end{definition}

\begin{proposition}
\label{prop:rotationPoset}
For any deformed permutahedron~$\polytope{P}$, the rotation poset~$\le_\polytope{P}$ is the poset quotient of the weak order on~$\f{S}_n$ by the equivalence relation~$\equiv_\polytope{P}$.
\end{proposition}

\begin{definition}
\label{def:latticeCongruence}
A \defn{congruence} of a lattice~$(L, \le, \meet, \join)$ is an equivalence relation on~$L$ compatible with the meet and join operations, meaning that~$x \meet y \equiv x' \meet y'$ and~$x \join y \equiv x' \join y'$ for any~$x \equiv x'$ and~$y \equiv y'$.
The quotient $\le/{\equiv}$ is then automatically a lattice on~$L/{\equiv}$.
\end{definition}

\begin{proposition}
\label{prop:characterizationLatticeCongruence}
An equivalence relation~$\equiv$ on a lattice~$L$ is a lattice congruence if and only if
\begin{itemize}
\item its equivalent classes are intervals of~$L$, 
\item the map~$\projDown$ (resp.~$\projUp$) sending an element to the minimum (resp.~maximum) element in its equivalence class is order preserving.
\end{itemize}
\end{proposition}

In view of \cref{prop:characterizationLatticeCongruence}, we define the following properties of equivalence relations on permutations.
Note that the first two properties are independent, and are both implied by (but do not imply) the third one.

\begin{definition}
\label{def:intervalLatticeProperty}
We say that an equivalent relation~$\equiv$ on~$\f{S}_n$ has
\begin{itemize}
\item the \defn{ interval property} if its classes are intervals of the weak order,
\item the \defn{lattice property} if the quotient of the weak order by~$\equiv$ is a lattice on~$\f{S}_n/{\equiv}$,
\item the \defn{congruence property} if it is a lattice congruence (see \cref{def:latticeCongruence,prop:characterizationLatticeCongruence}).
\end{itemize}
By extension, we say that a deformed permutahedron~$\polytope{P}$ has these properties when the corresponding equivalence relation~$\equiv_\polytope{P}$ of \cref{def:equivalenceRelationDeformedPermutahedron} does.
In particular, $\polytope{P}$ has the lattice property when the rotation poset~$\le_\polytope{P}$ is a lattice.
\end{definition}

To illustrate these notions, we characterize in the next statements the graphs whose zonotope has the interval, the lattice, or the congruence property.
We will need the following definitions, see \cite{BarnardMcConville, Pilaud-acyclicReorientationLattices}.

\begin{definition}
\label{def:filledEdges}
An integer graph~$G$ is 
\begin{itemize}
\item \defn{filled} if ${(i,k) \in E(G)}$ implies~$(i,j) \in E(G)$ and~$(j,k) \in E(G)$ for all~$i < j < k$,
\item \defn{half-filled} if ${(i,k) \in E(G)}$ implies~$(i,j) \in E(G)$ or~$(j,k) \in E(G)$ for all~$i < j < k$,
\item \defn{vertebrate} if the transitive reduction of any induced subgraph of~$G$ is a forest.
\end{itemize}
\end{definition}

\begin{proposition}
\label{prop:latticePropertiesGraphicalZonotopes}
The graphical zonotope~$\Zono$ has the interval (resp.~lattice, resp.~congruence) property if and only if~$G$ is half-filled (resp.~vertebrate, resp.~filled).
\end{proposition}

The characterizations of the lattice and congruence properties in \cref{prop:latticePropertiesGraphicalZonotopes} were already proved in~\cite{Pilaud-acyclicReorientationLattices}.
We just prove here the characterization of the interval property as it did not appear in the literature.
For this, we need the classical characterization of the weak order intervals~\cite{BjornerWachs}.

\begin{proposition}
\label{prop:characterizationWOIP}
A poset~$\less$ on~$[n]$ defines an interval of the weak order if and only if~$i \less k$ implies~$i \less j$ or~$j \less k$, and~$i \more k$ implies~$i \more j$ or~$j \more k$, for every~${1 \le i < j < k \le n}$.
\end{proposition}

\begin{proof}[Proof of \cref{prop:latticePropertiesGraphicalZonotopes}]
As just explained, we only prove here the characterization of the interval property.
Assume first that~$G$ is half-filled.
Consider a poset~$\preccurlyeq_O$ corresponding to an acyclic orientation~$O$ of~$G$ and let~${1 \le i < j < k \le n}$ be such that $i \preccurlyeq_O k$.
By definition, there is a sequence~$i = j_0, j_1, \dots, j_p = k$ such that~$(j_{q-1}, j_q)$ is an oriented arc of~$O$ for all~$q \in [p]$.
Moreover, since~${1 \le i < j < k \le n}$, there is~$q \in [p]$ such that~${j_{q-1} < j \le j_q}$.
If~$j = j_q$, then we obtain that~$i \preccurlyeq_O j$ and~$j \preccurlyeq k$.
Otherwise, since~${(j_{q-1}, j_q) \in E(G)}$ and~$G$ is half-filled, we also have~$(j_{q-1}, j) \in E(G)$ or~$(j, j_q) \in E(G)$.
Assume for instance that~$(j, j_q) \in E(G)$ (the other case is symmetric).
If the edge~$(j, j_q)$ is oriented from~$j$ to~$j_q$ in~$O$, then we obtain that~$j \preccurlyeq_O j_q \preccurlyeq_O k$, so that~$j \preccurlyeq_O k$.
Otherwise, we have~$i \preccurlyeq_O j_q \preccurlyeq j$ so that~$i \preccurlyeq j$.
Therefore, $i \preccurlyeq_O k$ implies $i \preccurlyeq_O j$ or~$j \preccurlyeq_O k$.
By symmetry, we conclude from~\cref{prop:characterizationWOIP} that $\Zono$ has the interval property.
Conversely, if~$G$ is not half-filled, it is immediate to construct an acyclic orientation~$O$ of~$G$ whose corresponding poset~$\preccurlyeq_O$ fails to satisfy the conditions of \cref{prop:characterizationWOIP}.
\end{proof}

\begin{corollary}
\label{coro:latticePropertiesBipartiteGraphicalZonotope}
The graphical zonotope~$\polytope{Z}_{m,n} \eqdef \Zono[K_{m,n}]$ has the interval (resp.~lattice, resp.~congruence) property if and only if~$m,n \ge 1$ (resp.~$m = 1$ or $n = 1$, resp.~$m = n = 1$).
\end{corollary}

\begin{proof}
It follows immediately from \cref{def:filledEdges} that the complete bipartite graph~$K_{m,n}$ is always half-filled, vertebrate only when~$m = 1$ or $n = 1$, and filled only when~$m = n = 1$.
\end{proof}

We finally want to underline which of the properties of \cref{def:intervalLatticeProperty} are preserved by the Cartesian product and the Minkowski sum of \cref{def:CartesianProductMinkowskiSum}.
The proofs are immediate for the Cartesian product, and rely on the fact that the congruence~$\equiv_{\polytope{P} + \polytope{Q}}$ is the intersection of the congruences~$\equiv_\polytope{P}$ and~$\equiv_\polytope{Q}$ for the Minkowski sum.

\begin{proposition}
\label{prop:CartesianProductMinkowskiSumLatticeProperties}
The Cartesian product preserves the interval, lattice, and congruence properties.
The Minkowski sum preserve the interval and congruence properties, but not the lattice property.
\end{proposition}


\section{Shuffles of deformed permutahedra}
\label{sec:shuffles}

In this section, we introduce the shuffle operation on deformed permutahedra (\cref{subsec:shuffle}), provide a combinatorial description of the resulting polytopes (\cref{subsec:combinatorialDescription}), and discuss the shuffle with a point (\cref{subsec:shufflePoint}) and the shuffle of graphical zonotopes (\cref{subsec:shuffleGraphicalZonotopes}).


\subsection{Shuffle operation}
\label{subsec:shuffle}

This paper focuses on the following operation on the deformed permutahedra of \cref{subsec:deformedPermutahedra}.

\begin{definition}
\label{def:shuffleDeformedPermutahedra}
The \defn{shuffle} of two deformed permutahedra~$\polytope{P} \in \DefoPerms[m]$ and~$\polytope{Q} \in \DefoPerms[n]$ is
\[
\polytope{P} \shuffleDP \polytope{Q} \eqdef (\polytope{P} \times \polytope{Q}) + \polytope{Z}_{m,n} = (\polytope{P} \times \polytope{Q}) + \sum_{\substack{i \in [m] \\ j \in [n]}} [\b{e}_i, \b{e}_{m+j}],
\]
where~$\times$ denotes the Cartesian product, and~$+$ and~$\sum$ the Minkowski sum (see \cref{def:CartesianProductMinkowskiSum}).
\end{definition}

For instance, we have~$\Perm[m] \shuffleDP \Perm[n] = \Perm[m+n]$.
We will study in more details certain particular shuffles: the shuffle with a point in \cref{subsec:shufflePoint}, shuffles of graphical zonotopes in \cref{subsec:shuffleGraphicalZonotopes}, and shuffles of permutahedra and associahedra in \cref{sec:multiplihedra,sec:constrainahedra,sec:biassociahedra}.
At the moment, we observe that the shuffle operation~$\shuffleDP$ preserves the family of deformed permutahedra, which directly follows from \cref{def:shuffleDeformedPermutahedra,prop:CartesianProductMinkowskiSumDeformedPermutahedra}.

\begin{proposition}
\label{prop:shuffleDeformedPermutahedra}
For all deformed permutahedra~$\polytope{P} \in \DefoPerms[m]$ and~$\polytope{Q} \in \DefoPerms[n]$, the shuffle~$\polytope{P} \shuffleDP \polytope{Q}$ is a deformed permutahedron in~$\DefoPerms[m+n]$.
\end{proposition}

We now gather in~\cref{rem:shuffleAssociative,rem:simpleShuffle,rem:latticeShuffle} some elementary observations on the shuffle operation~$\shuffleDP$.

\begin{remark}
\label{rem:shuffleAssociative}
The shuffle is an associative operation on deformed permutahedra.
Indeed, for any $k$ deformed permutahedra~$\polytope{P}_1 \in \DefoPerms[n_1], \dots, \polytope{P}_k \in \DefoPerms[n_k]$, we have
\[
\polytope{P}_1 \shuffleDP \cdots \shuffleDP \polytope{P}_k = (\polytope{P}_1 \times \cdots \times \polytope{P}_k) + \polytope{Z}_{(n_1, ..., n_k)}.
\]
The shuffle is also commutative up to permutation of coordinates.
Indeed, for any deformed permutahedra~$\polytope{P} \in \DefoPerms[m]$ and~$\polytope{Q} \in \DefoPerms[n]$, we have~$\polytope{P} \shuffleDP \polytope{Q} = s(\polytope{Q} \shuffleDP \polytope{P})$ where~$s : \R^{n+m} \to \R^{m+n}$ denotes the swap~$s(x,y) = (y, x)$.
\end{remark}

\begin{remark}
\label{rem:simpleShuffle}
The shuffle operation~$\shuffleDP$ does not preserve simple polytopes.
For instance, while the permutahedron~$\Perm$ of \cref{subsec:permutahedra} and the associahedron~$\Asso$ of \cref{subsec:associahedra} are simple, the multiplihedron~$\Multiplihedron \eqdef \Perm[m] \shuffleDP \Asso[n]$ of \cref{sec:multiplihedra}, the constrainahedron ${\Constrainahedron \eqdef \Asso[m] \shuffleDP \Asso[n]}$ of \cref{sec:constrainahedra}, and the biassociahedron $\Biassociahedron \eqdef \Asso[m] \shuffleDP \Asso[n]$ of \cref{sec:biassociahedra} are not simple in general (see \cref{rem:simpleMultiplihedron,rem:simpleConstrainahedron,rem:simpleBiassociahedron}).
\end{remark}


\subsection{Combinatorial description}
\label{subsec:combinatorialDescription}

We now aim at describing the behavior of the shuffle operation~$\shuffleDP$ of \cref{def:shuffleDeformedPermutahedra} in terms of the face preposets of \cref{def:facePreposet}.
Such a description immediately follows from \cref{prop:CartesianProductMinkowskiSumFacePreposets,prop:facesGraphicalZonotope}.
A more convenient description arises by combining as well with the description of the face preposets of~$\polytope{Z}_{m,n}$ provided in \cref{prop:completeMultipartiteGraphicalZonotopeFaces}.
Recall that for an ordered partition~$\mu$ on~$[m+n]$, we denote by~$\preccurlyeq_\mu^{m,n}$ the preposet obtained from~$\preccurlyeq_\mu$ by deleting all relations inside each part of~$\mu$ completely contained in~$[m]$ or in~$[n]^{+m}$.

\begin{proposition}
\label{prop:shuffleFacePreposets}
Consider two deformed permutahedra~$\polytope{P} \in \DefoPerms[m]$ and~$\polytope{Q} \in \DefoPerms[n]$, two faces~$\polytope{F}$ of~$\polytope{P}$ and~$\polytope{G}$ of~$\polytope{Q}$, and an ordered partition~$\mu$ of~$[m+n]$ such that 
\begin{itemize}
\item $\preccurlyeq_\mu$ extends both $\preccurlyeq_\polytope{F}$ and  ${\preccurlyeq_\polytope{G}}^{+m}$,
\item no two consecutive parts of~$\mu$ are both contained in $[m]$ or both contained in~$[n]^{+m}$,
\item if~$\mu_k \cap [m] \ne \varnothing \ne \mu_k \cap [n]^{+m}$, then any two elements of~$\mu_k \cap [m]$ are equal or incomparable in~$\preccurlyeq_\polytope{F}$ and any two elements of~$\mu_k \cap [n]^{+m}$ are equal or incomparable~in~${\preccurlyeq_\polytope{G}}^{+m}$.
\end{itemize}
Then the preposet~$\preccurlyeq_{\polytope{F}, \polytope{G}, \mu} \eqdef (\preccurlyeq_\polytope{F} \sqcup \; {\preccurlyeq_\polytope{G}}^{+m}) \; \cup \preccurlyeq_\mu^{m,n}$ is a face preposet of~$\polytope{P} \shuffleDP \polytope{Q}$, 
and any face preposet of~$\polytope{P} \shuffleDP \polytope{Q}$ is uniquely obtained this way.
\end{proposition}

\begin{proof}
Combine  \cref{prop:CartesianProductMinkowskiSumFacePreposets,prop:completeMultipartiteGraphicalZonotopeFaces}.
\end{proof}

\begin{remark}
\label{rem:shuffleFaces}
The deformed permutahedron~$\polytope{P}$ (resp.~$\polytope{Q}$) itself appear as a face of~$\polytope{P} \shuffleDP \polytope{Q}$.
The corresponding face preposets are given by~$\preccurlyeq_{\polytope{P}, \b{w}, \mu}$ (resp.~$\preccurlyeq_{\b{v}, \polytope{Q}, \mu}$) where~$\b{w}$ (resp.~$\b{v}$) is an arbitrary vertex of~$\polytope{Q}$ (resp.~of~$\polytope{P}$) and $\mu$ is one of the two ordered partitions with parts~$[m]$ and~$[n]^{+m}$.
\end{remark}

\begin{remark}
\label{rem:biPreposets}
The face preposet~$\preccurlyeq_{\polytope{F}, \polytope{G}, \mu}$ of \cref{prop:shuffleFacePreposets} can be represented visually by drawing the Hasse diagrams of the face preposets~$\preccurlyeq_\polytope{F}$ and~${\preccurlyeq_\polytope{G}}^{+m}$ side by side, with their vertices separated in blocks organized from bottom to top according to~$\mu$. Then $i \preccurlyeq_{\polytope{F}, \polytope{G}, \mu} j$ if
\begin{itemize}
\item either there is an oriented path from~$i$ to~$j$ in~$\preccurlyeq_\polytope{F}$ or in~${\preccurlyeq_\polytope{G}}^{+m}$,
\item or~$i$ is in a block lower than~$j$, 
\item or~$i$ and~$j$ belong to the same block which is not contained in~$[m]$ or in~$[n]^{+m}$.
\end{itemize}
We call such pictures \defn{$(\polytope{P}, \polytope{Q})$-bipreposets}.
Examples of bipreposets where the preposets are trees are illustrated in \cref{fig:cotrees,fig:bitrees}.
\end{remark}

\begin{remark}
\label{rem:biPosets}
The preposet~$\preccurlyeq_{\polytope{F}, \polytope{G}, \mu}$ is a poset if and only if $\polytope{F}$ and~$\polytope{G}$ are vertices, and the parts of~$\mu$ are alternatively contained in~$[m]$ and~$[n]^{+m}$.
In other words, the vertex posets of~$\polytope{P} \shuffleDP \polytope{Q}$ are obtained by interspersing the vertex posets of~$\polytope{P}$ with the vertex posets of~$\polytope{Q}$ as explained in \cref{rem:biPreposets}.
We call such pictures \defn{$(\polytope{P}, \polytope{Q})$-biposets}.
\end{remark}

\cref{rem:biPosets} yields the following statement.

\begin{definition}
\label{def:partitionedPoset}
A \defn{partitioned poset} is a pair~$(\trianglelefteq, \mu)$ where~$\trianglelefteq$ is a poset on~$[n]$ and~$\mu$ is an ordered partition of~$[n]$ such that~$i \trianglelefteq j$ implies~$i \preccurlyeq_\mu j$.
\end{definition}

\begin{corollary}
\label{coro:numberVertices}
The number of vertices of~$\polytope{P} \shuffleDP \polytope{Q}$ is given by the summation formula
\[
\sum_\ell \pp{\ell}{\polytope{P}} \, \big( \pp{\ell-1}{\polytope{Q}} + 2 \, \pp{\ell}{\polytope{Q}} + \pp{\ell+1}{\polytope{Q}} \big),
\]
where~$\pp{\ell}{\polytope{P}}$ denotes the number of partitioned posets~$(\trianglelefteq, \mu)$ where~$\trianglelefteq$ is a vertex poset of~$\polytope{P}$ and~$\mu$ has~$\ell$ parts.
In particular, it only depends on the repartition of partitioned vertex posets of~$\polytope{P}$ and~$\polytope{Q}$.
\end{corollary}

\begin{remark}
\label{rem:noSymmetry}
\cref{coro:numberVertices} implies for instance that the constrainahedron~$\Constrainahedron \eqdef \Asso[m] \shuffleDP \Asso[n]$ of \cref{sec:constrainahedra} and the biassociahedron~$\Biassociahedron \eqdef \Ossa[m] \shuffleDP \Asso[n]$ of \cref{sec:biassociahedra} have the same number of vertices for any~$m,n \ge 1$.
This symmetry property is lost beyond vertices: for instance, $\Constrainahedron[3,3]$ has $1550$ edges, while $\Biassociahedron[3,3]$ has $1549$ edges.
\cref{coro:numberVertices} also implies that ${\Perm[m] \shuffleDP \Para[n]}$ and~${\Perm[m+1] \shuffleDP \Point[n-1]}$ have the same number of vertices while their number of facets differ for~$n \ge 4$, see \cref{rem:numberVerticesGraphicalZonotopes}.
\end{remark}

We now describe the behavior of the shuffle operation~$\shuffleDP$ of \cref{def:shuffleDeformedPermutahedra} at the level of the equivalence relations on ordered partitions and permutations of~\cref{def:equivalenceRelationDeformedPermutahedron}.
It immediately follows from \cref{def:graphicalRelation,prop:CartesianProductMinkowskiSumEquivalenceRelations}.

\begin{proposition}
\label{prop:shuffleEquivalenceRelations}
For all deformed permutahedra~$\polytope{P} \in \DefoPerms[m]$ and~$\polytope{Q} \in \DefoPerms[n]$, the equivalence relation~$\equiv_{\polytope{P} \shuffleDP \polytope{Q}}$ on ordered partitions is given by $\mu \equiv_{\polytope{P} \shuffleDP \polytope{Q}} \nu$ if and only if~$\mu_{[m]} \equiv_{\polytope{P}} \nu_{[m]}$ while ${\mu^{-m}}_{[n]} \equiv_{\polytope{Q}} {\nu^{-m}}_{[n]}$ and~$i \preccurlyeq_\mu m+j \iff i \preccurlyeq_\nu m+j$ for all~$i \in [m]$ and~$j \in [n]$.
\end{proposition}

Finally, we describe the behavior of the shuffle operation~$\shuffleDP$ of \cref{def:shuffleDeformedPermutahedra} on rotation posets of \cref{def:rotationPoset}.
It immediately follows from \cref{prop:completeMultipartiteGraphicalZonotopeRotationPosets,prop:CartesianProductMinkowskiSumRotationPosets}.

\begin{proposition}
\label{prop:shuffleRotationPosets}
For all deformed permutahedra~$\polytope{P} \in \DefoPerms[m]$ and~$\polytope{Q} \in \DefoPerms[n]$, the rotation poset~$\le_{\polytope{P} \shuffleDP \polytope{Q}}$ on the vertex posets of~$\polytope{P} \shuffleDP \polytope{Q}$ is given by~${\preccurlyeq_{\b{v}, \b{w}, \mu}} \le_{\polytope{P} \shuffleDP \polytope{Q}} {\preccurlyeq_{\b{v}', \b{w}', \mu'}}$ if and only if~${\preccurlyeq_{\b{v}}} \le_{\polytope{P}} {\preccurlyeq_{\b{v}'}}$ and ${\preccurlyeq_{\b{w}}} \le_{\polytope{Q}} {\preccurlyeq_{\b{w}'}}$ and~$p \preccurlyeq_\mu q$ implies~$p \preccurlyeq_{\mu'} q$ for all~$p \in [m]$ and~$q \in [n]^{+m}$.
\end{proposition}

\begin{remark}
\label{rem:latticeShuffle}
It follows from \cref{coro:latticePropertiesBipartiteGraphicalZonotope,prop:CartesianProductMinkowskiSumLatticeProperties} that the shuffle operation~$\shuffleDP$ preserves the interval property.
In contrast, \cref{rem:noCoTamari,rem:noBiTamari} show that neither the \mbox{$(3,3)$-constrainahe}\-dron ${\Constrainahedron[3][3] \eqdef \Asso[3] \shuffleDP \Asso[3]}$ nor the $(3,3)$-biassociahedron ${\Biassociahedron[3][3] \eqdef \Ossa[3] \shuffleDP \Asso[3]}$ have the lattice and congruence properties, while $\Ossa[3]$ and~$\Asso[3]$ both do.
However, we will see in \cref{coro:latticeShufflePermutahedron} that the shuffle with a permutahedron~$\Perm$ preserves the lattice property (but not the congruence property).
\end{remark}


\subsection{Shuffle with a point}
\label{subsec:shufflePoint}

We mark a little pause to specialize the observations of \cref{subsec:combinatorialDescription} to the case where~$\polytope{Q}$ is reduced to a point~$\b{0}$.
The bipreposets (and biposets) where the second poset is a singleton can then be encoded as painted preposets (and posets) defined below.
We first define antichains, upper sets and lower sets in preposets, generalizing the classical notions for posets.

\begin{definition}
\label{def:antichainsLowerUpperSetsPreposets}
Consider a preposet~$\preccurlyeq$ on~$[n]$.
An \defn{antichain} of~$\preccurlyeq$ is a subset~$A$ of~$[n]$ such that $i \in A \iff j \preccurlyeq i$ for any $i \preccurlyeq j$ with~$j \in A$.
An \defn{upper} (resp.~\defn{lower}) \defn{set} of~$\preccurlyeq$ is a subset~$U$ (resp.~$L$) of~$[n]$ such that~$i \in U$ implies $j \in U$ (resp.~$j \in L$ implies~$i \in L$) for any~$i \preccurlyeq j$.
In other words, an antichain (resp.~an upper set, resp.~a lower set) of a preposet~$\preccurlyeq$ is the union of the classes of an antichain (resp.~an upper set, resp.~a lower set) in the quotient poset~$\preccurlyeq/{\equiv}$ on the classes of the equivalence relation~$\equiv$ defined by~$i \equiv j \iff i \preccurlyeq j$ and~$j \preccurlyeq i$.
\end{definition}

\begin{definition}
\label{def:paintedPreposet}
A \defn{painted preposet} is a preposet~$\preccurlyeq$ on~$[n]$ together with a partition~$[n] = L \sqcup A \sqcup U$ where~$L$ is a lower set, $A$ is an antichain, and $U$ is an upper set (all possibly empty) of~$\preccurlyeq$.
\end{definition}

\begin{proposition}
\label{prop:paintedPreposets}
For any deformed permutahedron~$\polytope{P} \in \DefoPerms[n]$, the faces of the shuffle~$\polytope{P} \shuffleDP \b{0}$ are in bijection with the painted $\polytope{P}$-preposets.
\end{proposition}

\begin{proof}
Each face preposet~$\preccurlyeq_{\polytope{F}, \b{0}, \mu}$ of \cref{prop:shuffleFacePreposets} corresponds to a painted poset~$(\preccurlyeq_\polytope{F}, L \sqcup A \sqcup U)$ where~$L$ (resp.~$A$, resp.~$U$) is the subset of elements of~$[n]$ that appear in a part of~$\mu$ before (resp.~equal to, resp.~after) the part of~$\mu$ containing~$n+1$.
\end{proof}

\begin{definition}
\label{def:paintedPoset}
A \defn{painted poset} is a poset~$\preccurlyeq$ together with a partition~$[n] = L \sqcup U$ where~$L$ is a lower set and~$U$ is an upper set (both possibly empty) of~$\preccurlyeq$.
Two painted posets~$(\preccurlyeq, L \sqcup U)$ and~$(\preccurlyeq', L' \sqcup U')$ are connected by a \defn{right rotation} if
\begin{itemize}
\item either~$\preccurlyeq$ and~$\preccurlyeq$ are related by a right flip, while~$L = L'$ and~$U = U'$,
\item or~${\preccurlyeq} = {\preccurlyeq}$ and~$L = L' \cup \{i\}$ and~$U = U' \ssm \{i\}$ for some~$i \in [n]$.
\end{itemize}
\end{definition}

\begin{proposition}
\label{prop:paintedPoset}
The rotation graph of the shuffle~$\polytope{P} \shuffleDP \b{0}$ is isomorphic to the rotation graph on painted $\polytope{P}$-posets.
For any two painted $\polytope{P}$-posets~$\poset \eqdef (\preccurlyeq, L \sqcup U)$ and~$\poset' \eqdef (\preccurlyeq', L' \sqcup U')$, there is a path from~$\poset$ to~$\poset'$ in this graph if and only if~${\preccurlyeq} \le_\polytope{P} {\preccurlyeq'}$ and~$L \subseteq L'$.
\end{proposition}

\begin{proof}
This is a specialization of \cref{prop:shuffleRotationPosets} to~$\polytope{P} \shuffleDP \b{0}$.
\end{proof}

This description of the rotation graph enables us to show the following statement.

\begin{proposition}
\label{prop:latticeShufflePoint}
A deformed permutahedron~$\polytope{P}$ has the lattice property if and only if the shuffle~${\polytope{P} \shuffleDP \b{0}}$ has the lattice property.
\end{proposition}

\begin{proof}
Observe first that the rotation poset~$\le_\polytope{P}$ is isomorphic to the interval of the rotation poset~$\le_{\polytope{P} \shuffleDP \b{0}}$ given by the painted $\polytope{P}$-posets~$(\preccurlyeq, L \sqcup U)$ where~$L = \varnothing$.
This proves that~$\le_{\polytope{P} \shuffleDP \b{0}}$ is a lattice implies that~$\le_\polytope{P}$ is a lattice, since any interval of a lattice is a lattice.

Conversely, assume that~$\le_\polytope{P}$ is a lattice.
Consider $k$ painted $\polytope{P}$-posets~$\poset_1 \eqdef (\preccurlyeq_1, L_1 \sqcup U_1)$, \dots, $\poset_k \eqdef (\preccurlyeq_k, L_k \sqcup U_k)$.
Then it is immediate from \cref{prop:paintedPoset} that the join of~$\poset_1, \dots, \poset_k$ in~$\le_{\polytope{P} \shuffleDP \b{0}}$ is the painted $\polytope{P}$-poset~$\poset_\join \eqdef (\preccurlyeq_\join, L_\join \sqcup U_\join)$, where~$\preccurlyeq_\join$ is the join of the $\polytope{P}$-posets~$\preccurlyeq_1, \dots, \preccurlyeq_k$ in~$\le_\polytope{P}$, and~$L_\join$ is the lower set of~$\preccurlyeq_\join$ generated by the union~$L_1 \cup \dots \cup L_k$.
A symmetric expression obviously holds for the meet using~$U$ instead of~$L$.
\end{proof}

\begin{corollary}
\label{coro:latticeShufflePermutahedron}
If a deformed permutahedron~$\polytope{P}$ has the lattice property, then the shuffle~$\polytope{P} \shuffleDP \Perm$ has the lattice property for any integer~$n \ge 1$.
\end{corollary}


\subsection{Shuffle of graphical zonotopes}
\label{subsec:shuffleGraphicalZonotopes}

As it turns out, the family of graphical zonotopes is stable by the shuffle operation~$\shuffleDP$ on deformed permutahedra.
The corresponding operation on graphs is well-known in graph theory.

\begin{definition}
\label{def:joinGraphs}
The \defn{join} of two graphs~$G$ and~$H$ with disjoint vertex sets is the graph~$G \joinGraph H$ obtained by taking the disjoint union of~$G$ and~$H$ and connecting all vertices of~$G$ to all vertices of~$H$.
In other words, $V(G \joinGraph H) = V(G) \sqcup V(H)$ and $E(G \joinGraph H) = E(G) \sqcup E(H) \sqcup \big( V(G) \times V(H) \big)$.
If~$V(G) = [m]$ and~$V(H) = [n]$, we write~$G \joinGraph H$ for the graph~$G \joinGraph H^{+m} = (G \otimes H) \oplus K_{m,n}$.
\end{definition}

\begin{example}
\label{exm:joinGraphs}
The following families provide some relevant examples:
\begin{itemize}
\item the join of two empty graphs~$E_m$ and~$E_n$ is the complete bipartite graph~$K_{m,n}$ (more generally, the join of $k$ empty graphs~$E_{n_1}, \dots, E_{n_k}$ is the complete $k$-partite graph~$K_{n_1, \dots, n_k}$),
\item the join of a path~$P_m$ by an empty graph~$E_n$ is a fan graph~$F_{m,n}$,
\item the join of two complete graphs~$K_m$ and~$K_n$ is the complete graph~$K_{m+n}$.
\end{itemize}
\end{example}

The next statement immediately follows from \cref{def:shuffleDeformedPermutahedra,def:joinGraphs,prop:operationsGraphsZonotopes}.

\begin{proposition}
\label{prop:shuffleGraphicalZonotopes}
For all integer graphs~$G$ and~$H$, we have $\Zono[G] \shuffleDP \Zono[H] = \Zono[G \joinGraph H]$.
\end{proposition}

\begin{example}
\label{exm:shufflePermutahedra}
For instance, the permutahedra are stable by~$\shuffleDP$ since
\[
\Perm[m] \shuffleDP \Perm[n] = \Zono[K_m] \shuffleDP \Zono[K_n] = \Zono[K_m \joinGraph K_n] = \Zono[K_{m+n}] = \Perm[m+n].
\]
\end{example}

In view of \cref{prop:shuffleGraphicalZonotopes}, it was tempting to call~$\polytope{P} \shuffleDP \polytope{Q}$ the join of the deformed permutahedra~$\polytope{P}$ and~$\polytope{Q}$.
Recall however that there is a classical join operation on polytopes with the property that the graph of a join of polytopes is the join of the graphs of the polytopes (see \cite[Ex.\,9.9,~p.\,323]{Ziegler-polytopes}).

The number of vertices and facets of the graphical zonotopes arising from shuffles of graphical zonotopes are interesting.
See \cref{table:verticesPermCube,table:facetsPermCube,table:verticesPermPoint,table:facetsPermPoint,table:verticesPointPoint,table:facetsPointPoint} in \cref{subsec:tablesZonotopes} for tables of particularly relevant families.
We just mention here some relevant facts.

\pagebreak

\begin{proposition}
\label{prop:numberVerticesShuffleGraphicalZonotopes}
For all graphs~$G$ and~$H$, the number of vertices of~$\Zono[G] \shuffleDP \Zono[H]$ is the number of acyclic orientations of the join~$G \joinGraph H$.
In particular,
\begin{itemize}
\item $f_0 \big( \Perm[m] \shuffleDP \Para[n] \big) = (m+1)! \, (m+2)^{n-1}$,
\item $f_0 \big( \Perm[m] \shuffleDP \Point[n] \big) = m! \, (m+1)^n $,
\item $f_0 \big( \Point[m] \shuffleDP \Point[n] \big) = B(-m,n) \eqdef \sum_{\ell \ge 0} \surjections{m+1}{\ell+1} \, \surjections{n+1}{\ell+1} / (\ell+1)^2$, where~$\surjections{n}{k}$ denotes the number of surjections from~$[n]$ to~$[k]$ (see \OEIS{A019538} in~\cite{OEIS}),
\item $f_0 \big( \Point[n_1] \shuffleDP \cdots \shuffleDP \Point[n_k] \big) = \sum_{w \in W_k} \prod_{i \in [k]} \surjections{n_i}{|w|_i}$, where $W_k$ is the set of words on the alphabet~$[k]$ containing at least one copy of each letter and no consecutive identical letters, $|w|_i$ denotes the number of letters~$i$ in the word~$w$, and~$\surjections{n}{k}$ denotes the number of surjections from~$[n]$ to~$[k]$ (see \OEIS{A019538} in~\cite{OEIS}).
\end{itemize}
\end{proposition}

\begin{proof}
The first sentence of the statement is a direct consequence of \cref{prop:shuffleGraphicalZonotopes,prop:facesGraphicalZonotope}.
The numbers of vertices of~$\Perm[m] \shuffleDP \Para[n]$ and~$\Perm[m] \shuffleDP \Point[n]$ are easily computed by induction.
Finally, the numbers of vertices of~$\Point[m] \shuffleDP \Point[n]$ and~$\Point[n_1] \shuffleDP \cdots \shuffleDP \Point[n_k]$ follow from \cref{prop:completeMultipartiteGraphicalZonotopeNumberVertices}.
\end{proof}

\begin{proposition}
\label{prop:numberFacetsShuffleGraphicalZonotopes}
For all graphs~$G_1, \dots, G_k$ on~$[n_1], \dots, [n_k]$ respectively, such that $k > 2$ or at least one of $G_1$ and $G_2$ is connected, the number of facets of~$\Zono[G_1] \shuffleDP \cdots \shuffleDP \Zono[G_k]$ is given~by
\[
f_{n_1 + \dots + n_k - 2} \big( \Zono[G_1] \shuffleDP \cdots \shuffleDP \Zono[G_k] \big) = 2^{\sum_{i \in [k]} n_i} - 2 \sum_{i \in [k]} \nc{G_i} - 2,
\]
where~$\nc{G}$ denotes the number of disconnected subsets of~$G$.
For two disconnected graphs~$G_1$ and~$G_2$ on~$[n_1]$ and~$[n_2]$ respectively, the number of facets of~$\Zono[G_1] \shuffleDP \Zono[G_2]$ is
\[
{f_{n_1+n_2-2} \big( \Zono[G_1] \shuffleDP \Zono[G_2] \big) = 2^{n_1 + n_2} - 2 \, \nc{G_1} - 2 \, \nc{G_2}}.
\]
In particular,
\begin{itemize}
\item $f_{m+n-2} \big( \Perm[m] \shuffleDP \Para[n] \big) = 2^{m+n} - 2^{n+1} + n(n+1)$,
\item $f_{m+n-2} \big( \Perm[m] \shuffleDP \Point[n] \big) = 2^{m+n} - 2^{n+1} + 2n$,
\item $f_{m+n-2} \big( \Point[m] \shuffleDP \Point[n] \big) = 2^{m+n} - 2^{m+1} - 2^{n+1} + 2m + 2n + 4$,
\item $f_{m+n-2} \big( \Point[n_1] \shuffleDP \cdots \shuffleDP \Point[n_k] \big) = 2^{\sum_{i \in [k]} n_i} - 2 \sum_{i \in [k]} (2^{n_i} - n_i -2) - 2$ for~$k > 2$.
\end{itemize}
\end{proposition}

\begin{proof}
By \cref{prop:shuffleGraphicalZonotopes,prop:facesGraphicalZonotope}, the number of facets of~$\Zono[G_1] \shuffleDP \cdots \shuffleDP \Zono[G_k]$ is the number of biconnected subsets of~$G \eqdef G_1 \joinGraph \cdots \joinGraph G_k$.
Consider a subset~$U$ of the vertex set of~$G$.
If~$U$ meets the vertex sets of~$\ell > 1$ of the graphs~$G_i$, then the subgraph of~$G$ induced by~$U$ contains a complete $\ell$-partite graph and is thus connected.
Therefore, the subsets of vertices of~$G$ that are not biconnected are precisely the disconnected subsets of the graphs~$G_i$ and their complements.
When~$k > 2$ or at least one of~$G_1$ and~$G_2$ is connected, there is no ambiguity between these sets.
It follows that the number of biconnected subsets of~$G$ is ${2^{\sum_{i \in [k]} n_i} - 2 - 2\sum_{i \in [k]} \nc{G_i}}$.
If~$k = 2$ and both~$G_1$ and~$G_2$ are disconnected, we are counting~$G_1$ (resp.~$G_2$) both as a disconnected subset of~$G_1$ (resp.~$G_2$) and as the complement of~$G_2$ (resp.~$G_1$), which yield the correction ${2^{n_1 + n_2} - 2 \, \nc{G_1} - 2 \, \nc{G_2}}$.
The specific formulas then follow from the immediate observation that $\nc{K_n} = 0$ for the complete graph~$K_n$, $\nc{P_n} = 2^n - \binom{n+1}{2} - 1$ for the path graph~$P_n$, and~$\nc{E_n} = 2^n - n - 1$ for the empty graph~$E_n$.
\end{proof}

\begin{remark}
\label{rem:numberVerticesGraphicalZonotopes}
Note that~$\Perm[m] \shuffleDP \Para[n]$ and~$\Perm[m+1] \shuffleDP \Point[n-1]$ have the same number of vertices by \cref{prop:numberVerticesShuffleGraphicalZonotopes}, but not the same number of facets when~$n \ge 4$ by \cref{prop:numberFacetsShuffleGraphicalZonotopes}.
The equality for the number of vertices can be seen from \cref{coro:numberVertices}.
\end{remark}

\begin{remark}
Note that the results of this section extend to all hypergraphic polytopes.
A \defn{hypergraphic polytope} is the Minkowski sum of the faces of the standard simplex corresponding to the hyperedges of an arbitrary hypergraph.
See for instance~\cite{AguiarArdila, BenedettiBergeronMachacek}.
Hypergraphic polytopes contain in particular graphical zonotopes (when the hypergraph is a graph) and nestohedra (when the hypergraph is a building set~\cite{Postnikov, FeichtnerSturmfels}).
It immediately follows from \cref{def:shuffleDeformedPermutahedra} that the shuffle of two hypergraphic polytopes is a hypergraphic polytope (and the hypergraph of the shuffle is the join of the hypergraphs of the factors in the sense of \cref{def:joinGraphs}).
\end{remark}


\section{Multiplihedra}
\label{sec:multiplihedra}

In this section, we study the family of $(m,n)$-multiplihedra, obtained as the shuffle of an $m$-permutahedron~$\Perm[m]$ with an $n$-associahedron~$\Asso[n]$.
It extends the classical multiplihedron studied in~\cite{Stasheff-HSpaces, SaneblidzeUmble-diagonals, Forcey-multiplihedra, ForceyLauveSottile, MauWoodward, ArdilaDoker}, which corresponds to the case~$m = 1$.
We generalize the classical model of painted trees to $(m,n)$-multiplihedron (\cref{subsec:paintedTrees}), describe the face lattice, fan and oriented skeleton of the $(m,n)$-multiplihedron in terms of these trees (\cref{subsec:multiplihedra}), provide explicit vertex, facet and Minkowski sum descriptions of the $(m,n)$-multiplihedron (\cref{subsec:vertexFacetMinkowskiDescriptionsMultiplihedra}), and present enumerative results on the number of vertices, faces and facets of the $(m,n)$-multiplihedron (\cref{subsec:numerologyMultiplihedra}).
One relevant byproduct of this section is a lattice structure on binary $m$-painted $n$-trees, containing simultaneously the weak order on permutations of~$\f{S}_m$ and the Tamari lattice on binary trees of~$\f{B}_n$.
We are not aware that this lattice structure was noticed in the literature, even for the classical painted trees (with~$m = 1$).


\subsection{Painted trees}
\label{subsec:paintedTrees}

We start by defining $m$-painted $n$-trees, see \cref{fig:paintedTrees}.
Intuitively, an \mbox{$m$-painted} $n$-tree is just a Schr\"oder $n$-tree with some disjoint cuts that can pass through vertices or through edges and are labeled by a partition of~$[m]$.
To make our definitions precise, it is convenient to introduce unary nodes when a cut passes through an edge.
Recall from \cref{def:tree,def:inorder,def:treeDeletion} our conventions for rooted plane trees, inorder labelings, and node deletions.

\begin{definition}
\label{def:stumpCut}
For a tree~$T$, we call
\begin{itemize}
\item \defn{cut} of~$T$ a subset~$C$ of nodes of~$T$ containing precisely one node along the path from the root to any leaf of~$T$,
\item \defn{stump} of~$T$ a subset~$S$ of nodes of~$T$ containing the root of~$T$ and such that the parent of a node in~$S$ also belongs to~$S$, and conversely either none or all children of a node in~$S$ also belong to~$S$.
\end{itemize}
Clearly, to a cut~$C$ of~$T$ corresponds the stump~$\stump{C}$ of all nodes located along a path from the root of~$T$ to a node of~$C$.
Conversely, to a stump~$S$ of~$T$ corresponds the cut~$\cut{S}$ of nodes of~$S$ with no child in~$S$.
\end{definition}

\begin{definition}
\label{def:paintedTree}
A \defn{$m$-painted $n$-tree}~$\PT \eqdef (T, C, \mu)$ consists of an $n$-tree~$T$, a sequence~$C \eqdef (C_1, \cdots, C_k)$ of $k \le m$ cuts of~$T$, and an ordered partition~$\mu$ of~$[m]$ into~$k$ parts, such~that 
\begin{itemize}
\item $C_{i+1} \subseteq \stump{C_i} \ssm C_i$ for all~$i \in [k-1]$, and
\item any unary node of~$T$ belongs to one of the cuts~$C_1, \dots, C_k$.
\end{itemize}
We denote by~$\f{PT}_{m,n}$ the set of $m$-painted $n$-trees.
\end{definition}

In other words, an $m$-painted $n$-tree is an $n$-tree with at most~$m$ cuts, where each cut is disjoint and below the previous one, the union of the cuts covers all unary nodes, and the cuts are labeled by an ordered partition of~$[m]$.
In the sequel, we write~$|C|$ for~$k$ and~$|\bigcup C|$ for~$|\bigcup_{i \in [k]} C_i|$.
To represent an $m$-painted $n$-tree~$\PT \eqdef (T, C, \mu)$, we draw the tree~$T$ in such a way that all nodes in the cut~$C_i$ belong to the same (red) horizontal line labeled by~$\mu_i$ (which is abbreviated as a word rather than a set).
Examples are illustrated in \cref{fig:paintedTrees}.
Note that when~$k = 1$, the $1$-painted $n$-trees are precisely the painted posets of \cref{def:paintedPreposet,def:paintedPoset} for the associahedron~$\Asso$, since it is equivalent to remember the cut and to remember which vertices are below, on, or above the cut.

\begin{figure}[b]
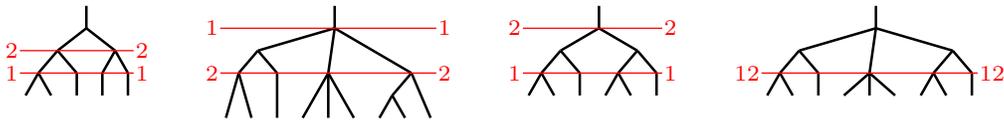

	\centerline{
		\paintedTree[1.2]{[[2[1[][]][1[]]][2[1[]][1[]]]]}{1,2}
		\hspace{-.5cm}
		\paintedTree[1.2]{[1[[2[][]][2[]]][2[][][]][2[[][]][]]]}{1,2}
		\hspace{-.5cm}
		\paintedTree[1.2]{[2[[1[][]][1[]]][[1[][]][1[]]]]}{1,2}
		\hspace{-.5cm}
		\paintedTree[1.2]{[[[12[][]][12[]]][12[][][]][[12[][]][12[]]]]}{12}
	}
	\caption{Some $2$-painted trees.}
	\label{fig:paintedTrees}
\end{figure}

We now define the painted tree deletion poset.
\cref{def:paintedTreeDeletion} provides a direct description in terms of painted trees, while \cref{def:paintedTreePreposet} provides an alternative simpler but indirect description in terms of preposets.
To illustrate the following definition, \cref{fig:paintedTreeDeletions} represents a sequence of deletions in $2$-painted $5$-trees.

\begin{figure}
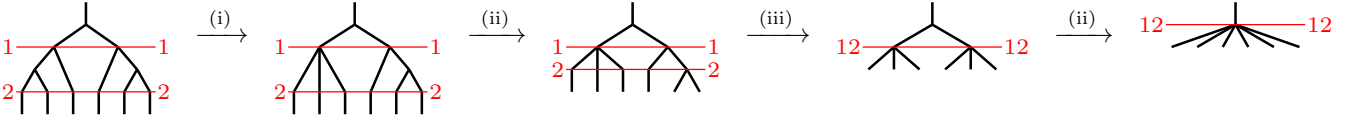

	\centerline{
		\paintedTree[1.2]{[[1[[2[]][2[]]][2[]]][1[2[]][[2[]][2[]]]]]}{1,2}
		\hspace{-.3cm}\raisebox{-.7cm}{$\xrightarrow{\text{ (i) }}$}\hspace{-.3cm}
		\paintedTree[1.2]{[[1[2[]][2[]][2[]]][1[2[]][[2[]][2[]]]]]}{1,2}
		\hspace{-.3cm}\raisebox{-.7cm}{$\xrightarrow{\text{ (ii) }}$}\hspace{-.3cm}
		\paintedTree[1.2]{[[1[2[]][2[]][2[]]][1[2[]][2[][]]]]}{1,2}
		\hspace{-.3cm}\raisebox{-.7cm}{$\xrightarrow{\text{ (iii) }}$}\hspace{-.3cm}
		\paintedTree[1.2]{[[12[][][]][12[][][]]]}{12}
		\hspace{-.3cm}\raisebox{-.7cm}{$\xrightarrow{\text{ (ii) }}$}\hspace{-.3cm}
		\paintedTree[1.2]{[12[][][][][][]]}{12}
	}
	\caption{Deletions in $2$-painted $5$-trees.}
	\label{fig:paintedTreeDeletions}
\end{figure}

\begin{definition}
\label{def:paintedTreeDeletion}
Let~$\PT \eqdef (T, C, \mu)$ and $\PT' \eqdef (T', C', \mu')$ be two $m$-painted $n$-trees.
We say that~$\PT'$ is obtained by a \defn{deletion} in~$\PT$ in either of the following three cases:
\begin{enumerate}[(i)]
\item \textbf{Free child:} A node~$\node{n}$ of~$T$ distinct from the root is in none of the cuts of~$C$, and $T'$ is obtained by deleting~$\node{n}$ in~$T$, while $C' = C$ and~$\mu' = \mu$.
\item \textbf{Free parent:} A node~$\node{p}$ is in none of the cuts of~$C$ while all its children~$\children(\node{p})$ belong to the same cut~$C_i$, and $T'$ is obtained by deleting all~$\children(\node{p})$ in~$T$, $C'$ is obtained from~$C$ by replacing~$C_i$ by~$C'_i \eqdef \big( C_i \ssm \children(\node{p}) \big) \cup \{\node{p}\}$, and~$\mu' = \mu$.
\item \textbf{Twin cuts:} There is~$i$ such that a node belongs to~$C_{i+1}$ if and only if its children belong to~$C_i$, and~$T'$ is obtained by deleting simultaneously all nodes in~$C_i$, $C'$ is obtained from~$C$ by deleting~$C_i$, and~$\mu'$ is obtained from~$\mu$ by merging~$\mu_i$ and~$\mu_{i+1}$.
\end{enumerate}
\end{definition}

\begin{proposition}
\label{prop:paintedTreeDeletion}
For all integers~$m, n \ge 0$, the set~$\f{PT}_{m,n}$ is stable by deletion, and the deletion graph is the Hasse diagram of a poset ranked by~$\rank(T, C, \mu) = m + n - |T| - |C| + |\bigcup C|$.
In particular an $m$-painted $n$-tree~$\PT \eqdef (T, C, \mu)$ has
\begin{itemize}
\item rank~$0$ if and only~$|C| = m$, and all nodes in $\bigcup C$ (resp.~not in~$\bigcup C$) have degree $1$ (resp.~$2$),
\item rank~$m+n-2$ if and only if either~$|C| = 1$ and all but one node are contained in~$C_1$, or~$|C| = 2$ and all nodes are contained in~$C_1 \cup C_2$,
\item rank~$m+n-1$ if and only if it has a single node.
\end{itemize}
\end{proposition}

\begin{proof}
Consider a deletion transforming~$\PT \eqdef (T, C, \mu)$ to~$\PT' \eqdef (T', C', \mu')$.
Then~$\PT'$ is clearly an $m$-painted $n$-tree since the cuts of~$C'$ are still disjoint and one below the other, and~$|C'| = |\mu'|$.
For the rank, we distinguish three cases corresponding to that of \cref{def:paintedTreeDeletion}:
\begin{enumerate}[(i)]
\item \textbf{Free child:} $|T'| = |T|-1$ while~$C' = C$ so that~$|C'| = |C|$ and~$|\bigcup C'| = |\bigcup C|$.
\item \textbf{Free parent:} $|T'| = |T| - |\children(\node{p})|$, $|C'| = |C|$ and~$|\bigcup C'| = |\bigcup C| - |\children(\node{p})| + 1$.
\item \textbf{Twin cuts:} $|T'| = |T| - |C_i|$, $|C'| = |C|-1$ and~$|\bigcup C'| = |\bigcup C| - |C_i|$.
\end{enumerate}
In all three situations, we get~$\rank(\PT') = \rank(\PT)+1$.
The end of the statement immediately follows.
\end{proof}

\begin{definition}
\label{def:paintedTreeDeletionPoset}
The \defn{$m$-painted $n$-tree deletion poset} is the poset on $\f{PT}_{m,n}$ where an $m$-painted $n$-tree is covered by all $m$-painted $n$-trees that can be obtained by a deletion.
\end{definition}

The $m$-painted $n$-tree deletion poset can alternatively be defined using preposets.

\begin{definition}
\label{def:paintedTreePreposet}
A $m$-painted $n$-tree~$\PT \eqdef (T, C, \mu)$ defines a preposet~$\preccurlyeq_{\PT}$ on~$[m+n]$ that can be read as follows.
Label each node~$\node{n}$ of~$T$ by the union of the part~$\mu_i$ if~$\node{n} \in C_i$ (empty set if~$\node{n} \notin \bigcup C$) and the inorder label of~$\node{n}$ in~$T$ shifted by~$m$ (empty set if~$\node{n}$ has degree~$1$).
Then merge together all nodes contained in each cut.
Then, for any~$i,j \in [m+n]$, we have~$i \preccurlyeq_{\PT} j$ if there is a (possibly empty) oriented path from the node containing~$i$ to the node containing~$j$ in the resulting oriented graph.
\end{definition}

\begin{proposition}
\label{prop:characterizationPaintedTreePreposets}
The preposets~$\preccurlyeq_{\PT}$ for~$\PT \in \f{PT}_{m,n}$ are precisely the preposets~$\preccurlyeq$ on~$[m+n]$ in which any~$1 \le i < k \le m+n$ are comparable (\ie~$i \preccurlyeq k$ or $i \succcurlyeq k$ or both) if and only if
\begin{itemize}
\item either~$i \le m$,
\item or~$m < i$ and at least one of the following holds:
	\begin{itemize}
	\item there exists no~$i < j < k$ such that $i \prec j \succ k$,
	\item there exists~$j \in [m]$ such that $i \preccurlyeq j \preccurlyeq k$ or $i \succcurlyeq j \succcurlyeq k$.
	\end{itemize}
\end{itemize}
\end{proposition}

\begin{proof}
Any preposet~$\preccurlyeq_{\PT}$ clearly satisfies these conditions.
Conversely, given a preposet~$\preccurlyeq$ on~${[m+n]}$ satisfying these conditions, consider the preposet~$\preccurlyeq'$ on~$[n]$ defined by~$i \preccurlyeq' k$ if and only if ${i + m \preccurlyeq k+m}$ and there is no~$i < j < k$ such that~$i + m \prec j + m \succ k + m$.
The preposet~$\preccurlyeq'$ is clearly the preposet~$\preccurlyeq_T$ of a Schr\"oder $n$-tree~$T$.
We then obtain the cuts~$C$ and the partition~$\mu$ by considering the relations~$i \preccurlyeq k$ with~$i \le m < k$.
Details are left to the reader.
\end{proof}

\begin{proposition}
\label{prop:paintedTreeDeletionPosetOnPreposets}
In the $m$-painted $n$-tree deletion poset, $\PT$ is smaller than~$\PT'$ if and only if~$\preccurlyeq_{\PT}$ refines~$\preccurlyeq_{\PT'}$.
\end{proposition}

\begin{proof}
An immediate case analysis shows that deletions in a painted tree~$\PT$ defined in \cref{def:paintedTreeDeletion} precisely translate all possible refinements in the corresponding preposet~$\preccurlyeq_{\PT}$.
\end{proof}

Finally, we define the rotations in painted trees, which correspond to rank~$1$ painted trees.
To illustrate the following definition, \cref{fig:paintedTreeRotations} represents a sequence of right rotations in binary $2$-painted $3$-trees.

\begin{figure}
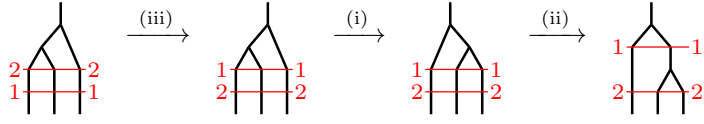

	\centerline{
		\paintedTree[1.2]{[[[2[1[]]][2[1[]]]][2[1[]]]]}{1,2}
		\hspace{-.3cm}\raisebox{-.7cm}{$\xrightarrow{\text{ (iii) }}$}\hspace{-.3cm}
		\paintedTree[1.2]{[[[1[2[]]][1[2[]]]][1[2[]]]]}{1,2}
		\hspace{-.3cm}\raisebox{-.7cm}{$\xrightarrow{\text{ (i) }}$}\hspace{-.3cm}
		\paintedTree[1.2]{[[1[2[]]][[1[2[]]][1[2[]]]]]}{1,2}
		\hspace{-.3cm}\raisebox{-.7cm}{$\xrightarrow{\text{ (ii) }}$}\hspace{-.3cm}
		\paintedTree[1.2]{[[1[2[]]][1[[2[]][2[]]]]]}{1,2}
	}
	\caption{Right rotations in binary $2$-painted $3$-trees.}
	\label{fig:paintedTreeRotations}
\end{figure}

\begin{definition}
\label{def:rotationsPaintedTrees}
We call \defn{binary $m$-painted $n$-trees} the rank~$0$ $m$-painted $n$-trees, \ie where all nodes in $\bigcup C$ have degree $1$ while all nodes not in~$\bigcup C$ have degree~$2$.
We say that two binary $m$-painted $n$-trees~$\PT \eqdef (T, C, \mu)$ and $\PT' \eqdef (T', C', \mu')$ are connected by a \defn{right rotation} if:
\begin{enumerate}[(i)]
\item \textbf{Edge rotation:} $T'$ is obtained from~$T$ by the right rotation of an edge connecting two binary nodes, $C' = C$ and~$\mu' = \mu$,
\item \textbf{Node--cut sweep:} $T'$ is obtained from~$T$ by replacing a binary node~$\node{n}_1$ with two unary children~$\node{n}_2, \node{n}_3$ by a unary node~$\node{n}'_1$ with a binary child~$\node{n}'_2$, $C'$ is obtained by replacing~$\node{n}_2$ and~$\node{n}_3$ by~$\node{n}'_1$, and~$\mu' = \mu$,
\item \textbf{Twin cuts:} There is~$i$ such that~$\mu_i < \mu_{i+1}$ and a node belongs to~$C_i$ if and only if its children belong to~$C_{i+1}$, and $T' = T$, $C' = C$ and~$\mu'$ is obtained from~$\mu$ by exchanging the values~$\mu_i$ and~$\mu_{i+1}$.
\end{enumerate}
\end{definition}

\begin{remark}
\label{rem:algebraicInterpretationMultiplihedra}
The binary $m$-painted $n$-trees can be interpreted algebraically as follows.
We consider a non-associative magma~$(X, \ast)$ and $m$ functions $f_1, \dots, f_m$ from $X$ to $X$ which are not magma homomorphisms. We then consider the terms than can be produced by starting from a sequence of~$n$ elements of~$X$ and iteratively applying either $\ast$ to two consecutive terms in the sequence or one function~$f_i$ (each one used exactly once) to all terms in the sequence.
For instance, the terms corresponding to the $4$ trees of \cref{fig:paintedTreeRotations} are
\begin{gather*}
\big( f_2 \circ f_1(x) \ast f_2 \circ f_1(y) \big) \ast f_2 \circ f_1(z), \\
\big( f_1 \circ f_2(x) \ast f_1 \circ f_2(y) \big) \ast f_1 \circ f_2(z), \\
f_1 \circ f_2(x) \ast \big( f_1 \circ f_2(y) \ast f_1 \circ f_2(z) \big), \\
f_1 \circ f_2(x) \ast f_1 \big( f_2(y) \ast f_2(z) \big).
\end{gather*}
\end{remark}


\subsection{Permutahedra $\shuffleDP$ Associahedra}
\label{subsec:multiplihedra}

We now consider shuffles of permutahedra with associahedra.

\begin{definition}
\label{def:multiplihedron}
The \defn{$(m,n)$-multiplihedron} is the polytope~$\Multiplihedron = \Perm[m] \shuffleDP \Asso[n]$.
\end{definition}

\begin{remark}
\label{rem:multipliheadra}
When~$n = 1$, we obtain the multiplihedron studied in~\cite{Stasheff-HSpaces, SaneblidzeUmble-diagonals, Forcey-multiplihedra, ForceyLauveSottile, MauWoodward, ArdilaDoker}.
Our geometric realization is different from that of~\cite{Forcey-multiplihedra}.
For instance, the two facets corresponding to associahedra are translated copies in our realizations of the $(1,n)$-multiplihedron, while they are dilated copies in the realization of~\cite{Forcey-multiplihedra}.
\end{remark}

This family of polytopes is illustrated in \cref{fig:multiplihedra2,fig:multiplihedra3,fig:multiplihedra4}.
We have labeled with $m$-painted $n$-trees all faces in \cref{fig:multiplihedra2}, and all vertices in \cref{fig:multiplihedra3} (see also \cref{fig:multiplihedron23}).
We let the reader complete the pictures in \cref{fig:multiplihedra3,fig:multiplihedra4}.

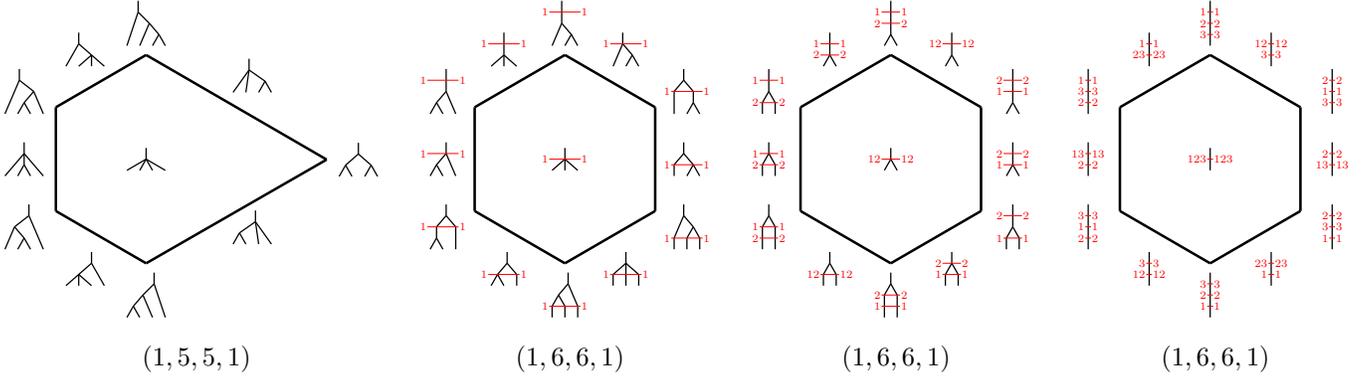
\begin{figure}
	\centerline{
	\begin{tabular}{c@{\hspace{-.7cm}}c@{\hspace{-.7cm}}c@{\hspace{-.7cm}}c}
		\scalebox{1.2}{\begin{tikzpicture}[align=center]

  \coordinate (v0) at (0, 0);
  \coordinate (v1) at (-1, .58);
  \coordinate (v2) at (2, 1.15);
  \coordinate (v3) at (-1, 1.73);
  \coordinate (v4) at (0, 2.31);

  \draw[thick] (v0) -- (v1) coordinate[pos=0.5] (e0);
  \draw[thick] (v0) -- (v2) coordinate[pos=0.5] (e1);
  \draw[thick] (v1) -- (v3) coordinate[pos=0.5] (e2);
  \draw[thick] (v2) -- (v4) coordinate[pos=0.5] (e3);
  \draw[thick] (v3) -- (v4) coordinate[pos=0.5] (e4);

  \node[xshift=0pt, yshift=-10pt] at (v0) {\paintedTree[.5]{[[[[][]][]][]]}{}};
  \node[xshift=-10pt, yshift=-5pt] at (v1) {\paintedTree[.5]{[[[][[][]]][]]}{}};
  \node[xshift=10pt, yshift=0pt] at (v2) {\paintedTree[.5]{[[[][]][[][]]]]}{}};
  \node[xshift=-10pt, yshift=5pt] at (v3) {\paintedTree[.5]{[[][[[][]][]]]}{}};
  \node[xshift=0pt, yshift=10pt] at (v4) {\paintedTree[.5]{[[][[][[][]]]]}{}};
  
  \node[xshift=-5pt, yshift=-10pt] at (e0) {\paintedTree[.5]{[[[][][]][]]]}{}};
  \node[xshift=5pt, yshift=-5pt] at (e1) {\paintedTree[.5]{[[[][]][][]]}{}};
  \node[xshift=-10pt, yshift=0pt] at (e2) {\paintedTree[.5]{[[][[][]][]]}{}};
  \node[xshift=5pt, yshift=10pt] at (e3) {\paintedTree[.5]{[[][][[][]]]}{}};
  \node[xshift=-5pt, yshift=10pt] at (e4) {\paintedTree[.5]{[[][[][][]]]}{}};
  
  \node[xshift=0pt, yshift=0pt] at ($0.5*(v0)+0.5*(v4)$) {\paintedTree[.5]{[[][][][]]}{}};

\end{tikzpicture}} &
		\scalebox{1.2}{\begin{tikzpicture}[align=center]

  \coordinate (v0) at (0, 0);
  \coordinate (v1) at (-1, .58);
  \coordinate (v2) at (1, .58);
  \coordinate (v3) at (-1, 1.73);
  \coordinate (v4) at (1, 1.73);
  \coordinate (v5) at (0, 2.31);

  \draw[thick] (v0) -- (v1) coordinate[pos=0.5] (e0);
  \draw[thick] (v0) -- (v2) coordinate[pos=0.5] (e1);
  \draw[thick] (v1) -- (v3) coordinate[pos=0.5] (e2);
  \draw[thick] (v2) -- (v4) coordinate[pos=0.5] (e3);
  \draw[thick] (v3) -- (v5) coordinate[pos=0.5] (e4);
  \draw[thick] (v4) -- (v5) coordinate[pos=0.5] (e5);

  \node[xshift=0pt, yshift=-10pt] at (v0) {\paintedTree[.5]{[[[1[]][1[]]][1[]]]}{1}};
  \node[xshift=-10pt, yshift=-5pt] at (v1) {\paintedTree[.5]{[[1[[][]]][1[]]]}{1}};
  \node[xshift=10pt, yshift=-5pt] at (v2) {\paintedTree[.5]{[[1[]][[1[]][1[]]]]}{1}};
  \node[xshift=-10pt, yshift=5pt] at (v3) {\paintedTree[.5]{[1[[[][]][]]]}{1}};
  \node[xshift=10pt, yshift=5pt] at (v4) {\paintedTree[.5]{[[1[]][1[[][]]]]}{1}};
  \node[xshift=0pt, yshift=10pt] at (v5) {\paintedTree[.5]{[1[[][[][]]]]}{1}};

  \node[xshift=-5pt, yshift=-10pt] at (e0) {\paintedTree[.5]{[[1[][]][1[]]]}{1}};
  \node[xshift=5pt, yshift=-10pt] at (e1) {\paintedTree[.5]{[[1[]][1[]][1[]]]}{1}};
  \node[xshift=-10pt, yshift=0pt] at (e2) {\paintedTree[.5]{[1[[][]][]]}{1}};
  \node[xshift=10pt, yshift=0pt] at (e3) {\paintedTree[.5]{[[1[]][1[][]]]}{1}};
  \node[xshift=-5pt, yshift=10pt] at (e4) {\paintedTree[.5]{[1[[][][]]]}{1}};
  \node[xshift=5pt, yshift=10pt] at (e5) {\paintedTree[.5]{[1[][[][]]]}{1}};
  
  \node[xshift=0pt, yshift=0pt] at ($0.5*(v0)+0.5*(v5)$) {\paintedTree[.5]{[1[][][]]}{1}};

\end{tikzpicture}} &
		\scalebox{1.2}{\begin{tikzpicture}[align=center]

  \coordinate (v0) at (0, 0);
  \coordinate (v1) at (-1, .58);
  \coordinate (v2) at (1, .58);
  \coordinate (v3) at (-1, 1.73);
  \coordinate (v4) at (1, 1.73);
  \coordinate (v5) at (0, 2.31);

  \draw[thick] (v0) -- (v1) coordinate[pos=0.5] (e0);
  \draw[thick] (v0) -- (v2) coordinate[pos=0.5] (e1);
  \draw[thick] (v1) -- (v3) coordinate[pos=0.5] (e2);
  \draw[thick] (v2) -- (v4) coordinate[pos=0.5] (e3);
  \draw[thick] (v3) -- (v5) coordinate[pos=0.5] (e4);
  \draw[thick] (v4) -- (v5) coordinate[pos=0.5] (e5);

  \node[xshift=0pt, yshift=-10pt] at (v0) {\paintedTree[.5]{[[2[1[]]][2[1[]]]]}{1,2}};
  \node[xshift=-10pt, yshift=-5pt] at (v1) {\paintedTree[.5]{[[1[2[]]][1[2[]]]]}{1,2}};
  \node[xshift=10pt, yshift=-5pt] at (v2) {\paintedTree[.5]{[2[[1[]][1[]]]]}{1,2}};
  \node[xshift=-10pt, yshift=5pt] at (v3) {\paintedTree[.5]{[1[[2[]][2[]]]]}{1,2}};
  \node[xshift=10pt, yshift=5pt] at (v4) {\paintedTree[.5]{[2[1[[][]]]]}{1,2}};
  \node[xshift=0pt, yshift=10pt] at (v5) {\paintedTree[.5]{[1[2[[][]]]]}{1,2}};

  \node[xshift=-5pt, yshift=-10pt] at (e0) {\paintedTree[.5]{[[12[]][12[]]]}{12}};
  \node[xshift=5pt, yshift=-10pt] at (e1) {\paintedTree[.5]{[2[1[]][1[]]]}{1,2}};
  \node[xshift=-10pt, yshift=0pt] at (e2) {\paintedTree[.5]{[1[2[]][2[]]]}{1,2}};
  \node[xshift=10pt, yshift=0pt] at (e3) {\paintedTree[.5]{[2[1[][]]]}{1,2}};
  \node[xshift=-5pt, yshift=10pt] at (e4) {\paintedTree[.5]{[1[2[][]]]}{1,2}};
  \node[xshift=5pt, yshift=10pt] at (e5) {\paintedTree[.5]{[12[[][]]]}{12}};
  
  \node[xshift=0pt, yshift=0pt] at ($0.5*(v0)+0.5*(v5)$) {\paintedTree[.5]{[12[][]]}{12}};

\end{tikzpicture}} &
		\scalebox{1.2}{\begin{tikzpicture}[align=center]

  \coordinate (v0) at (0, 0);
  \coordinate (v1) at (-1, .58);
  \coordinate (v2) at (1, .58);
  \coordinate (v3) at (-1, 1.73);
  \coordinate (v4) at (1, 1.73);
  \coordinate (v5) at (0, 2.31);

  \draw[thick] (v0) -- (v1) coordinate[pos=0.5] (e0);
  \draw[thick] (v0) -- (v2) coordinate[pos=0.5] (e1);
  \draw[thick] (v1) -- (v3) coordinate[pos=0.5] (e2);
  \draw[thick] (v2) -- (v4) coordinate[pos=0.5] (e3);
  \draw[thick] (v3) -- (v5) coordinate[pos=0.5] (e4);
  \draw[thick] (v4) -- (v5) coordinate[pos=0.5] (e5);

  \node[xshift=0pt, yshift=-10pt] at (v0) {\paintedTree[.5]{[3[2[1[]]]]}{1,2,3}};
  \node[xshift=-10pt, yshift=-5pt] at (v1) {\paintedTree[.5]{[3[1[2[]]]]}{1,2,3}};
  \node[xshift=10pt, yshift=-5pt] at (v2) {\paintedTree[.5]{[2[3[1[]]]]}{1,2,3}};
  \node[xshift=-10pt, yshift=5pt] at (v3) {\paintedTree[.5]{[1[3[2[]]]]}{1,2,3}};
  \node[xshift=10pt, yshift=5pt] at (v4) {\paintedTree[.5]{[2[1[3[]]]]}{1,2,3}};
  \node[xshift=0pt, yshift=10pt] at (v5) {\paintedTree[.5]{[1[2[3[]]]]}{1,2,3}};

  \node[xshift=-5pt, yshift=-10pt] at (e0) {\paintedTree[.5]{[3[12[]]]}{12,3}};
  \node[xshift=5pt, yshift=-10pt] at (e1) {\paintedTree[.5]{[23[1[]]]}{1,23}};
  \node[xshift=-10pt, yshift=0pt] at (e2) {\paintedTree[.5]{[13[2[]]]}{13,2}};
  \node[xshift=10pt, yshift=0pt] at (e3) {\paintedTree[.5]{[2[13[]]]}{13,2}};
  \node[xshift=-5pt, yshift=10pt] at (e4) {\paintedTree[.5]{[1[23[]]]}{1,23}};
  \node[xshift=5pt, yshift=10pt] at (e5) {\paintedTree[.5]{[12[3[]]]}{12,3}};
  
  \node[xshift=0pt, yshift=0pt] at ($0.5*(v0)+0.5*(v5)$) {\paintedTree[.5]{[123[]]}{123}};

\end{tikzpicture}} \\
		$(1, 5, 5, 1)$ &
		$(1, 6, 6,1)$ &
		$(1, 6, 6,1)$ &
		$(1, 6, 6,1)$
	\end{tabular}
	}
	\caption{The $(m,n)$-multiplihedra $\Multiplihedron$ and their $f$-vectors for~$(m,n) = (0,3)$, $(1,2)$, $(2,1)$ and~$(3,0)$. The leftmost is the $2$-dimensional associahedron~$\Asso[3]$ while the other three are all relabelings of the $2$-dimensional permutahedron~$\Perm[3]$.}
	\label{fig:multiplihedra2}
\end{figure}

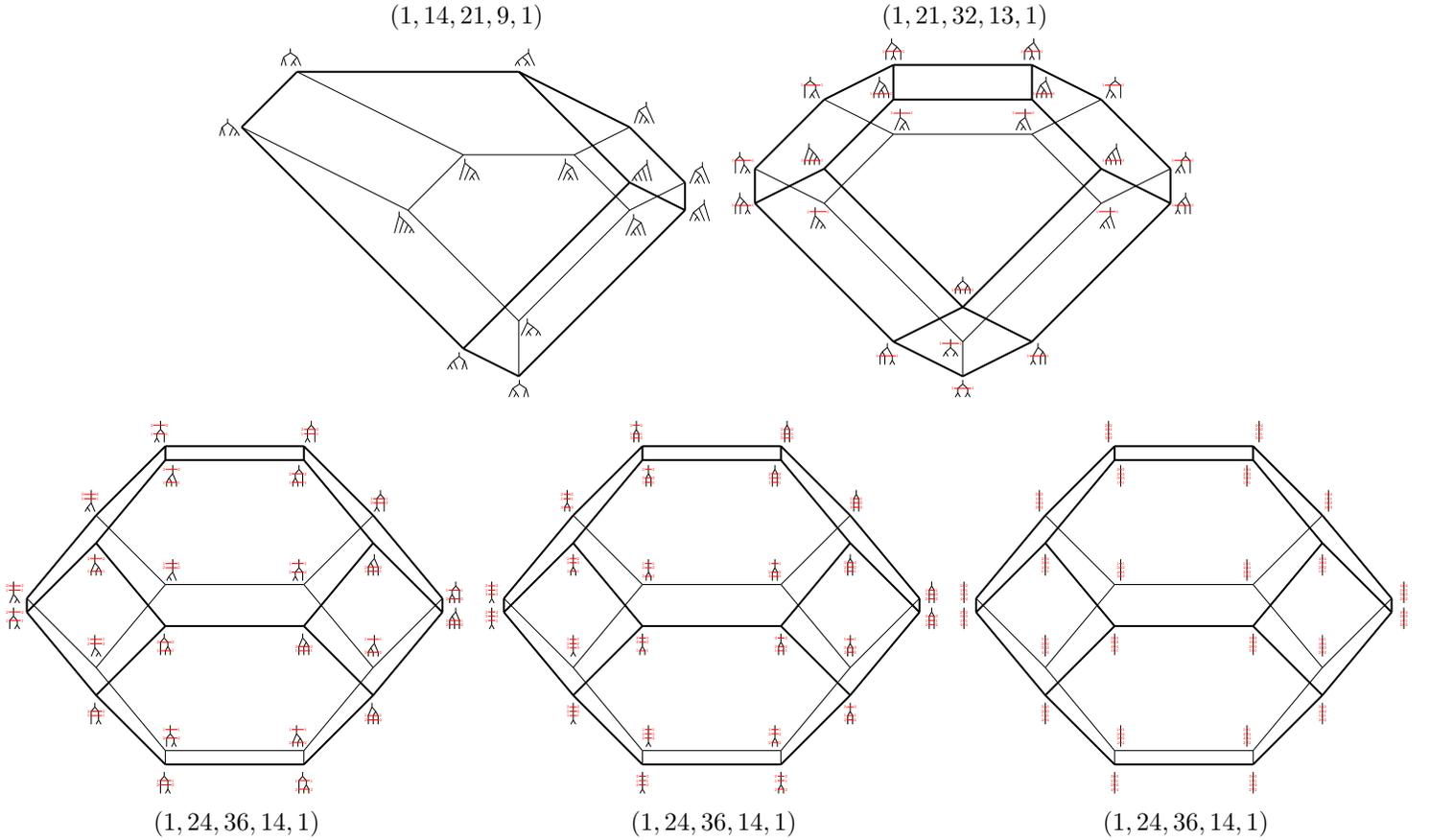
\begin{figure}
	\centerline{
	\begin{tabular}{c@{\hspace{-.2cm}}c}
		$(1, 14, 21, 9, 1)$ &
		$(1, 21, 32, 13, 1)$ \\
		\scalebox{.48}{\begin{tikzpicture}[x  = {(1cm,-.5cm)},
                    y  = {(-1cm,1cm)},
                    z  = {(1cm,1cm)},
                    scale = 1.6,
                    align=center]

  \coordinate (v66) at (0, -1, 1);
  \coordinate (v67) at (0, -1, -2);
  \coordinate (v68) at (0, 1, 1);
  \coordinate (v69) at (0, 3, -1);
  \coordinate (v70) at (0, 3, -2);
  \coordinate (v71) at (1, -1, 1);
  \coordinate (v72) at (1, -1, -2);
  \coordinate (v73) at (2, 0, 1);
  \coordinate (v74) at (2, 1, 1);
  \coordinate (v75) at (3, 1, 0);
  \coordinate (v76) at (3, 1, -2);
  \coordinate (v77) at (3, 2, 0);
  \coordinate (v78) at (3, 3, -1);
  \coordinate (v79) at (3, 3, -2);

  \tikzstyle{back} = [thick]
  \tikzstyle{front} = [ultra thick]

  \draw[back] (v75) -- (v73);
  \draw[back] (v76) -- (v72);
  \draw[back] (v76) -- (v75);
  \draw[back] (v77) -- (v74);
  \draw[back] (v77) -- (v75);
  \draw[back] (v78) -- (v69);
  \draw[back] (v78) -- (v77);
  \draw[back] (v79) -- (v70);
  \draw[back] (v79) -- (v76);
  \draw[back] (v79) -- (v78);

  \draw[front] (v67) -- (v66);
  \draw[front] (v68) -- (v66);
  \draw[front] (v69) -- (v68);
  \draw[front] (v70) -- (v67);
  \draw[front] (v70) -- (v69);
  \draw[front] (v71) -- (v66);
  \draw[front] (v72) -- (v67);
  \draw[front] (v72) -- (v71);
  \draw[front] (v73) -- (v71);
  \draw[front] (v74) -- (v68);
  \draw[front] (v74) -- (v73);

  \node[xshift=10pt, yshift=10pt] at (v66) {\paintedTree[.5]{[[[[[][]][]][]][]]}{}};
  \node[xshift=-5pt, yshift=-10pt] at (v67) {\paintedTree[.5]{[[[[][]][]][[][]]]}{}};
  \node[xshift=5pt, yshift=12pt] at (v68) {\paintedTree[.5]{[[[[][]][[][]]][]]}{}};
  \node[xshift=-5pt, yshift=12pt] at (v69) {\paintedTree[.5]{[[[][]][[[][]][]]]}{}};
  \node[xshift=-10pt, yshift=0pt] at (v70) {\paintedTree[.5]{[[[][]][[][[][]]]]}{}};
  \node[xshift=12pt, yshift=0pt] at (v71) {\paintedTree[.5]{[[[[][[][]]][]][]]}{}};
  \node[xshift=0pt, yshift=-10pt] at (v72) {\paintedTree[.5]{[[[][[][]]][[][]]]}{}};
  \node[xshift=12pt, yshift=8pt] at (v73) {\paintedTree[.5]{[[[][[[][]][]]][]]}{}};
  \node[xshift=12pt, yshift=12pt] at (v74) {\paintedTree[.5]{[[[][[][[][]]]][]]}{}};
  \node[xshift=5pt, yshift=-12pt] at (v75) {\paintedTree[.5]{[[][[[[][]][]][]]]}{}};
  \node[xshift=10pt, yshift=-5pt] at (v76) {\paintedTree[.5]{[[][[[][]][[][]]]]}{}};
  \node[xshift=-5pt, yshift=-12pt] at (v77) {\paintedTree[.5]{[[][[[][[][]]][]]]}{}};
  \node[xshift=5pt, yshift=-12pt] at (v78) {\paintedTree[.5]{[[][[][[[][]][]]]]}{}};
  \node[xshift=-3pt, yshift=-10pt] at (v79) {\paintedTree[.5]{[[][[][[][[][]]]]]}{}};

\end{tikzpicture}} &
		\scalebox{.48}{\begin{tikzpicture}[x  = {(1cm,-.5cm)},
                    y  = {(-1cm,1cm)},
                    z  = {(1cm,1cm)},
                    scale = 2,
                    align=center]

  \coordinate (v87) at (0, -1, 1);
  \coordinate (v88) at (0, -1, -1);
  \coordinate (v89) at (0, 0, 1);
  \coordinate (v90) at (0, 1, 0);
  \coordinate (v91) at (0, 1, -1);
  \coordinate (v92) at (1, -1, 1);
  \coordinate (v93) at (1, -1, -1);
  \coordinate (v94) at (1, 0, -2);
  \coordinate (v95) at (1, 1, 1);
  \coordinate (v96) at (1, 2, 0);
  \coordinate (v97) at (1, 2, -2);
  \coordinate (v98) at (2, 0, 1);
  \coordinate (v99) at (2, 0, -2);
  \coordinate (v100) at (2, 1, 1);
  \coordinate (v101) at (2, 3, -1);
  \coordinate (v102) at (2, 3, -2);
  \coordinate (v103) at (3, 1, 0);
  \coordinate (v104) at (3, 1, -2);
  \coordinate (v105) at (3, 2, 0);
  \coordinate (v106) at (3, 3, -1);
  \coordinate (v107) at (3, 3, -2);

  \tikzstyle{back} = [thick]
  \tikzstyle{front} = [ultra thick]

  \draw[back] (v103) -- (v98);
  \draw[back] (v104) -- (v99);
  \draw[back] (v104) -- (v103);
  \draw[back] (v105) -- (v100);
  \draw[back] (v105) -- (v103);
  \draw[back] (v106) -- (v101);
  \draw[back] (v106) -- (v105);
  \draw[back] (v107) -- (v102);
  \draw[back] (v107) -- (v104);
  \draw[back] (v107) -- (v106);

  \draw[front] (v88) -- (v87);
  \draw[front] (v89) -- (v87);
  \draw[front] (v90) -- (v89);
  \draw[front] (v91) -- (v88);
  \draw[front] (v91) -- (v90);
  \draw[front] (v92) -- (v87);
  \draw[front] (v93) -- (v88);
  \draw[front] (v93) -- (v92);
  \draw[front] (v94) -- (v88);
  \draw[front] (v95) -- (v89);
  \draw[front] (v96) -- (v90);
  \draw[front] (v96) -- (v95);
  \draw[front] (v97) -- (v91);
  \draw[front] (v97) -- (v94);
  \draw[front] (v98) -- (v92);
  \draw[front] (v99) -- (v93);
  \draw[front] (v99) -- (v94);
  \draw[front] (v100) -- (v95);
  \draw[front] (v100) -- (v98);
  \draw[front] (v101) -- (v96);
  \draw[front] (v102) -- (v97);
  \draw[front] (v102) -- (v101);

  \node[xshift=10pt, yshift=12pt] at (v87) {\paintedTree[.5]{[[[[1[]][1[]]][1[]]][1[]]]}{1}};
  \node[yshift=18pt] at (v88) {\paintedTree[.5]{[[[1[]][1[]]][[1[]][1[]]]]}{1}};
  \node[xshift=10pt, yshift=10pt] at (v89) {\paintedTree[.5]{[[[1[]][[1[]][1[]]]][1[]]]}{1}};
  \node[xshift=-10pt, yshift=10pt] at (v90) {\paintedTree[.5]{[[1[]][[[1[]][1[]]][1[]]]]}{1}};
  \node[xshift=-10pt, yshift=12pt] at (v91) {\paintedTree[.5]{[[1[]][[1[]][[1[]][1[]]]]]}{1}};
  \node[xshift=10pt] at (v92) {\paintedTree[.5]{[[[1[[][]]][1[]]][1[]]]}{1}};
  \node[xshift=5pt, yshift=-10pt] at (v93) {\paintedTree[.5]{[[1[[][]]][[1[]][1[]]]]}{1}};
  \node[xshift=-5pt, yshift=-10pt] at (v94) {\paintedTree[.5]{[[[1[]][1[]]][1[[][]]]]}{1}};
  \node[yshift=13pt] at (v95) {\paintedTree[.5]{[[[1[]][1[[][]]]][1[]]]}{1}};
  \node[yshift=13pt] at (v96) {\paintedTree[.5]{[[1[]][[1[[][]]][1[]]]]}{1}};
  \node[xshift=-10pt] at (v97) {\paintedTree[.5]{[[1[]][[1[]][1[[][]]]]]}{1}};
  \node[xshift=10pt, yshift=5pt] at (v98) {\paintedTree[.5]{[[1[[[][]][]]][1[]]]}{1}};
  \node[yshift=-10pt] at (v99) {\paintedTree[.5]{[[1[[][]]][1[[][]]]]}{1}};
  \node[xshift=10pt, yshift=10pt] at (v100) {\paintedTree[.5]{[[1[[][[][]]]][1[]]]}{1}};
  \node[xshift=-10pt, yshift=10pt] at (v101) {\paintedTree[.5]{[[1[]][1[[[][]][]]]]}{1}};
  \node[xshift=-10pt, yshift=5pt] at (v102) {\paintedTree[.5]{[[1[]][1[[][[][]]]]]}{1}};
  \node[xshift=5pt, yshift=-12pt] at (v103) {\paintedTree[.5]{[1[[[[][]][]][]]]}{1}};
  \node[xshift=-10pt, yshift=-5pt] at (v104) {\paintedTree[.5]{[1[[[][]][[][]]]]}{1}};
  \node[xshift=-7pt, yshift=12pt] at (v105) {\paintedTree[.5]{[1[[[][[][]]][]]]}{1}};
  \node[xshift=7pt, yshift=12pt] at (v106) {\paintedTree[.5]{[1[[][[[][]][]]]]}{1}};
  \node[xshift=-5pt, yshift=-12pt] at (v107) {\paintedTree[.5]{[1[[][[][[][]]]]]}{1}};

\end{tikzpicture}}
	\end{tabular}
	}
	\centerline{
	\begin{tabular}{c@{\hspace{-.2cm}}c@{\hspace{-.2cm}}c}
		\scalebox{.48}{\begin{tikzpicture}[x  = {(1cm,-1.2cm)},
                    y  = {(-1cm,1cm)},
                    z  = {(1cm,1cm)},
                    scale = 2,
                    align=center]

  \coordinate (v96) at (0, -1, 1);
  \coordinate (v97) at (0, -1, 0);
  \coordinate (v98) at (0, 0, 1);
  \coordinate (v99) at (0, 0, -1);
  \coordinate (v100) at (0, 1, 0);
  \coordinate (v101) at (0, 1, -1);
  \coordinate (v102) at (1, -1, 1);
  \coordinate (v103) at (1, -1, 0);
  \coordinate (v104) at (1, 1, 1);
  \coordinate (v105) at (1, 1, -2);
  \coordinate (v106) at (1, 2, 0);
  \coordinate (v107) at (1, 2, -2);
  \coordinate (v108) at (2, 0, 1);
  \coordinate (v109) at (2, 0, -1);
  \coordinate (v110) at (2, 1, 1);
  \coordinate (v111) at (2, 1, -2);
  \coordinate (v112) at (2, 3, -1);
  \coordinate (v113) at (2, 3, -2);
  \coordinate (v114) at (3, 1, 0);
  \coordinate (v115) at (3, 1, -1);
  \coordinate (v116) at (3, 2, 0);
  \coordinate (v117) at (3, 2, -2);
  \coordinate (v118) at (3, 3, -1);
  \coordinate (v119) at (3, 3, -2);

  \tikzstyle{back} = [thick]
  \tikzstyle{front} = [ultra thick]

  \draw[back] (v97) -- (v96);
  \draw[back] (v99) -- (v97);
  \draw[back] (v101) -- (v99);
  \draw[back] (v103) -- (v97);
  \draw[back] (v103) -- (v102);
  \draw[back] (v105) -- (v99);
  \draw[back] (v107) -- (v105);
  \draw[back] (v109) -- (v103);
  \draw[back] (v111) -- (v105);
  \draw[back] (v111) -- (v109);
  \draw[back] (v115) -- (v109);
  \draw[back] (v117) -- (v111);
  
  \draw[front] (v98) -- (v96);
  \draw[front] (v100) -- (v98);
  \draw[front] (v101) -- (v100);
  \draw[front] (v102) -- (v96);
  \draw[front] (v104) -- (v98);
  \draw[front] (v106) -- (v100);
  \draw[front] (v106) -- (v104);
  \draw[front] (v107) -- (v101);
  \draw[front] (v108) -- (v102);
  \draw[front] (v110) -- (v104);
  \draw[front] (v110) -- (v108);
  \draw[front] (v112) -- (v106);
  \draw[front] (v113) -- (v107);
  \draw[front] (v113) -- (v112);
  \draw[front] (v114) -- (v108);
  \draw[front] (v115) -- (v114);
  \draw[front] (v116) -- (v110);
  \draw[front] (v116) -- (v114);
  \draw[front] (v117) -- (v115);
  \draw[front] (v118) -- (v112);
  \draw[front] (v118) -- (v116);
  \draw[front] (v119) -- (v113);
  \draw[front] (v119) -- (v117);
  \draw[front] (v119) -- (v118);

  \node[xshift=5pt, yshift=12pt] at (v96) {\paintedTree[.5]{[[1[2[[][]]]][1[2[]]]]}{1,2}};
  \node[xshift=-5pt, yshift=12pt] at (v97) {\paintedTree[.5]{[1[[2[[][]]][2[]]]]]}{1,2}};
  \node[xshift=5pt, yshift=12pt] at (v98) {\paintedTree[.5]{[[2[1[[][]]]][2[1[]]]]}{1,2}};
  \node[xshift=5pt, yshift=12pt] at (v99) {\paintedTree[.5]{[1[2[[[][]][]]]]}{1,2}};
  \node[xshift=-5pt, yshift=12pt] at (v100) {\paintedTree[.5]{[2[[1[[][]]][1[]]]]}{1,2}};
  \node[xshift=-5pt, yshift=12pt] at (v101) {\paintedTree[.5]{[2[1[[[][]][]]]]}{1,2}};
  \node[xshift=10pt, yshift=5pt] at (v102) {\paintedTree[.5]{[[1[[2[]][2[]]]][1[2[]]]]}{1,2}};
  \node[xshift=0pt, yshift=18pt] at (v103) {\paintedTree[.5]{[1[[[2[]][2[]]][2[]]]]]}{1,2}};
  \node[xshift=-5pt, yshift=-13pt] at (v104) {\paintedTree[.5]{[[2[[1[]][1[]]]][2[1[]]]]}{1,2}};
  \node[xshift=0pt, yshift=18pt] at (v105) {\paintedTree[.5]{[1[2[[][[][]]]]]]}{1,2}};
  \node[xshift=5pt, yshift=-13pt] at (v106) {\paintedTree[.5]{[2[[[1[]][1[]]][1[]]]]}{1,2}};
  \node[xshift=-10pt, yshift=5pt] at (v107) {\paintedTree[.5]{[2[1[[][[][]]]]]}{1,2}};
  \node[xshift=10pt, yshift=-5pt] at (v108) {\paintedTree[.5]{[[[1[2[]]][1[2[]]]][1[2[]]]]}{1,2}};
  \node[xshift=-5pt, yshift=12pt] at (v109) {\paintedTree[.5]{[1[[2[]][[2[]][2[]]]]]]}{1,2}};
  \node[xshift=0pt, yshift=-18pt] at (v110) {\paintedTree[.5]{[[[2[1[]]][2[1[]]]][2[1[]]]]}{1,2}};
  \node[xshift=5pt, yshift=12pt] at (v111) {\paintedTree[.5]{[1[[2[]][2[[][]]]]]]}{1,2}};
  \node[xshift=0pt, yshift=-18pt] at (v112) {\paintedTree[.5]{[2[[1[]][[1[]][1[]]]]]}{1,2}};
  \node[xshift=-10pt, yshift=-5pt] at (v113) {\paintedTree[.5]{[2[[1[]][1[[][]]]]]}{1,2}};
  \node[xshift=0pt, yshift=-15pt] at (v114) {\paintedTree[.5]{[[1[2[]]][[1[2[]]][1[2[]]]]]}{1,2}};
  \node[xshift=0pt, yshift=-15pt] at (v115) {\paintedTree[.5]{[[1[2[]]][1[[2[]][2[]]]]]}{1,2}};
  \node[xshift=0pt, yshift=-15pt] at (v116) {\paintedTree[.5]{[[2[1[]]][[2[1[]]][2[1[]]]]]}{1,2}};
  \node[xshift=0pt, yshift=-15pt] at (v117) {\paintedTree[.5]{[[1[2[]]][1[2[[][]]]]]}{1,2}};
  \node[xshift=0pt, yshift=-15pt] at (v118) {\paintedTree[.5]{[[2[1[]]][2[[1[]][1[]]]]]}{1,2}};
  \node[xshift=0pt, yshift=-15pt] at (v119) {\paintedTree[.5]{[[2[1[]]][2[1[[][]]]]]}{1,2}};
  
\end{tikzpicture}} &
		\scalebox{.48}{\begin{tikzpicture}[x  = {(1cm,-1.2cm)},
                    y  = {(-1cm,1cm)},
                    z  = {(1cm,1cm)},
                    scale = 2,
                    align=center]

  \coordinate (v96) at (0, -1, 1);
  \coordinate (v97) at (0, -1, 0);
  \coordinate (v98) at (0, 0, 1);
  \coordinate (v99) at (0, 0, -1);
  \coordinate (v100) at (0, 1, 0);
  \coordinate (v101) at (0, 1, -1);
  \coordinate (v102) at (1, -1, 1);
  \coordinate (v103) at (1, -1, 0);
  \coordinate (v104) at (1, 1, 1);
  \coordinate (v105) at (1, 1, -2);
  \coordinate (v106) at (1, 2, 0);
  \coordinate (v107) at (1, 2, -2);
  \coordinate (v108) at (2, 0, 1);
  \coordinate (v109) at (2, 0, -1);
  \coordinate (v110) at (2, 1, 1);
  \coordinate (v111) at (2, 1, -2);
  \coordinate (v112) at (2, 3, -1);
  \coordinate (v113) at (2, 3, -2);
  \coordinate (v114) at (3, 1, 0);
  \coordinate (v115) at (3, 1, -1);
  \coordinate (v116) at (3, 2, 0);
  \coordinate (v117) at (3, 2, -2);
  \coordinate (v118) at (3, 3, -1);
  \coordinate (v119) at (3, 3, -2);

  \tikzstyle{back} = [thick]
  \tikzstyle{front} = [ultra thick]

  \draw[back] (v97) -- (v96);
  \draw[back] (v99) -- (v97);
  \draw[back] (v101) -- (v99);
  \draw[back] (v103) -- (v97);
  \draw[back] (v103) -- (v102);
  \draw[back] (v105) -- (v99);
  \draw[back] (v107) -- (v105);
  \draw[back] (v109) -- (v103);
  \draw[back] (v111) -- (v105);
  \draw[back] (v111) -- (v109);
  \draw[back] (v115) -- (v109);
  \draw[back] (v117) -- (v111);
  
  \draw[front] (v98) -- (v96);
  \draw[front] (v100) -- (v98);
  \draw[front] (v101) -- (v100);
  \draw[front] (v102) -- (v96);
  \draw[front] (v104) -- (v98);
  \draw[front] (v106) -- (v100);
  \draw[front] (v106) -- (v104);
  \draw[front] (v107) -- (v101);
  \draw[front] (v108) -- (v102);
  \draw[front] (v110) -- (v104);
  \draw[front] (v110) -- (v108);
  \draw[front] (v112) -- (v106);
  \draw[front] (v113) -- (v107);
  \draw[front] (v113) -- (v112);
  \draw[front] (v114) -- (v108);
  \draw[front] (v115) -- (v114);
  \draw[front] (v116) -- (v110);
  \draw[front] (v116) -- (v114);
  \draw[front] (v117) -- (v115);
  \draw[front] (v118) -- (v112);
  \draw[front] (v118) -- (v116);
  \draw[front] (v119) -- (v113);
  \draw[front] (v119) -- (v117);
  \draw[front] (v119) -- (v118);

  \node[xshift=5pt, yshift=12pt] at (v96) {\paintedTree[.5]{[[1[2[3[]]]][1[2[3[]]]]]}{1,2,3}};
  \node[xshift=-5pt, yshift=12pt] at (v97) {\paintedTree[.5]{[1[[2[3[]]][2[3[]]]]]]}{1,2,3}};
  \node[xshift=5pt, yshift=12pt] at (v98) {\paintedTree[.5]{[[2[1[3[]]]][2[1[3[]]]]]}{1,2,3}};
  \node[xshift=5pt, yshift=12pt] at (v99) {\paintedTree[.5]{[1[2[[3[]][3[]]]]]}{1,2,3}};
  \node[xshift=-5pt, yshift=12pt] at (v100) {\paintedTree[.5]{[2[[1[3[]]][1[3[]]]]]}{1,2,3}};
  \node[xshift=-5pt, yshift=12pt] at (v101) {\paintedTree[.5]{[2[1[[3[]][3[]]]]]}{1,2,3}};
  \node[xshift=10pt, yshift=5pt] at (v102) {\paintedTree[.5]{[[1[3[2[]]]][1[3[2[]]]]]}{1,2,3}};
  \node[xshift=0pt, yshift=18pt] at (v103) {\paintedTree[.5]{[1[[3[2[]]][3[2[]]]]]]}{1,2,3}};
  \node[xshift=-5pt, yshift=-13pt] at (v104) {\paintedTree[.5]{[[2[3[1[]]]][2[3[1[]]]]]}{1,2,3}};
  \node[xshift=0pt, yshift=18pt] at (v105) {\paintedTree[.5]{[1[2[3[[][]]]]]]}{1,2,3}};
  \node[xshift=5pt, yshift=-13pt] at (v106) {\paintedTree[.5]{[2[[3[1[]]][3[1[]]]]]}{1,2,3}};
  \node[xshift=-10pt, yshift=5pt] at (v107) {\paintedTree[.5]{[2[1[3[[][]]]]]}{1,2,3}};
  \node[xshift=10pt, yshift=-5pt] at (v108) {\paintedTree[.5]{[[3[1[2[]]]][3[1[2[]]]]]}{1,2,3}};
  \node[xshift=-5pt, yshift=12pt] at (v109) {\paintedTree[.5]{[1[3[[2[]][2[]]]]]}{1,2,3}};
  \node[xshift=0pt, yshift=-18pt] at (v110) {\paintedTree[.5]{[[3[2[1[]]]][3[2[1[]]]]]}{1,2,3}};
  \node[xshift=5pt, yshift=12pt] at (v111) {\paintedTree[.5]{[1[3[2[[][]]]]]]}{1,2,3}};
  \node[xshift=0pt, yshift=-18pt] at (v112) {\paintedTree[.5]{[2[3[[1[]][1[]]]]]}{1,2,3}};
  \node[xshift=-10pt, yshift=-5pt] at (v113) {\paintedTree[.5]{[2[3[1[[][]]]]]}{1,2,3}};
  \node[xshift=0pt, yshift=-15pt] at (v114) {\paintedTree[.5]{[3[[1[2[]]][1[2[]]]]]}{1,2,3}};
  \node[xshift=0pt, yshift=-15pt] at (v115) {\paintedTree[.5]{[3[1[[2[]][2[]]]]]}{1,2,3}};
  \node[xshift=0pt, yshift=-15pt] at (v116) {\paintedTree[.5]{[3[[2[1[]]][2[1[]]]]]}{1,2,3}};
  \node[xshift=0pt, yshift=-15pt] at (v117) {\paintedTree[.5]{[3[1[2[[][]]]]]}{1,2,3}};
  \node[xshift=0pt, yshift=-15pt] at (v118) {\paintedTree[.5]{[3[2[[1[]][1[]]]]]}{1,2,3}};
  \node[xshift=0pt, yshift=-15pt] at (v119) {\paintedTree[.5]{[3[2[1[[][]]]]]}{1,2,3}};
  
\end{tikzpicture}} &
		\scalebox{.48}{\begin{tikzpicture}[x  = {(1cm,-1.2cm)},
                    y  = {(-1cm,1cm)},
                    z  = {(1cm,1cm)},
                    scale = 2,
                    align=center]

  \coordinate (v96) at (0, -1, 1);
  \coordinate (v97) at (0, -1, 0);
  \coordinate (v98) at (0, 0, 1);
  \coordinate (v99) at (0, 0, -1);
  \coordinate (v100) at (0, 1, 0);
  \coordinate (v101) at (0, 1, -1);
  \coordinate (v102) at (1, -1, 1);
  \coordinate (v103) at (1, -1, 0);
  \coordinate (v104) at (1, 1, 1);
  \coordinate (v105) at (1, 1, -2);
  \coordinate (v106) at (1, 2, 0);
  \coordinate (v107) at (1, 2, -2);
  \coordinate (v108) at (2, 0, 1);
  \coordinate (v109) at (2, 0, -1);
  \coordinate (v110) at (2, 1, 1);
  \coordinate (v111) at (2, 1, -2);
  \coordinate (v112) at (2, 3, -1);
  \coordinate (v113) at (2, 3, -2);
  \coordinate (v114) at (3, 1, 0);
  \coordinate (v115) at (3, 1, -1);
  \coordinate (v116) at (3, 2, 0);
  \coordinate (v117) at (3, 2, -2);
  \coordinate (v118) at (3, 3, -1);
  \coordinate (v119) at (3, 3, -2);

  \tikzstyle{back} = [thick]
  \tikzstyle{front} = [ultra thick]

  \draw[back] (v97) -- (v96);
  \draw[back] (v99) -- (v97);
  \draw[back] (v101) -- (v99);
  \draw[back] (v103) -- (v97);
  \draw[back] (v103) -- (v102);
  \draw[back] (v105) -- (v99);
  \draw[back] (v107) -- (v105);
  \draw[back] (v109) -- (v103);
  \draw[back] (v111) -- (v105);
  \draw[back] (v111) -- (v109);
  \draw[back] (v115) -- (v109);
  \draw[back] (v117) -- (v111);
  
  \draw[front] (v98) -- (v96);
  \draw[front] (v100) -- (v98);
  \draw[front] (v101) -- (v100);
  \draw[front] (v102) -- (v96);
  \draw[front] (v104) -- (v98);
  \draw[front] (v106) -- (v100);
  \draw[front] (v106) -- (v104);
  \draw[front] (v107) -- (v101);
  \draw[front] (v108) -- (v102);
  \draw[front] (v110) -- (v104);
  \draw[front] (v110) -- (v108);
  \draw[front] (v112) -- (v106);
  \draw[front] (v113) -- (v107);
  \draw[front] (v113) -- (v112);
  \draw[front] (v114) -- (v108);
  \draw[front] (v115) -- (v114);
  \draw[front] (v116) -- (v110);
  \draw[front] (v116) -- (v114);
  \draw[front] (v117) -- (v115);
  \draw[front] (v118) -- (v112);
  \draw[front] (v118) -- (v116);
  \draw[front] (v119) -- (v113);
  \draw[front] (v119) -- (v117);
  \draw[front] (v119) -- (v118);

  \node[xshift=5pt, yshift=12pt] at (v96) {\paintedTree[.5]{[4[1[2[3[]]]]]}{1,2,3,4}};
  \node[xshift=-5pt, yshift=12pt] at (v97) {\paintedTree[.5]{[1[4[2[3[]]]]]]}{1,2,3,4}};
  \node[xshift=5pt, yshift=12pt] at (v98) {\paintedTree[.5]{[4[2[1[3[]]]]]}{1,2,3,4}};
  \node[xshift=5pt, yshift=12pt] at (v99) {\paintedTree[.5]{[1[2[4[3[]]]]]}{1,2,3,4}};
  \node[xshift=-5pt, yshift=12pt] at (v100) {\paintedTree[.5]{[2[4[1[3[]]]]]}{1,2,3,4}};
  \node[xshift=-5pt, yshift=12pt] at (v101) {\paintedTree[.5]{[2[1[4[3[]]]]]}{1,2,3,4}};
  \node[xshift=10pt, yshift=5pt] at (v102) {\paintedTree[.5]{[4[1[3[2[]]]]]}{1,2,3,4}};
  \node[xshift=0pt, yshift=18pt] at (v103) {\paintedTree[.5]{[1[4[3[2[]]]]]]}{1,2,3,4}};
  \node[xshift=-5pt, yshift=-13pt] at (v104) {\paintedTree[.5]{[4[2[3[1[]]]]]}{1,2,3,4}};
  \node[xshift=0pt, yshift=18pt] at (v105) {\paintedTree[.5]{[1[2[3[4[]]]]]]}{1,2,3,4}};
  \node[xshift=5pt, yshift=-13pt] at (v106) {\paintedTree[.5]{[2[4[3[1[]]]]]}{1,2,3,4}};
  \node[xshift=-10pt, yshift=5pt] at (v107) {\paintedTree[.5]{[2[1[3[4[]]]]]}{1,2,3,4}};
  \node[xshift=10pt, yshift=-5pt] at (v108) {\paintedTree[.5]{[4[3[1[2[]]]]]}{1,2,3,4}};
  \node[xshift=-5pt, yshift=12pt] at (v109) {\paintedTree[.5]{[1[3[4[2[]]]]]}{1,2,3,4}};
  \node[xshift=0pt, yshift=-18pt] at (v110) {\paintedTree[.5]{[4[3[2[1[]]]]]}{1,2,3,4}};
  \node[xshift=5pt, yshift=12pt] at (v111) {\paintedTree[.5]{[1[3[2[4[]]]]]]}{1,2,3,4}};
  \node[xshift=0pt, yshift=-18pt] at (v112) {\paintedTree[.5]{[2[3[4[1[]]]]]}{1,2,3,4}};
  \node[xshift=-10pt, yshift=-5pt] at (v113) {\paintedTree[.5]{[2[3[1[4[]]]]]}{1,2,3,4}};
  \node[xshift=0pt, yshift=-15pt] at (v114) {\paintedTree[.5]{[3[4[1[2[]]]]]}{1,2,3,4}};
  \node[xshift=0pt, yshift=-15pt] at (v115) {\paintedTree[.5]{[3[1[4[2[]]]]]}{1,2,3,4}};
  \node[xshift=0pt, yshift=-15pt] at (v116) {\paintedTree[.5]{[3[4[2[1[]]]]]}{1,2,3,4}};
  \node[xshift=0pt, yshift=-15pt] at (v117) {\paintedTree[.5]{[3[1[2[4[]]]]]}{1,2,3,4}};
  \node[xshift=0pt, yshift=-15pt] at (v118) {\paintedTree[.5]{[3[2[4[1[]]]]]}{1,2,3,4}};
  \node[xshift=0pt, yshift=-15pt] at (v119) {\paintedTree[.5]{[3[2[1[4[]]]]]}{1,2,3,4}};
  
\end{tikzpicture}} \\
		$(1, 24, 36, 14, 1)$ &
		$(1, 24, 36, 14, 1)$ &
		$(1, 24, 36, 14, 1)$
	\end{tabular}
	}
	\caption{The $(m,n)$-multiplihedra $\Multiplihedron$ and their $f$-vectors for~$(m,n) = (0,4)$, $(1,3)$, $(2,2)$, $(3,1)$ and~$(4,0)$. The top two are the $3$-dimensional associahedron~$\Asso[4]$ and multiplihedron, while the bottom three are all relabelings of the $3$-dimensional permutahedron~$\Perm[4]$.}
	\label{fig:multiplihedra3}
\end{figure}

\begin{figure}
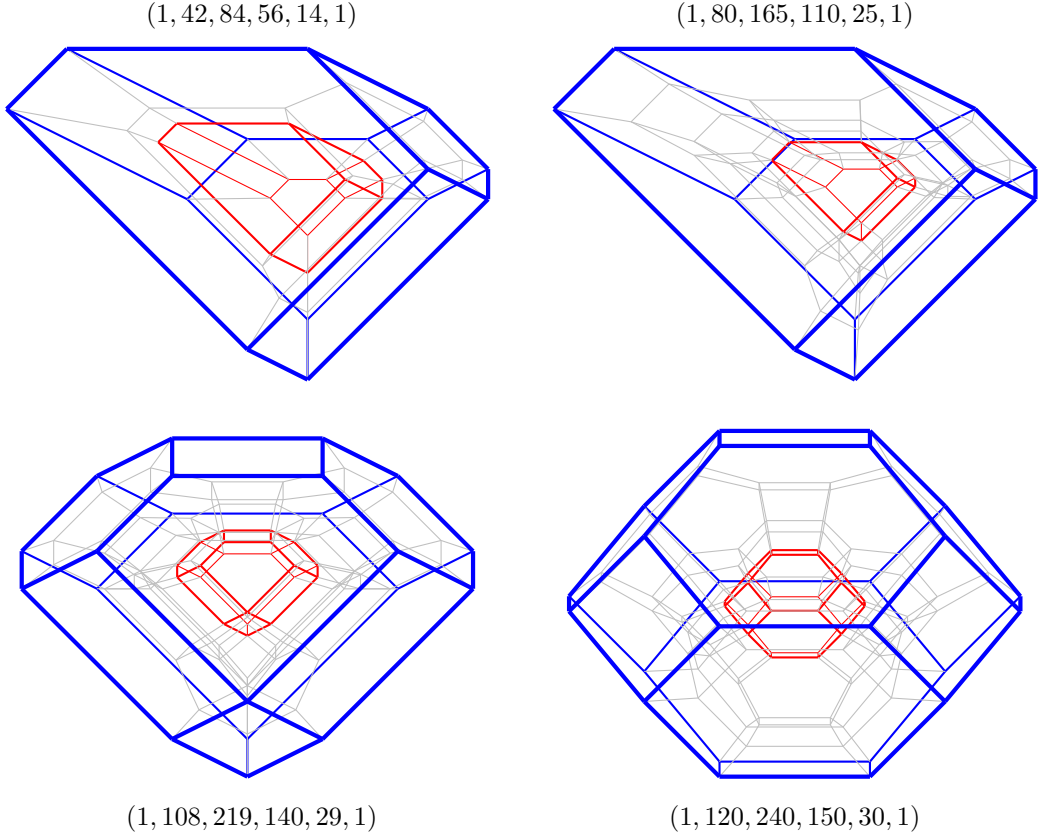

	\centerline{
		\begin{tabular}{c@{\qquad}c}
			$(1, 42, 84, 56, 14, 1)$ &
			$(1, 80, 165, 110, 25, 1)$ \\[.2cm]
			\begin{tikzpicture}[x  = {(1cm,-.5cm)},
                    y  = {(-1cm,1cm)},
                    z  = {(1cm,1cm)},
                    scale = .8,
                    color = {lightgray}]

  \coordinate (v0) at (1.03448, 0.37931, -0.0344828);
  \coordinate (v1) at (1.03448, 0.37931, -1.27586);
  \coordinate (v2) at (1.03448, 1.31034, -0.0344828);
  \coordinate (v3) at (1.03448, 2.24138, -0.965517);
  \coordinate (v4) at (1.03448, 2.24138, -1.27586);
  \coordinate (v5) at (1.65517, 0.37931, -0.0344828);
  \coordinate (v6) at (1.65517, 0.37931, -1.27586);
  \coordinate (v7) at (2.27586, 1, -0.0344828);
  \coordinate (v8) at (2.27586, 1.31034, -0.0344828);
  \coordinate (v9) at (2.89655, 1.62069, -0.655172);
  \coordinate (v10) at (2.89655, 1.62069, -1.27586);
  \coordinate (v11) at (2.89655, 1.93103, -0.655172);
  \coordinate (v12) at (2.89655, 2.24138, -0.965517);
  \coordinate (v13) at (2.89655, 2.24138, -1.27586);
  \coordinate (v14) at (0.65625, 0.0625, 0.15625);
  \coordinate (v15) at (0.65625, 0.0625, -1.34375);
  \coordinate (v16) at (0.65625, 1.1875, 0.15625);
  \coordinate (v17) at (0.65625, 2.3125, -0.96875);
  \coordinate (v18) at (0.65625, 2.3125, -1.34375);
  \coordinate (v19) at (0.552632, -0.421053, 0.447368);
  \coordinate (v20) at (0.552632, -0.421053, -1.44737);
  \coordinate (v21) at (1.02632, -0.421053, 0.447368);
  \coordinate (v22) at (1.02632, -0.421053, -1.44737);
  \coordinate (v23) at (0.375, -0.607143, 0.946429);
  \coordinate (v24) at (0.375, 0.678571, 0.946429);
  \coordinate (v25) at (1.01786, -0.607143, 0.946429);
  \coordinate (v26) at (1.66071, 0.0357143, 0.946429);
  \coordinate (v27) at (1.66071, 0.678571, 0.946429);
  \coordinate (v28) at (0, -1, 1);
  \coordinate (v29) at (0, -1, -2);
  \coordinate (v30) at (0, 1, 1);
  \coordinate (v31) at (0, 3, -1);
  \coordinate (v32) at (0, 3, -2);
  \coordinate (v33) at (1, -1, 1);
  \coordinate (v34) at (1, -1, -2);
  \coordinate (v35) at (2, 0, 1);
  \coordinate (v36) at (2, 1, 1);
  \coordinate (v37) at (3, 1, 0);
  \coordinate (v38) at (3, 1, -2);
  \coordinate (v39) at (3, 2, 0);
  \coordinate (v40) at (3, 3, -1);
  \coordinate (v41) at (3, 3, -2);

  \tikzstyle{frameBack} = [color=blue, thick]
  \tikzstyle{centerBack} = [color=red, thin]
  \tikzstyle{centerFront} = [color=red, thick]
  \tikzstyle{foam} = [color=lightgray, thin]
  \tikzstyle{frameFront} = [color=blue, ultra thick]

  \draw[frameBack] (v37) -- (v35);
  \draw[frameBack] (v38) -- (v34);
  \draw[frameBack] (v38) -- (v37);
  \draw[frameBack] (v39) -- (v36);
  \draw[frameBack] (v39) -- (v37);
  \draw[frameBack] (v40) -- (v31);
  \draw[frameBack] (v40) -- (v39);
  \draw[frameBack] (v41) -- (v32);
  \draw[frameBack] (v41) -- (v38);
  \draw[frameBack] (v41) -- (v40);

  \draw[centerBack] (v9) -- (v7);
  \draw[centerBack] (v10) -- (v6);
  \draw[centerBack] (v10) -- (v9);
  \draw[centerBack] (v11) -- (v8);
  \draw[centerBack] (v11) -- (v9);
  \draw[centerBack] (v12) -- (v3);
  \draw[centerBack] (v12) -- (v11);
  \draw[centerBack] (v13) -- (v4);
  \draw[centerBack] (v13) -- (v10);
  \draw[centerBack] (v13) -- (v12);

  \draw[centerFront] (v1) -- (v0);
  \draw[centerFront] (v2) -- (v0);
  \draw[centerFront] (v3) -- (v2);
  \draw[centerFront] (v4) -- (v1);
  \draw[centerFront] (v4) -- (v3);
  \draw[centerFront] (v5) -- (v0);
  \draw[centerFront] (v6) -- (v1);
  \draw[centerFront] (v6) -- (v5);
  \draw[centerFront] (v7) -- (v5);
  \draw[centerFront] (v8) -- (v2);
  \draw[centerFront] (v8) -- (v7);

  \draw[foam] (v14) -- (v0);
  \draw[foam] (v15) -- (v1);
  \draw[foam] (v15) -- (v14);
  \draw[foam] (v16) -- (v2);
  \draw[foam] (v16) -- (v14);
  \draw[foam] (v17) -- (v3);
  \draw[foam] (v17) -- (v16);
  \draw[foam] (v18) -- (v4);
  \draw[foam] (v18) -- (v15);
  \draw[foam] (v18) -- (v17);
  \draw[foam] (v19) -- (v14);
  \draw[foam] (v20) -- (v15);
  \draw[foam] (v20) -- (v19);
  \draw[foam] (v21) -- (v5);
  \draw[foam] (v21) -- (v19);
  \draw[foam] (v22) -- (v6);
  \draw[foam] (v22) -- (v20);
  \draw[foam] (v22) -- (v21);
  \draw[foam] (v23) -- (v19);
  \draw[foam] (v24) -- (v16);
  \draw[foam] (v24) -- (v23);
  \draw[foam] (v25) -- (v21);
  \draw[foam] (v25) -- (v23);
  \draw[foam] (v26) -- (v7);
  \draw[foam] (v26) -- (v25);
  \draw[foam] (v27) -- (v8);
  \draw[foam] (v27) -- (v24);
  \draw[foam] (v27) -- (v26);
  \draw[foam] (v28) -- (v23);
  \draw[foam] (v29) -- (v20);
  \draw[foam] (v30) -- (v24);
  \draw[foam] (v31) -- (v17);
  \draw[foam] (v32) -- (v18);
  \draw[foam] (v33) -- (v25);
  \draw[foam] (v34) -- (v22);
  \draw[foam] (v35) -- (v26);
  \draw[foam] (v36) -- (v27);
  \draw[foam] (v37) -- (v9);
  \draw[foam] (v38) -- (v10);
  \draw[foam] (v39) -- (v11);
  \draw[foam] (v40) -- (v12);
  \draw[foam] (v41) -- (v13);

  \draw[frameFront] (v29) -- (v28);
  \draw[frameFront] (v30) -- (v28);
  \draw[frameFront] (v31) -- (v30);
  \draw[frameFront] (v32) -- (v29);
  \draw[frameFront] (v32) -- (v31);
  \draw[frameFront] (v33) -- (v28);
  \draw[frameFront] (v34) -- (v29);
  \draw[frameFront] (v34) -- (v33);
  \draw[frameFront] (v35) -- (v33);
  \draw[frameFront] (v36) -- (v30);
  \draw[frameFront] (v36) -- (v35);

\end{tikzpicture} &
			\begin{tikzpicture}[x  = {(1cm,-.5cm)},
                    y  = {(-1cm,1cm)},
                    z  = {(1cm,1cm)},
                    scale = .8,
                    color = {lightgray}]

  \coordinate (v0) at (1.1069, 0.404069, 0.000797766);
  \coordinate (v1) at (1.1069, 0.404069, -0.893099);
  \coordinate (v2) at (1.1069, 1, 0.000797766);
  \coordinate (v3) at (1.1069, 1.59593, -0.595134);
  \coordinate (v4) at (1.1069, 1.59593, -0.893099);
  \coordinate (v5) at (1.40487, 0.404069, 0.000797766);
  \coordinate (v6) at (1.40487, 0.404069, -0.893099);
  \coordinate (v7) at (1.70283, 0.702034, 0.000797766);
  \coordinate (v8) at (1.70283, 1, 0.000797766);
  \coordinate (v9) at (2.0008, 1, -0.297168);
  \coordinate (v10) at (2.0008, 1, -0.893099);
  \coordinate (v11) at (2.0008, 1.29797, -0.297168);
  \coordinate (v12) at (2.0008, 1.59593, -0.595134);
  \coordinate (v13) at (2.0008, 1.59593, -0.893099);
  \coordinate (v14) at (0.735849, 0.0965167, 0.181422);
  \coordinate (v15) at (0.735849, 0.0965167, -0.902758);
  \coordinate (v16) at (0.735849, 0.819303, 0.181422);
  \coordinate (v17) at (0.735849, 1.54209, -0.541364);
  \coordinate (v18) at (0.735849, 1.54209, -0.902758);
  \coordinate (v19) at (1.09724, 0.45791, -1.26415);
  \coordinate (v20) at (1.09724, 1.1807, 0.181422);
  \coordinate (v21) at (1.09724, 1.90348, -0.541364);
  \coordinate (v22) at (1.09724, 1.90348, -1.26415);
  \coordinate (v23) at (1.45864, 0.0965167, 0.181422);
  \coordinate (v24) at (1.45864, 0.0965167, -0.902758);
  \coordinate (v25) at (1.45864, 0.45791, -1.26415);
  \coordinate (v26) at (1.82003, 0.45791, 0.181422);
  \coordinate (v27) at (1.82003, 1.1807, 0.181422);
  \coordinate (v28) at (2.18142, 0.819303, -0.179971);
  \coordinate (v29) at (2.18142, 0.819303, -0.902758);
  \coordinate (v30) at (2.18142, 1.1807, -1.26415);
  \coordinate (v31) at (2.18142, 1.54209, -0.179971);
  \coordinate (v32) at (2.18142, 1.90348, -0.541364);
  \coordinate (v33) at (2.18142, 1.90348, -1.26415);
  \coordinate (v34) at (0.623233, -0.377382, 0.459742);
  \coordinate (v35) at (0.623233, -0.377382, -0.91764);
  \coordinate (v36) at (0.623233, 0.0817455, -1.37677);
  \coordinate (v37) at (0.623233, 1, 0.459742);
  \coordinate (v38) at (0.623233, 1.91825, -0.458513);
  \coordinate (v39) at (0.623233, 1.91825, -1.37677);
  \coordinate (v40) at (1.08236, -0.377382, 0.459742);
  \coordinate (v41) at (1.08236, -0.377382, -0.91764);
  \coordinate (v42) at (1.08236, 2.37738, -0.91764);
  \coordinate (v43) at (1.08236, 2.37738, -1.37677);
  \coordinate (v44) at (1.54149, 0.0817455, -1.37677);
  \coordinate (v45) at (2.00061, 0.540873, 0.459742);
  \coordinate (v46) at (2.00061, 1, 0.459742);
  \coordinate (v47) at (2.45974, 1, 0.000614628);
  \coordinate (v48) at (2.45974, 1, -1.37677);
  \coordinate (v49) at (2.45974, 1.45913, 0.000614628);
  \coordinate (v50) at (2.45974, 2.37738, -0.91764);
  \coordinate (v51) at (2.45974, 2.37738, -1.37677);
  \coordinate (v52) at (0.427127, -0.573294, 0.944398);
  \coordinate (v53) at (0.427127, -0.573294, -1.57287);
  \coordinate (v54) at (0.427127, 0.685341, 0.944398);
  \coordinate (v55) at (0.427127, 2.57329, -0.943555);
  \coordinate (v56) at (0.427127, 2.57329, -1.57287);
  \coordinate (v57) at (1.05644, -0.573294, 0.944398);
  \coordinate (v58) at (1.05644, -0.573294, -1.57287);
  \coordinate (v59) at (1.68576, 0.0560236, 0.944398);
  \coordinate (v60) at (1.68576, 0.685341, 0.944398);
  \coordinate (v61) at (2.9444, 1.31466, -0.314238);
  \coordinate (v62) at (2.9444, 1.31466, -1.57287);
  \coordinate (v63) at (2.9444, 1.94398, -0.314238);
  \coordinate (v64) at (2.9444, 2.57329, -0.943555);
  \coordinate (v65) at (2.9444, 2.57329, -1.57287);
  \coordinate (v66) at (0, -1, 1);
  \coordinate (v67) at (0, -1, -2);
  \coordinate (v68) at (0, 1, 1);
  \coordinate (v69) at (0, 3, -1);
  \coordinate (v70) at (0, 3, -2);
  \coordinate (v71) at (1, -1, 1);
  \coordinate (v72) at (1, -1, -2);
  \coordinate (v73) at (2, 0, 1);
  \coordinate (v74) at (2, 1, 1);
  \coordinate (v75) at (3, 1, 0);
  \coordinate (v76) at (3, 1, -2);
  \coordinate (v77) at (3, 2, 0);
  \coordinate (v78) at (3, 3, -1);
  \coordinate (v79) at (3, 3, -2);

  \tikzstyle{frameBack} = [color=blue, thick]
  \tikzstyle{centerBack} = [color=red, thin]
  \tikzstyle{centerFront} = [color=red, thick]
  \tikzstyle{foam} = [color=lightgray, thin]
  \tikzstyle{frameFront} = [color=blue, ultra thick]

  \draw[frameBack] (v75) -- (v73);
  \draw[frameBack] (v76) -- (v72);
  \draw[frameBack] (v76) -- (v75);
  \draw[frameBack] (v77) -- (v74);
  \draw[frameBack] (v77) -- (v75);
  \draw[frameBack] (v78) -- (v69);
  \draw[frameBack] (v78) -- (v77);
  \draw[frameBack] (v79) -- (v70);
  \draw[frameBack] (v79) -- (v76);
  \draw[frameBack] (v79) -- (v78);

  \draw[centerBack] (v9) -- (v7);
  \draw[centerBack] (v10) -- (v6);
  \draw[centerBack] (v10) -- (v9);
  \draw[centerBack] (v11) -- (v8);
  \draw[centerBack] (v11) -- (v9);
  \draw[centerBack] (v12) -- (v3);
  \draw[centerBack] (v12) -- (v11);
  \draw[centerBack] (v13) -- (v4);
  \draw[centerBack] (v13) -- (v10);
  \draw[centerBack] (v13) -- (v12);

  \draw[centerFront] (v1) -- (v0);
  \draw[centerFront] (v2) -- (v0);
  \draw[centerFront] (v3) -- (v2);
  \draw[centerFront] (v4) -- (v1);
  \draw[centerFront] (v4) -- (v3);
  \draw[centerFront] (v5) -- (v0);
  \draw[centerFront] (v6) -- (v1);
  \draw[centerFront] (v6) -- (v5);
  \draw[centerFront] (v7) -- (v5);
  \draw[centerFront] (v8) -- (v2);
  \draw[centerFront] (v8) -- (v7);
  
  \draw[foam] (v14) -- (v0);
  \draw[foam] (v15) -- (v1);
  \draw[foam] (v15) -- (v14);
  \draw[foam] (v16) -- (v2);
  \draw[foam] (v16) -- (v14);
  \draw[foam] (v17) -- (v3);
  \draw[foam] (v17) -- (v16);
  \draw[foam] (v18) -- (v4);
  \draw[foam] (v18) -- (v15);
  \draw[foam] (v18) -- (v17);
  \draw[foam] (v19) -- (v1);
  \draw[foam] (v20) -- (v2);
  \draw[foam] (v21) -- (v3);
  \draw[foam] (v21) -- (v20);
  \draw[foam] (v22) -- (v4);
  \draw[foam] (v22) -- (v19);
  \draw[foam] (v23) -- (v5);
  \draw[foam] (v24) -- (v6);
  \draw[foam] (v24) -- (v23);
  \draw[foam] (v25) -- (v6);
  \draw[foam] (v25) -- (v19);
  \draw[foam] (v26) -- (v7);
  \draw[foam] (v26) -- (v23);
  \draw[foam] (v27) -- (v8);
  \draw[foam] (v27) -- (v20);
  \draw[foam] (v28) -- (v9);
  \draw[foam] (v28) -- (v26);
  \draw[foam] (v29) -- (v10);
  \draw[foam] (v29) -- (v24);
  \draw[foam] (v29) -- (v28);
  \draw[foam] (v30) -- (v10);
  \draw[foam] (v30) -- (v25);
  \draw[foam] (v31) -- (v11);
  \draw[foam] (v31) -- (v27);
  \draw[foam] (v32) -- (v12);
  \draw[foam] (v32) -- (v21);
  \draw[foam] (v32) -- (v31);
  \draw[foam] (v33) -- (v13);
  \draw[foam] (v33) -- (v22);
  \draw[foam] (v33) -- (v30);
  \draw[foam] (v34) -- (v14);
  \draw[foam] (v35) -- (v15);
  \draw[foam] (v35) -- (v34);
  \draw[foam] (v36) -- (v15);
  \draw[foam] (v36) -- (v19);
  \draw[foam] (v37) -- (v16);
  \draw[foam] (v37) -- (v20);
  \draw[foam] (v38) -- (v17);
  \draw[foam] (v38) -- (v21);
  \draw[foam] (v38) -- (v37);
  \draw[foam] (v39) -- (v18);
  \draw[foam] (v39) -- (v22);
  \draw[foam] (v39) -- (v36);
  \draw[foam] (v40) -- (v23);
  \draw[foam] (v40) -- (v34);
  \draw[foam] (v41) -- (v24);
  \draw[foam] (v41) -- (v35);
  \draw[foam] (v41) -- (v40);
  \draw[foam] (v42) -- (v21);
  \draw[foam] (v43) -- (v22);
  \draw[foam] (v43) -- (v42);
  \draw[foam] (v44) -- (v24);
  \draw[foam] (v44) -- (v25);
  \draw[foam] (v45) -- (v26);
  \draw[foam] (v46) -- (v27);
  \draw[foam] (v46) -- (v45);
  \draw[foam] (v47) -- (v28);
  \draw[foam] (v47) -- (v45);
  \draw[foam] (v48) -- (v29);
  \draw[foam] (v48) -- (v30);
  \draw[foam] (v48) -- (v44);
  \draw[foam] (v49) -- (v31);
  \draw[foam] (v49) -- (v46);
  \draw[foam] (v49) -- (v47);
  \draw[foam] (v50) -- (v32);
  \draw[foam] (v50) -- (v42);
  \draw[foam] (v51) -- (v33);
  \draw[foam] (v51) -- (v43);
  \draw[foam] (v51) -- (v50);
  \draw[foam] (v52) -- (v34);
  \draw[foam] (v53) -- (v35);
  \draw[foam] (v53) -- (v36);
  \draw[foam] (v54) -- (v37);
  \draw[foam] (v54) -- (v52);
  \draw[foam] (v55) -- (v38);
  \draw[foam] (v55) -- (v42);
  \draw[foam] (v56) -- (v39);
  \draw[foam] (v56) -- (v43);
  \draw[foam] (v56) -- (v55);
  \draw[foam] (v57) -- (v40);
  \draw[foam] (v57) -- (v52);
  \draw[foam] (v58) -- (v41);
  \draw[foam] (v58) -- (v44);
  \draw[foam] (v58) -- (v53);
  \draw[foam] (v59) -- (v45);
  \draw[foam] (v59) -- (v57);
  \draw[foam] (v60) -- (v46);
  \draw[foam] (v60) -- (v54);
  \draw[foam] (v60) -- (v59);
  \draw[foam] (v61) -- (v47);
  \draw[foam] (v62) -- (v48);
  \draw[foam] (v62) -- (v61);
  \draw[foam] (v63) -- (v49);
  \draw[foam] (v63) -- (v61);
  \draw[foam] (v64) -- (v50);
  \draw[foam] (v64) -- (v63);
  \draw[foam] (v65) -- (v51);
  \draw[foam] (v65) -- (v62);
  \draw[foam] (v65) -- (v64);
  \draw[foam] (v66) -- (v52);
  \draw[foam] (v67) -- (v53);
  \draw[foam] (v68) -- (v54);
  \draw[foam] (v69) -- (v55);
  \draw[foam] (v70) -- (v56);
  \draw[foam] (v71) -- (v57);
  \draw[foam] (v72) -- (v58);
  \draw[foam] (v73) -- (v59);
  \draw[foam] (v74) -- (v60);
  \draw[foam] (v75) -- (v61);
  \draw[foam] (v76) -- (v62);
  \draw[foam] (v77) -- (v63);
  \draw[foam] (v78) -- (v64);
  \draw[foam] (v79) -- (v65);

  \draw[frameFront] (v67) -- (v66);
  \draw[frameFront] (v68) -- (v66);
  \draw[frameFront] (v69) -- (v68);
  \draw[frameFront] (v70) -- (v67);
  \draw[frameFront] (v70) -- (v69);
  \draw[frameFront] (v71) -- (v66);
  \draw[frameFront] (v72) -- (v67);
  \draw[frameFront] (v72) -- (v71);
  \draw[frameFront] (v73) -- (v71);
  \draw[frameFront] (v74) -- (v68);
  \draw[frameFront] (v74) -- (v73);

\end{tikzpicture} \\[.5cm]
			\input{figures/multiplihedron23.tikz} &
			\input{figures/multiplihedron32.tikz} \\[.2cm]
			$(1, 108, 219, 140, 29, 1)$ &
			$(1, 120, 240, 150, 30, 1)$			
		\end{tabular}
	}
	\caption{Schlegel diagrams and $f$-vectors of the $(m,n)$-multiplihedra $\Multiplihedron$ for~$(m,n) = (0,5)$, $(1,4)$, $(2,3)$ and~$(3,2) \sim (4,1) \sim (5,0)$. The top left, top right, and bottom right polytopes are the $4$-dimensional associahedron~$\Asso[5]$, multiplihedron, and permutahedron~$\Perm[5]$. The bottom left polytope is the $(2,3)$-multiplihedron, labeled in \cref{fig:multiplihedron23}.}
	\label{fig:multiplihedra4}
\end{figure}

\begin{proposition}
\label{prop:faceLatticeMultiplihedron}
The face lattice of the $(m,n)$-multiplihedron~$\Multiplihedron$ is isomorphic to the $m$-painted $n$-tree deletion poset (augmented with a minimal element).
\end{proposition}

\begin{proof}
This follows from \cref{prop:shuffleFacePreposets} (see also \cref{rem:biPreposets}).
Indeed, associate to an $m$-painted $n$-tree~$\PT \eqdef (T, C, \mu)$ the face preposet~$\preccurlyeq_{\polytope{F}, \polytope{G}, \lambda}$ where
\begin{itemize}
\item $\polytope{F}$ is the face of the permutahedron~$\Perm[m]$ corresponding to the partition~$\mu$,
\item $\polytope{G}$ is the face of the associahedron~$\Asso[n]$ corresponding to the Schr\"oder tree obtained by deleting all unary nodes in~$T$, and
\item $\lambda$ is the partition of~$[m+n]$ with
	\begin{itemize}
	\item a part formed by~$\mu_i$ and the inorder labels of the nodes of~$C_i$ for each cut~$C_i$ containing a non-unary node, 
	\item a part formed by~$\mu_i \cup \dots \cup \mu_j$ for each maximal sequence of cuts~$C_i, \dots, C_j$ containing only unary nodes and such that~$\stump{C_{k+1}} = \stump{C_k} \ssm C_k$ for all~$i < k \le j$, and 
	\item a part formed by the inorder labels of the nodes in between the cuts~$C_i$ and~$C_{i+1}$ (\ie the nodes of~$\stump{C_i} \ssm (C_i \cup \stump{C_{i+1}})$) for each~$i \in [|C|-1]$.
	\end{itemize}
\end{itemize}
We leave to the reader the immediate verification that this yields a poset isomorphism from the deletion poset on $m$-painted $n$-trees to the refinement poset on the face preposets of the $(m,n)$-multiplihedron~$\Multiplihedron = \Perm[m] \shuffleDP \Asso[n]$.
\end{proof}

\begin{remark}
\label{rem:simpleMultiplihedron}
In contrast to the permutahedron~$\Perm[m]$ and the associahedron~$\Asso[n]$, the multiplihedron~$\Multiplihedron$ is simple if and only if~$m = 0$ or~$n \le 2$.
\end{remark}

\begin{proposition}
\label{prop:fanMultiplihedron}
The normal fan of the $(m,n)$-multiplihedron~$\Multiplihedron$ is the fan containing one cone $\polytope{C}(\PT) \eqdef \set{\b{x} \in \R^{m+n}}{x_i \le x_j \text{ if } i \preccurlyeq_{\PT} j}$ for each~${\PT \in \f{PT}_{m,n}}$.
\end{proposition}

\begin{proof}
Immediate from \cref{prop:faceLatticeMultiplihedron,def:paintedTreePreposet}.
\end{proof}

\begin{proposition}
\label{prop:graphMultiplihedron}
When oriented in the direction~${\b{\omega} \eqdef (n,\dots,1) - (1,\dots,n) = \sum_{i \in [n]} (n+1-2i) \, \b{e}_i}$, the graph of the $(m,n)$-multiplihedron~$\Multiplihedron$ is isomorphic to the right rotation graph on binary $m$-painted $n$-trees, and is the Hasse diagram of a lattice.
\end{proposition}

\begin{proof}
It follows from \cref{prop:faceLatticeMultiplihedron} that the vertices of~$\Multiplihedron$ correspond to the binary $m$-painted $n$-trees.
It is easy to check that the edges of~$\Multiplihedron$ oriented by~$\b{\omega}$ correspond to right rotations on binary $m$-painted $n$-trees.
Finally, the lattice property is a special case of \cref{coro:latticeShufflePermutahedron}.
\end{proof}

\begin{remark}
\label{rem:notSemidistributive}
In contrast to the weak order on~$\f{S}_m$ and the Tamari lattice on~$\f{B}_n$, the lattice of \cref{prop:graphMultiplihedron} is not a lattice quotient of the weak order and is not even semidistributive when~$m \ge 1$ and~$n \ge 3$.
\end{remark}

\begin{remark}
\label{rem:graphMultiplihedra}
Similarly, the shuffle of an $m$-permutahedron with a graph associahedron is a generalization of the graph multiplihedron of~\cite{DevadossForcey}.
It follows from \cref{coro:latticeShufflePermutahedron} that the resulting rotation graph is a lattice as soon as the graph associahedron is a lattice (necessary and sufficient conditions for the latter are discussed in~\cite{BarnardMcConville}).
\end{remark}


\subsection{Vertex, facet, and Minkowski sum descriptions}
\label{subsec:vertexFacetMinkowskiDescriptionsMultiplihedra}

Our next three statements, illustrated in \cref{fig:coordinatesPaintedTrees,fig:inequalitiesPaintedTrees}, provide the vertex, facet, and Minkowski sum descriptions of the $(m,n)$-multi\-plihedron $\Multiplihedron$.
The proofs are elementary computations from \mbox{\cref{def:permutahedron,def:associahedron,def:graphicalZonotope,def:multiplihedron}}.

\begin{proposition}
\label{prop:verticesMultiplihedron}
For any~$i \in [m+n]$, the $i$-th coordinate of the vertex of the $(m,n)$-multiplihedron $\Multiplihedron$ corresponding to a binary $m$-painted $n$-tree is given by
\begin{itemize}
\item if~$i \le m$, the number of binary nodes and cuts weakly below the cut labeled by~$i$,
\item if~$i \ge m+1$, the number of cuts below~$\node{n}$ plus the product of the numbers of leaves in the left and right subtrees of~$\node{n}$, where~$\node{n}$ is the node labeled by~$i-m$ in inorder.
\end{itemize}
In particular, the sum of the coordinates is always~$\binom{m+1}{2} + \binom{n+1}{2} + mn = \binom{m+n+1}{2}$.
\end{proposition}

\pagebreak

\begin{proposition}
\label{prop:facetsMultiplihedron}
Let~$\PT \eqdef (T, C, \mu)$ be an $m$-painted $n$-tree of rank~$m+n-2$.
Let~$A$ be the elements of~$[m]$ which label a cut not containing the root of~$T$ ($A = \varnothing$ if~$C$ has only one cut, which contains the root of~$T$), and~$B \eqdef B_1 \cup \dots \cup B_k$ where~$B_1, \dots, B_k$ are the inorder labels shifted by~$m$ of the non-unary nodes of~$T$ distinct from the root of~$T$.
Then the facet of the $(m,n)$-multiplihedron~$\Multiplihedron$ corresponding to~$\PT$ is defined by the inequality
\[
\dotprod{\b{x}}{\one_{A \cup B}} \ge \binom{|A|+1}{2} + \binom{|B_1|+1}{2} + \dots + \binom{|B_k|+1}{2} + |A| \cdot |B|.
\]
Moreover, this inequality is a facet defining inequality of the permutahedron~$\Perm[m+n]$ if and only if~$k \le 1$, that is, if~$T$ has at most two non-unary nodes.
\end{proposition}

\begin{proposition}
\label{prop:MinkowskiMultiplihedron}
The $(m,n)$-multiplihedron~$\Multiplihedron$ is the Minkowski sum of the faces \linebreak ${\simplex_I \eqdef \conv\set{\b{e}_i}{i \in I}}$ of the standard simplex~$\simplex_{[m+n]}$ corresponding to all subsets~$I \subseteq [m+n]$ such that~$|I| \le 2$ and~$|I \cap [n]^{+m}| \le 1$, or~$I$ is a subinterval of~$[n]^{+m}$.
\end{proposition}

\begin{example}
\cref{fig:coordinatesPaintedTrees} illustrates some vertex coordinates of $\Multiplihedron[3][3]$ computed by \cref{prop:verticesMultiplihedron} and \cref{fig:inequalitiesPaintedTrees} illustrates some facet inequalities of $\Multiplihedron[3][3]$ computed by \cref{prop:facetsMultiplihedron}. Note that all vertices of $\Multiplihedron[3][3]$ have coordinate sum~$21$. Note that for any pair $(i, j) \in \{(1, 1), (1, 2), (2, 2), (2, 3), (3, 2), (3, 3), (4, 4), (5, 4), (5, 5)\}$, we have~$\PT_i$ smaller than~$\PT[S]_j$ in deletion order, so that the vertex corresponding to~$\PT_i$ is contained in the facet corresponding to~$\PT[S]_j$.

\begin{figure}
	\centerline{
	\begin{tabular}{c@{}c@{}c@{}c@{}c}
		$\PT_1$
		&
		$\PT_2$
		&
		$\PT_3$
		&
		$\PT_4$
		&
		$\PT_5$
		\\[-.3cm]
		\paintedTree[1.2]{[3[[1[2[[[][]][]]]][1[2[]]]]]}{1,2,3}
		&
		\paintedTree[1.2]{[[3[[2[1[]]][2[[1[]][1[]]]]]][3[2[1[]]]]]}{1,2,3}
		&
		\paintedTree[1.2]{[[3[[1[[2[]][2[]]]][1[2[]]]]][3[1[2[]]]]]}{1,2,3}
		&
		\paintedTree[1.2]{[3[[1[[2[]][2[]]]][1[[2[]][2[]]]]]]}{1,2,3}
		&
		\paintedTree[1.2]{[[3[1[[2[]][2[]]]]][3[1[[2[]][2[]]]]]]}{1,2,3}
		\\[.1cm]
		$(4, 3, 6, 1, 2, 5)$
		&
		$(1, 3, 5, 4, 2, 6)$
		&
		$(3, 1, 5, 2, 4, 6)$
		&
		$(4, 1, 6, 2, 6, 2)$
		&
		$(4, 1, 5, 2, 7, 2)$
	\end{tabular}
	}
	\caption{Vertices of $\Multiplihedron[3][3]$ corresponding to five binary $3$-painted $3$-trees.}
	\label{fig:coordinatesPaintedTrees}
\end{figure}

\begin{figure}
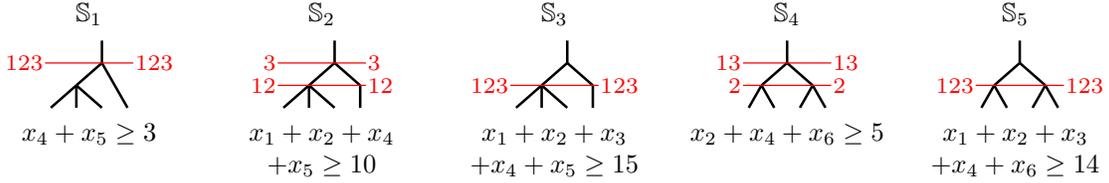

	\centerline{
	\begin{tabular}{c@{}c@{}c@{}c@{}c}
		$\PT[S]_1$
		&
		$\PT[S]_2$
		&
		$\PT[S]_3$
		&
		$\PT[S]_4$
		&
		$\PT[S]_5$
		\\[-.3cm]
		\paintedTree[1.2]{[123[[][][]][]]}{123}
		&
		\paintedTree[1.2]{[3[12[][][]][12[]]]}{12,3}
		&
		\paintedTree[1.2]{[[123[][][]][123[]]]}{123}
		&
		\paintedTree[1.2]{[13[2[][]][2[][]]]}{13,2}
		&
		\paintedTree[1.2]{[[123[][]][123[][]]]}{123}
		\\[.1cm]
		$x_4 + x_5 \ge 3$
		&
		$x_1 + x_2 + x_4$
		&
		$x_1 + x_2 + x_3$
		&
		$x_2 + x_4 + x_6 \ge 5$
		&
		$x_1 + x_2 + x_3$
		\\
		&
		$+ x_5 \ge 10$
		&
		$+ x_4 + x_5 \ge 15$
		&
		&
		$+ x_4 + x_6 \ge 14$
	\end{tabular}
	}
	\caption{Facet defining inequalities of $\Multiplihedron[3][3]$ corresponding to five rank $4$ $3$-painted $3$-trees.}
	\label{fig:inequalitiesPaintedTrees}
\end{figure}
\end{example}


\subsection{Numerology}
\label{subsec:numerologyMultiplihedra}

We now present enumerative results on the number of vertices, faces and facets of the $(m,n)$-multiplihedron~$\Multiplihedron$, using standard techniques from generating functionology~\cite{FlajoletSedgewick}.
The first few values of these numbers are collected in \cref{table:verticesMultiplihedra,table:facetsMultiplihedra,table:facesMultiplihedra} in \cref{subsec:tablesMultiplihedra}.
We start with vertices, which appear as \OEIS{A158825} in~\cite{OEIS} up to a factorial factor, generalizing the formula of~\cite[Thm.~3.1]{Forcey-multiplihedra}.
See~\cref{table:verticesMultiplihedra}.

\begin{proposition}
\label{prop:numberVerticesMultiplihedra}
The number of vertices of the $(m,n)$-multiplihedron~$\Multiplihedron$ (equivalently, of binary $m$-painted $n$-trees) is
\[
m! \, [y^{n+1}] \, \CGF^{(m+1)}(y),
\]
where~$[y^{n+1}]$ selects the coefficient of~$y^{n+1}$, and $\CGF^{(i)}(y)$ is defined for~$i \ge 1$ by
\[
\CGF^{(1)}(y) \eqdef \CGF(y)
\qquad\text{and}\qquad
\CGF^{(i+1)}(y) \eqdef \CGF \big( \CGF^{(i)}(y) \big),
\]
where
\[
\CGF(y) = \frac{1-\sqrt{1-4y}}{2}
\]
is the Catalan generating function (see \cref{prop:CatalanSchroderGF}).
\end{proposition}

\begin{proof}
According to \cref{prop:paintedTreeDeletion,prop:faceLatticeMultiplihedron}, we need to count the binary $m$-painted $n$-trees.
We construct a binary $m$-painted tree by
\begin{itemize}
\item choosing a binary tree~$T$ above the topmost cut (thus the apparition of~$\CGF$),
\item grafting at each leaf of~$T$ a binary tree with $m-1$ cuts (thus the substitution of the~$y$ variable in~$\CGF$),
\item choosing the permutation of~$[m]$ that will label the $m$ cuts (thus the factor $m!$).
\qedhere
\end{itemize}
\end{proof}

We now consider the number of facets of the $(m,n)$-multiplihedron~$\Multiplihedron$, generalizing the formula of~\cite[Thm.~2.4]{Forcey-multiplihedra}.
See~\cref{table:facetsMultiplihedra}.

\begin{proposition}
\label{prop:numberFacetsMultiplihedra}
The number of facets of the $(m,n)$-multiplihedron~$\Multiplihedron$ is
\[
 \binom{n+1}{2} - 1 + 2^{m+n} - 2^n.
\]
\end{proposition}

\begin{proof}
According to \cref{prop:paintedTreeDeletion,prop:faceLatticeMultiplihedron}, there are two types of $m$-painted $n$-trees corresponding to facets of the $(m,n)$-multiplihedron~$\Multiplihedron$:
\begin{itemize}
\item those where the bottommost cut contains the root: this amounts to choose a corolla with $n+1$ leaves, thus $\binom{n+1}{2}-1$ choices,
\item those where the bottommost cut contains all children of the root: this amounts to choose a non-empty subset of~$[m]$ to label this bottommost cut (the complement, if non-empty, will label the topmost cut containing the root), and a subset of~$[n]$ for the inorder label of the root, thus $(2^m-1) 2^n$ choices.
\qedhere
\end{itemize}
\end{proof}

Finally, adapting the approach of \cref{prop:numberVerticesMultiplihedra}, we can count all faces of the $(m,n)$-multipli\-hedron~$\Multiplihedron$ according to their dimension.

\begin{proposition}
\label{prop:numberFacesMultiplihedra}
Let~$PT(m,n,p)$ denote the number of $p$-dimensional faces of the $(m,n)$-multipli\-hedron~$\Multiplihedron$, or equivalently the number of $m$-painted $n$-trees of rank~$p$.
Then the generating function~$\PTGF(x,y,z) \eqdef \sum_{m,n,p} PT(m,n,p) \, x^m \, y^n \, z^p$ is given by
\[
\PTGF(x,y,z) = \sum_m x^m \sum_{k = 0}^m \SGF \big( \tSGF^{(k)}(y,z), z \big) \, \surjections{m}{k} \, z^{m-k},
\]
where~$\surjections{m}{k}$ is the number of surjections from~$[m]$ to~$[k]$,
\[
\SGF(y,z) = \frac{1+yz-\sqrt{1-4y-2yz+y^2z^2}}{2(z+1)}
\]
is the Schr\"oder generating function (see \cref{prop:CatalanSchroderGF}), and~$\tSGF^{(i)}(y,z)$ is defined for~$i \ge 0$~by
\[
\tSGF^{(0)}(y,z) \eqdef y,
\quad
\tSGF^{(1)}(y,z) \eqdef (1+z) \, \SGF(y,z) - yz
\quad\text{and}\quad
\tSGF^{(i+1)}(y,z) \eqdef \tSGF^{(i)} \big( \tSGF^{(1)}(y,z), z \big).
\]
\end{proposition}

\begin{proof}
According to \cref{prop:paintedTreeDeletion,prop:faceLatticeMultiplihedron}, we need to count the $m$-painted $n$-trees of rank~$p$.
We count them according to their number~$k$ of cuts.
For~$k = 0$, we just obtain the Schr\"oder generating function~$\SGF(y,z)$ multiplied by~$z^m$ to take the rank shift into account.
For~$k \ge 1$, we construct an $m$-painted tree with $k$ cuts by
\begin{itemize}
\item choosing a Schr\"oder tree~$S$ above the topmost cut (thus the apparition of~$\SGF$), 
\item grafting at each leaf of~$S$ a Schr\"oder tree with~$k-1$ cuts (thus the substitution of the $y$ variable in~$\SGF$), whose root may or may not lie on the topmost cut (explaining the twist from~$\SGF(y,z)$ to~$(1+z) \, \SGF(y,z) - yz$),
\item choosing the ordered partition of~$[m]$ that will label the $k$ cuts (thus the factor $\surjections{m}{k}$).
\end{itemize}
Finally, since an $m$-painted $n$-tree~$(T,C,\mu)$ yields a monomial~$y^n z^{n - |T| + |\bigcup C|}$ in the generating function~$\SGF \big( \tSGF^{(k)}(y,z), z \big)$, we multiply by the factor~$z^{m-k}$ to take into account~$m - |C|$ in the definition of the rank~$\rank(T,C) \eqdef m + n - |T| - |C| + |\bigcup C|$.
\end{proof}

We derive from \cref{prop:numberFacesMultiplihedra} the total number of faces of the $(m,n)$-multiplihedron~$\Multiplihedron$.
See~\cref{table:facesMultiplihedra}.

\begin{proposition}
\label{prop:numberFacesMultiplihedra2}
The number of faces of the $(m,n)$-multiplihedron~$\Multiplihedron$ (equivalently, of $m$-painted $n$-trees) is
\[
\sum_{k = 0}^m \surjections{m}{k} \, [y^{n+1}] \SGF \big( \tSGF^{(k)}(y) \big)
\]
where~$\surjections{m}{k}$ is the number of surjections from~$[m]$ to~$[k]$,
\[
\SGF(y) = \frac{1+y-\sqrt{1-6y+y^2}}{4}
\]
is the Schr\"oder generating function (see \cref{prop:CatalanSchroderGF}), and~$\tSGF^{(i)}(y)$ is defined for~$i \ge 0$~by
\[
\tSGF^{(0)}(y) \eqdef y,
\quad
\tSGF^{(1)}(y) \eqdef 2 \, \SGF(y) - y
\quad\text{and}\quad
\tSGF^{(i+1)}(y) \eqdef \tSGF^{(i)} \big( \tSGF^{(1)}(y) \big).
\]
\end{proposition}

For instance, the $f$-vectors of all multiplihedra~$\Multiplihedron$ with~$m + n \le 5$ are displayed in \cref{fig:multiplihedra2,fig:multiplihedra3,fig:multiplihedra4}.
The $f$-vector of the $(3,3)$-multiplihedron~$\Multiplihedron[3][3]$ is
\[
f(\Multiplihedron[3][3]) = (1, 660, 1668, 1467, 518, 61, 1).
\]


\pagebreak

\section{Constrainahedra}
\label{sec:constrainahedra}

In this section, we study the family of $(m,n)$-constrainahedra, obtained as the shuffle of an $m$-associahedron~$\Asso[m]$ with an $n$-associahedron~$\Asso[n]$.
A first description of the constrainahedra appeared in~\cite{Tierney}, based on a private communication from N.~Bottman.
The rigorous description in terms of good rectangular preorders was worked out in~\cite{Poliakova, BottmanPoliakova}, where the constrainahedra are also realized as convex polytopes.
We provide the alternative combinatorial model of cotrees (\cref{subsec:cotrees}), describe the face lattice, fan and oriented skeleton of the $(m,n)$-constrainahedron in terms of these cotrees (\cref{subsec:constrainahedra}), provide explicit vertex, facet and Minkowski sum descriptions of the $(m,n)$-constrainahedron (\cref{subsec:vertexFacetDescriptionsConstrainahedra}), and present enumerative results on the number of vertices, faces and facets of the $(m,n)$-constrainahedron (\cref{subsec:numerologyConstrainahedra}).


\subsection{Cotrees}
\label{subsec:cotrees}

We start by defining cotrees, illustrated in \cref{fig:cotrees}.
Intuitively, a cotree is a pair of Schr\"oder trees both growing in the same direction (say down), drawn side to side, together with the information of the relative positions of their nodes.
Examples are illustrated in \cref{fig:cotrees}.

\begin{definition}
\label{def:cotree}
A \defn{$(m,n)$-cotree} is a triple~$\BT \eqdef (L, R, \mu)$, where $L$ is a Schr\"oder $m$-tree, $R$ is a Schr\"oder $n$-tree, and $\mu$ is an ordered partition of the nodes of~$L$ and~$R$~such~that
\begin{itemize}
\item the part of~$\mu$ containing a node~$\node{n}$ of~$L$ (resp.~$R$) distinct from the root is before or equal to the part of~$\mu$ containing the parent of~$\node{n}$,
\item no two consecutive parts of~$\mu$ are both contained in~$L$ or both contained in~$R$,
\item there is no oriented path in~$L$ (resp.~in~$R$) joining two nodes in a part of~$\mu$ which meets both~$L$ and~$R$.
\end{itemize}
We say that a part of~$\mu$ is of type~$\ell$, $r$ or~$b$ when it contains nodes from~$L$, $R$ or both~$L$ and~$R$, and we call \defn{type} of the cotree the word given by the sequence of types of the parts of~$\mu$.
We denote by~$\f{CT}_{m,n}$ the set of $(m,n)$-cotrees.
\end{definition}

To represent a $(m,n)$-cotree~$\CT \eqdef (L, R, \mu)$, we draw the two trees~$L$ and~$R$ side by side, and we mark the separations between the parts of~$\mu$ by (red) horizontal lines.
Note that~$\mu$ is read from bottom to top.
Examples are illustrated in \cref{fig:cotrees}.

\begin{figure}[b]
	\centerline{
		\cotree[1.2]{[[1[[2[3[[[4[[5[6[]]]]][4[5[6[]]]][4[5[6[]]]]][[4[5[[6[]][6[]]]]][4[5[6[]]]]]]]][2[[3[4[5[6[]]]]][3[4[5[[6[]][6[]]]]]][3[4[5[6[[][]]]]]]]]]]]}{[[1[[2[[3[[[4[[5[[6[]]]][5[6[]]]]]]]][3[4[[5[6[]]][5[6[]]][5[6[[][]]]]]]]]]]][1[2[[3[4[5[6[]]]]][3[4[5[6[]]]]]]]]]}{1,2,3,4,5,6}
		\hspace*{-.5cm}
		\cotree[1.2]{[[[1[[[2[[3[[4[5[[]]]]]][3[4[[5[]][5[]]]]]]]]]][1[2[3[4[5[]]]]]]][1[2[[[3[4[[5[]][5[]]]]][3[4[5[]]]]][[3[4[5[]]]][3[4[5[]]]]]]]]]}{[[[1[[2[[[3[[4[5[]]][4[[5[]]]]]]]]][[2[3[[4[5[[][]]]][4[5[]]]]]][2[3[4[5[[][]]]]]]]]]]]}{1,2,3,4,5}
	}
	\caption{A $(10,7)$-cotree of type~$r\ell r\ell b\ell r$ (left), a binary $(8,6)$-cotree of type~$r\ell r\ell r\ell$ (right).}
	\label{fig:cotrees}
\end{figure}

We now define the cotree deletion poset.
\cref{def:cotreeDeletion} provides a direct description in terms of cotrees, while \cref{def:cotreePreposet} provides an alternative simpler but indirect description in terms of preposets.
To illustrate the following definition, \cref{fig:cotreeDeletions} represents a sequence of deletions in $(7,5)$-cotrees.

\begin{figure}
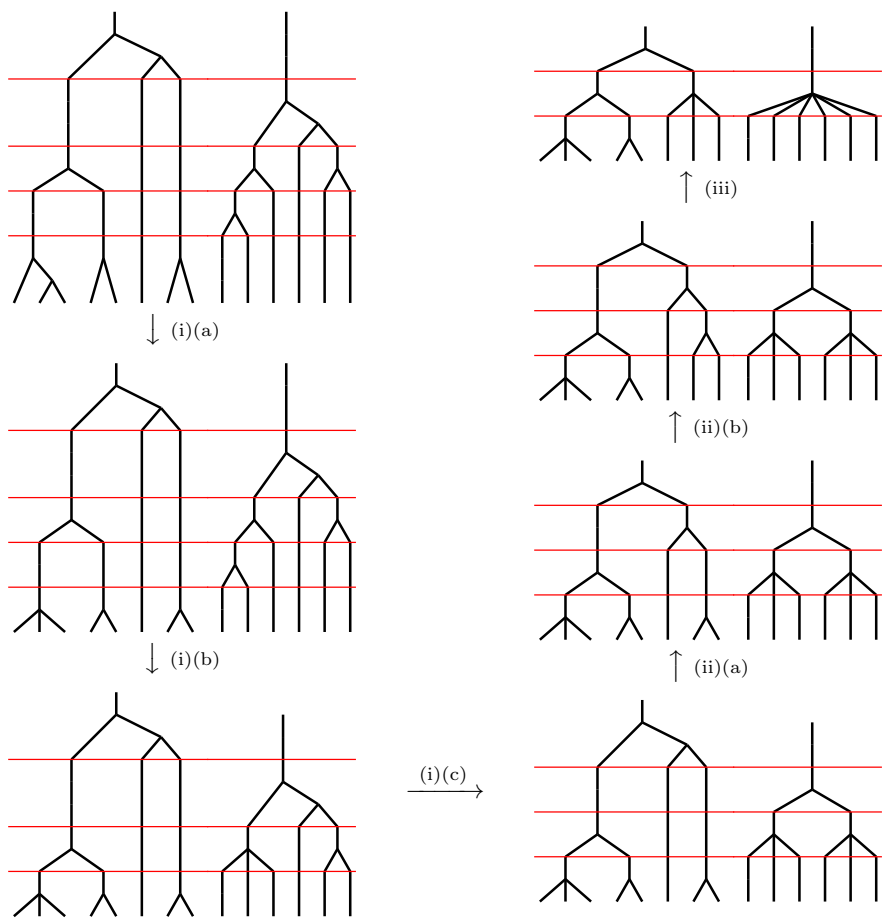

	\centerline{
		\begin{tabular}{c}
		\cotree[1.2]{[[1[[[2[[3[[4[[][[][]]]]]][3[4[[][]]]]]]]]][[1[2[3[4[]]]]][1[2[3[4[[][]]]]]]]]}{[[[1[[2[[3[[4[[[]]]][4[]]]][3[4[]]]]][[2[3[4[]]]][2[[3[4[]]][3[4[]]]]]]]]]]}{1,2,3,4} \\
		$\big\downarrow$ {\scriptsize (i)(a)} \\[-.2cm]
		\cotree[1.2]{[[1[[[2[[3[[4[[][][]]]]][3[4[[][]]]]]]]]][[1[2[3[4[]]]]][1[2[3[4[[][]]]]]]]]}{[[[1[[2[[3[[4[[]]][4[]]]][3[4[]]]]][[2[3[4[]]]][2[[3[4[]]][3[4[]]]]]]]]]]}{1,2,3,4} \\
		$\big\downarrow$ {\scriptsize (i)(b)} \\[-.2cm]
		\cotree[1.2]{[[1[[[2[[3[[][][]]][3[[][]]]]]]]][[1[2[3[]]]][1[2[3[[][]]]]]]]}{[[1[[2[[3[[]]][3[]][3[]]]][[2[3[]]][2[[3[]][3[]]]]]]]]}{1,2,3} \\
		\end{tabular}
		\hspace{-.3cm}\raisebox{-4.5cm}{$\xrightarrow{\text{ (i)(c) }}$}\hspace{-.3cm}
		\begin{tabular}{c}
		\cotree[1.2]{[[1[[2[[][][]]][2[[][]]]]][1[[2[]][2[]][2[]]]]]}{[[1[[2[[]]][2[]][2[]][2[]][2[]][2[]]]]]]}{1,2} \\
		$\big\uparrow$ {\scriptsize (iii)} \\[-.2cm]
		\cotree[1.2]{[[1[[2[[3[[][][]]][3[[][]]]]]]][1[[2[3[]]][2[[3[]][3[]]]]]]]}{[[1[[2[[3[[]]][3[]][3[]]]][2[[3[]][3[]][3[]]]]]]]]}{1,2,3} \\
		$\big\uparrow$ {\scriptsize (ii)(b)} \\[-.2cm]
		\cotree[1.2]{[[1[[2[[3[[][][]]][3[[][]]]]]]][1[[2[3[]]][2[3[[][]]]]]]]}{[[1[[2[[3[[]]][3[]][3[]]]][2[[3[]][3[]][3[]]]]]]]]}{1,2,3} \\
		$\big\uparrow$ {\scriptsize (ii)(a)} \\[-.2cm]
		\cotree[1.2]{[[1[[2[[3[[][][]]][3[[][]]]]]]][[1[2[3[]]]][1[2[3[[][]]]]]]]}{[[1[[2[[3[[]]][3[]][3[]]]][2[[3[]][3[]][3[]]]]]]]]}{1,2,3} \\
		\end{tabular}
	}
	\caption{Deletions in $(7,5)$-cotrees.}
	\label{fig:cotreeDeletions}
\end{figure}

\begin{definition}
\label{def:cotreeDeletion}
Let~$\CT \eqdef (L, R, \mu)$ and $\CT' \eqdef (L', R', \mu')$ be two $(m,n)$-cotrees.
We say that~$\CT'$ is obtained by a \defn{deletion} in~$\CT$ in either of the following three cases:
\begin{enumerate}[(i)]
\item \textbf{Node deletion:} $L'$ (resp.~$R'$) is obtained by deleting a node~$\node{n}$ with parent~$\node{p}$ in~$L$ (resp.~$R$) in the following situations:
	\begin{enumerate}[(a)]
	\item $\node{n}$ and $\node{p}$ belong to the same part of~$\mu$, then~$\mu'$ is obtained by deleting~$\node{n}$ from~$\mu$,
	\item  the part of~$\node{n}$ is of type~$\ell$ (resp.~$r$), the part of~$\node{p}$ is of type~$b$, and the parts of~$\node{n}$ and~$\node{p}$ are consecutive, then~$\mu'$ is obtained by deleting~$\node{n}$ from~$\mu$,
	\item the part of~$\node{n}$ is of type~$b$, the part of~$\node{p}$ is of type~$\ell$ (resp.~$r$), and the parts of~$\node{n}$ and~$\node{p}$ are consecutive, then~$\mu'$ is obtained from~$\mu$ by moving~$\node{p}$ to the part of~$\node{n}$ and deleting~$\node{n}$.
	\end{enumerate}
\item \textbf{Nodes move:} $L' = L$, $R' = R$, and~$\mu'$ is obtained from~$\mu$ by 
	\begin{enumerate}[(a)]
	\item either creating, in between two consecutive parts~$\mu_i$ of type~$\ell$ (resp.~$r$) and~$\mu_{i+1}$ of type~$r$ (resp.~$\ell$), a new part containing a node of~$\mu_i$ whose children are not in~$\mu_i$ and a node of~$\mu_{i+1}$ whose parent is not in~$\mu_{i+1}$ (and removing these nodes from their original parts~in~$\mu$),
	\item or moving a node~$\node{n}$ from its part~$\mu_i$ to the previous (or next) part~$\mu_{i \pm 1}$ in~$\mu$, provided that the part $\mu_i$ is not of type~$b$, that the part~$\mu_{i \pm 1}$ is of type~$b$, and that the parent (or children) of~$\node{n}$ does not belong to~$\mu_i \cup \mu_{i \pm 1}$,
	\end{enumerate}
\item \textbf{Twin parts merge:} $\mu'$ is obtained by merging two consecutive parts of~$\mu$ of type~$b$, and~$L'$ (resp.~$R'$) is obtained by deleting any node~$\node{n}$ in~$L$ (resp.~$R$) such that both~$\node{n}$ and its parent belong to these parts.
\end{enumerate}
\end{definition}

\pagebreak

\begin{proposition}
\label{prop:cotreeDeletion}
For all integers~$m, n \ge 0$, the set~$\f{CT}_{m,n}$ is stable by deletion, and the deletion graph is the Hasse diagram of a poset ranked by~$\rank(L, R, \mu) = m + n - |L| - |R| + \beta(\mu)$, where~$\beta(\mu)$ is the sum of~$|\mu_i|-1$ over all parts~$\mu_i$ of~$\mu$ with~$\mu_i \cap L \ne \varnothing \ne \mu_i \cap R$.
In particular a $(m,n)$-cotree~$\BT \eqdef (L, R, \mu)$ has
\begin{itemize}
\item rank~$0$ if and only if both~$L$ and~$R$ are binary trees, and no part of~$\mu$ meets both~$L$ and~$R$,
\item rank~$m+n-2$ if and only if~$\mu$ has two parts, and each part of~$\mu$ either meets both~$L$ and~$R$ or contains a single node,
\item rank~$m+n-1$ if and only if~$\mu$ has a single part (hence, both~$L$ and~$R$ have a single node).
\end{itemize}
\end{proposition}

\begin{proof}
Consider a deletion transforming~$\CT \eqdef (L, R, \mu)$ to~$\CT' \eqdef (L', R', \mu')$.
Then~$\BT'$ is clearly a $(m,n)$-cotree since~$L'$ and~$R'$ are still Schr\"oder trees, and the partition~$\mu'$ fulfills the conditions of \cref{def:cotree}.
For the rank, we distinguish three cases corresponding to that of \cref{def:cotreeDeletion}:
\begin{enumerate}[(i)]
\item \textbf{Node deletion:} $|L'| + |R'| = |L| + |R| - 1$ while~$\beta(\mu') = \beta(\mu')$.
\item \textbf{Nodes move:} $|L'| = |L|$, $|R'| = |R|$, while~$\beta(\mu') = \beta(\mu) + 1$.
\item \textbf{Twin parts merge:} if~$\delta$ denotes the number of nodes~$\node{n}$ of~$L$ and $R$ such that both~$\node{n}$ and its parent belong to the merged parts of~$\mu$, then $|L'| + |R'| = |L| + |R| - \delta$ and~$\beta(\mu') = \beta(\mu) - \delta + 1$.
\end{enumerate}
In all three situations, we get~$\rank(\CT') = \rank(\CT)+1$.
The end of the statement immediately follows.
\end{proof}

\begin{definition}
\label{def:cotreeDeletionPoset}
The \defn{$(m,n)$-cotree deletion poset} is the poset on $\f{CT}_{m,n}$ where a $(m,n)$-cotree is covered by all $(m,n)$-cotrees that can be obtained by a deletion.
\end{definition}

The $(m,n)$-cotree deletion poset can alternatively be defined using preposets.

\begin{definition}
\label{def:cotreePreposet}
A $(m,n)$-cotree~$\CT \eqdef (L, R, \mu)$ defines a preposet~$\preccurlyeq_{\CT}$ on~$[m+n]$ that can be read as follows.
Label~$L$ by~$[m]$ in inorder and $R$ by~$[n]^{+m}$ in inorder (shifted by~$m$).
Then, for any~$i,j \in [m+n]$, we have~$i \preccurlyeq_{\CT} j$ if the part of~$\mu$ containing~$i$ is before the part of~$\mu$ containing~$j$, or if there is a (possibly empty) path from the node containing~$i$ to the node containing~$j$ in the tree~$L$ or in the tree~$D$ oriented towards their roots.
\end{definition}

\begin{proposition}
\label{prop:characterizationCotreePreposets}
The preposets~$\preccurlyeq_{\CT}$ for~$\CT \in \f{CT}_{m,n}$ are precisely the preposets~$\preccurlyeq$ on~$[m+n]$ in which any~$1 \le i < k \le m+n$ are comparable (\ie~$i \preccurlyeq k$ or $i \succcurlyeq k$ or both) if and only if
\begin{itemize}
\item either~$i \le m < k$,
\item or~$m < i$ (resp.~$k \le m$) and at least one of the following holds:
	\begin{itemize}
	\item there exists no~$i < j < k$ such that $i \prec j \succ k$,
	\item there exists~$j \in [m]$ (resp.~$j \in [n]^{+m}$) such that $i \preccurlyeq j \preccurlyeq k$ or $i \succcurlyeq j \succcurlyeq k$.
	\end{itemize}
\end{itemize}
\end{proposition}

\begin{proof}
Any preposet~$\preccurlyeq_{\CT}$ clearly satisfies these conditions.
Conversely, given a preposet~$\preccurlyeq$ on~${[m+n]}$ satisfying these conditions, consider 
\begin{itemize}
\item the preposet~$\preccurlyeq_\ell$ on~$[m]$ defined by~$i \preccurlyeq_\ell k$ if and only if $i \preccurlyeq k$ and there is no~$i < j < k$ such that~$i \prec j \succ k$,
\item the preposet~$\preccurlyeq_r$ on~$[n]$ defined by~$i \preccurlyeq_r k$ if and only if ${i + m \preccurlyeq k+m}$ and there is no~${i < j < k}$ such that~$i + m \prec j + m \succ k + m$.
\end{itemize}
The preposet~$\preccurlyeq_\ell$ (resp.~$\preccurlyeq_r$) is clearly the preposet~$\preccurlyeq_L$ (resp.~$\preccurlyeq_R$) of a Schr\"oder $m$-tree~$L$ (resp.~a Schr\"oder $n$-tree~$R$).
We then obtain the partition~$\mu$ by considering the relations~$i \preccurlyeq k$ with~${i \le m < k}$.
Details are left to the reader.
\end{proof}

\begin{proposition}
\label{prop:cotreeDeletionPosetOnPreposets}
In the cotree deletion poset, $\CT$ is smaller than~$\CT'$ if and only if~$\preccurlyeq_{\CT}$ refines~$\preccurlyeq_{\CT'}$.
\end{proposition}

\begin{proof}
An immediate case analysis shows that deletions in a cotree~$\CT$ defined in \cref{def:cotreeDeletion} precisely translate all possible refinements in the corresponding preposet~$\preccurlyeq_{\CT}$.
\end{proof}

\enlargethispage{.3cm}
Finally, we define the rotations in cotrees, which correspond to rank~$1$ cotrees.
To illustrate the following definition, \cref{fig:cotreeRotations} represents a sequence of right rotations in binary $(3,2)$-cotrees.

\begin{figure}
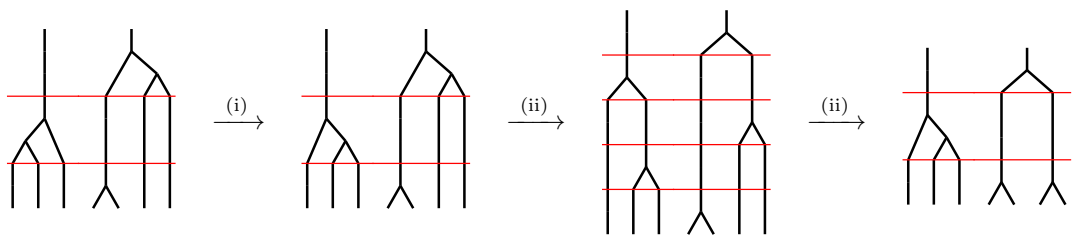

	\centerline{
		\cotree[1.2]{[[[1[[[2[[]]][2[]]][2[]]]]]]}{[[1[[[2[[][]]]]]][[1[2[]]][1[2[]]]]]}{1,2}
		\hspace{-.3cm}\raisebox{-1.5cm}{$\xrightarrow{\text{ (i) }}$}\hspace{-.3cm}
		\cotree[1.2]{[[[1[[2[[]]][[2[]][2[]]]]]]]}{[[1[[[2[[][]]]]]][[1[2[]]][1[2[]]]]]}{1,2}
		\hspace{-.3cm}\raisebox{-1.5cm}{$\xrightarrow{\text{ (ii) }}$}\hspace{-.3cm}
		\raisebox{.25cm}{\cotree[1.2]{[[1[[2[[3[4[[]]]]]][2[3[[4[]][4[]]]]]]]]}{[[1[[2[[3[[4[[][]]]]]]]]][1[2[[3[4[]]][3[4[]]]]]]]}{1,2,3,4}}
		\hspace{-.3cm}\raisebox{-1.5cm}{$\xrightarrow{\text{ (ii) }}$}\hspace{-.3cm}
		\raisebox{-.25cm}{\cotree[1.2]{[[1[[2[[]]][[2[]][2[]]]]]]}{[[1[[[2[[][]]]]]][1[2[[][]]]]]}{1,2}}
	}
	\caption{Right rotations in binary $(3,2)$-cotrees.}
	\label{fig:cotreeRotations}
\end{figure}

\begin{definition}
\label{def:rotationsCotrees}
We call \defn{binary $(m,n)$-cotrees} the rank~$0$ $(m,n)$-cotrees, \ie where both~$L$ and~$R$ are binary trees, and no part of~$\mu$ meets both~$L$ and~$R$.
We say that two binary $(m,n)$-cotrees~$\CT \eqdef (L, R, \mu)$ and $\CT' \eqdef (L', R', \mu')$ are connected by a \defn{right rotation} if:
\begin{enumerate}[(i)]
\item \textbf{Edge rotation:} $L'$ (resp.~$R'$) is obtained from~$L$ (resp.~$R$) by the right rotation of an edge whose endpoints belong to the same part of~$\mu$,
\item \textbf{Twin parts:} $L' = L$, $R' = R$, and $\mu'$ is obtained from~$\mu$ by creating, in between two consecutive parts~$\mu_i$ of type~$\ell$ and~$\mu_{i+1}$ of type~$r$, first a new part containing a node of~$\mu_{i+1}$ whose children are not in~$\mu_{i+1}$, and second a new part containing a node of~$\mu_i$ whose children are not in~$\mu_i$ (and removing these nodes from their original parts in~$\mu$, and merging consecutive parts of the same type~$\ell$ or~$r$ if any).
\end{enumerate}
\end{definition}

\begin{remark}
\label{rem:algebraicInterpretationConstrainahedra}
The $(m,n)$-cotrees are algebraically motivated by the $(m,n)$-constrainahedron defined in~\cite{Tierney,Poliakova,BottmanPoliakova} as a constrained version of the $2$-associahedra of~\cite{Bottman}.
This structure was already studied in details in particular in~\cite[Sect.~5]{Poliakova}, where
\begin{itemize}
\item the preposets of \cref{prop:characterizationCotreePreposets} are already described under the name ``good rectangular preorders'' in~\cite[Sect.~5.1.3]{Poliakova} and \cite[Sect.~2.1]{BottmanPoliakova},
\item an alternative combinatorial model is given by ``rectangular bracketings'' in~\cite[Sect.~5.1.3]{Poliakova} and \cite[Sect.~2.3]{BottmanPoliakova} (see \cref{fig:rectangularTubingsRotations} which illustrates the immediate bijection between binary $(m,n)$-cotrees and maximal $(m,n)$-bracketings),
\item the contraction poset on good rectangular preorders is proved to be a lattice in~\cite[Sect.~5.2]{Poliakova} and \cite[Sect.~3]{BottmanPoliakova}  (here, this property is a direct consequence of \cref{prop:faceLatticeConstrainahedron}),
\item a polytopal realization of this contraction poset is constructed in~\cite[Sect.~5.3]{Poliakova} and \cite[Sect.~4]{BottmanPoliakova} (which differs from our construction of \cref{subsec:constrainahedra} as discussed in \cite[Sect.~5]{BottmanPoliakova}).
\end{itemize}
Note that these realizations extend to higher dimension: bracketings of a $n_1 \times n_2 \times \dots \times n_d$ grid are naturally encoded by the shuffle of associahedra~$\Asso[n_1] \shuffleDP \Asso[n_2] \shuffleDP \dots \shuffleDP \Asso[n_d]$.

\begin{figure}[h]
	\centerline{
		\includegraphics[scale=1.1]{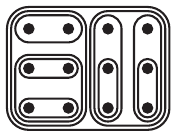}
		\raisebox{1cm}{$\longrightarrow$}
		\includegraphics[scale=1.1]{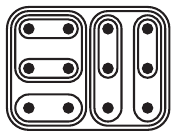}
		\raisebox{1cm}{$\longrightarrow$}
		\includegraphics[scale=1.1]{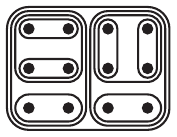}
		\raisebox{1cm}{$\longrightarrow$}
		\includegraphics[scale=1.1]{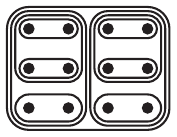}
	}
	\caption{$(3,2)$-rectangular bracketings, corresponding to the binary $(3,2)$-cotrees of \cref{fig:cotreeRotations}.}
	\label{fig:rectangularTubingsRotations}
\end{figure}
\end{remark}

\begin{remark}
\label{rem:algebraicInterpretationConstrainahedraBis}
A simple-minded algebraic interpretation of the binary $(m,n)$-cotrees involves two magmatic products $\bullet$ and $\circ$ on a set $X$.
The nodes in the left part of a cotree are associated with the product~$\bullet$, those in the right part with the product~$\circ$.
One starts at the bottom with a $(m+1) \times (n+1)$-matrix of elements of $X$ (with $m+1$ columns and $n+1$ rows).
Intermediate steps will go through rectangular $p \times q$-matrices of elements of $X$ with decreasing~$1 \le p \le m+1$ and~$1 \le q \le n+1$, until one reaches a $1 \times 1$-matrix of elements of $X$ at the top.
Going up through a node in the left part of the cotree means applying $\bullet$ to corresponding elements in two consecutive columns of the matrix, replacing these two columns by a single column and decreasing $p$ by $1$.
Similarly, going up through a node in the right part of the cotree means applying $\circ$ to corresponding elements in two consecutive rows of the matrix, replacing these two rows by a single row and decreasing $q$ by $1$.
In short, a left node stands for $\bullet$ merging two consecutive columns, and a right node for $\circ$ merging two consecutive rows.
\end{remark}


\subsection{Associahedra $\shuffleDP$ Associahedra}
\label{subsec:constrainahedra}

We now consider shuffles of associahedra with associahedra.

\begin{definition}
\label{def:constrainahedron}
The \defn{$(m,n)$-constrainahedron} is the polytope~$\Constrainahedron = \Asso[m] \shuffleDP \Asso[n]$.
\end{definition}

Note that since~$\Perm[1] = \Asso[1]$ and $\Perm[2] = \Asso[2]$, the first $(m,n)$-constrainahedron which is neither an associahedron, nor a $(m,n)$-multiplihedron, is the $(3,3)$-constrainahedron \linebreak $\Constrainahedron[3][3]$, which is a $5$-dimensional polytope.
There is thus no reasonable example to be drawn in this section.

\begin{proposition}
\label{prop:faceLatticeConstrainahedron}
The face lattice of the $(m,n)$-constrainahedron~$\Constrainahedron$ is isomorphic to the $(m,n)$-cotree deletion poset (augmented with a minimal element).
\end{proposition}

\begin{proof}
This follows from \cref{prop:shuffleFacePreposets} (see also \cref{rem:biPreposets}), since~$(m,n)$-cotrees are just a specialization of bipreposets.
\end{proof}

\begin{remark}
\label{rem:simpleConstrainahedron}
In contrast to the associahedron~$\Asso$, the constrainahedron~$\Constrainahedron$ is simple if and only if~$m = 0$, or~$n = 0$, or $\max(m,n) \le 2$.
\end{remark}

\begin{proposition}
\label{prop:fanConstrainahedron}
The normal fan of the $(m,n)$-constrainahedron~$\Constrainahedron$ is the fan containing one cone $\polytope{C}(\BT) \eqdef \set{\b{x} \in \R^{m+n}}{x_i \le x_j \text{ if } i \preccurlyeq_{\BT} j}$ for each~${\CT \in \f{CT}_{m,n}}$.
\end{proposition}

\begin{proof}
Immediate from \cref{prop:faceLatticeConstrainahedron,def:cotreePreposet}.
\end{proof}

\begin{proposition}
\label{prop:graphConstrainahedron}
When oriented in the direction~${\b{\omega} \eqdef (n,\dots,1) - (1,\dots,n) = \sum_{i \in [n]} (n+1-2i) \, \b{e}_i}$, the graph of the $(m,n)$-constrainahedron~$\Constrainahedron$ is isomorphic to the right rotation graph on binary $(m,n)$-cotrees.
\end{proposition}

\begin{proof}
It follows from \cref{prop:faceLatticeConstrainahedron} that the vertices of~$\Constrainahedron$ correspond to the binary $(m,n)$-cotrees.
It is easy to check that the edges of~$\Constrainahedron$ oriented by~$\b{\omega}$ correspond to right rotations on binary $(m,n)$-cotrees.
\end{proof}

\begin{remark}
\label{rem:noCoTamari}
In contrast to \cref{prop:graphMultiplihedron}, note that the right rotation graph on binary $(m,n)$-cotrees is not the Hasse diagram of a lattice when~$m \ge 3$ and~$n \ge 3$.
See \cref{fig:constrainahedronNotLattice} for examples of a pair of binary $(3,3)$-cotrees with no join and a pair of binary $(3,3)$-cotrees with no meet.
\begin{figure}[h]
	\centerline{
	\begin{tikzpicture}
		\node[label=west:{$\CT_1 =$}] (a) at (-4,4.3) {\cotree[1.2]{[[1[[2[3[[4[[5[[]]][5[]]]]]]]]][1[2[[3[4[5[]]]][3[4[5[]]]]]]]]}{[[1[[2[[3[4[[5[[][]]]]]]]][2[3[[4[5[[]]]][4[5[]]]]]]]]]}{1,2,3,4,5}};
		\node[label=east:{$= \CT_2$}] (b) at (4,4.3) {\cotree[1.2]{[[1[2[[3[4[5[]]]][3[4[5[]]]]]]][1[[2[3[[4[[5[[]]][5[]]]]]]]]]]}{[[1[[2[3[[4[5[[]]]][4[5[]]]]]][2[[3[4[[5[[][]]]]]]]]]]]}{1,2,3,4,5}};
		\node (c) at (-6,0) {\cotree[1.2]{[[1[[[2[[3[[]]][3[]]]]]]][1[2[[3[]][3[]]]]]]}{[[1[[2[3[[][]]]][[2[[3[]]]][2[3[]]]]]]]]]}{1,2,3}};
		\node (d) at (2,0) {\cotree[1.2]{[[1[2[[3[4[]]][3[4[]]]]]][1[[2[3[[4[[][]]]]]]]]]}{[[1[[2[[3[[4[[]]][4[]]]]]][2[3[[4[]][4[]]]]]]]]}{1,2,3,4}};
		\node (e) at (-2,0) {\cotree[1.2]{[[1[[2[3[[4[[][]]]]]]]][1[2[[3[4[]]][3[4[]]]]]]]}{[[1[[2[[3[[4[[]]][4[]]]]]][2[3[[4[]][4[]]]]]]]]}{1,2,3,4}};
		\node (f) at (6,0) {\cotree[1.2]{[[1[[[2[[3[[]]][3[]]]]]]][1[2[[3[]][3[]]]]]]}{[[1[[[2[[3[]]]][2[3[]]]][2[3[[][]]]]]]]]]}{1,2,3}};
		\node[label=west:{$\CT[S]_1 =$}] (g) at (-4,-4.6) {\cotree[1.2]{[[1[[[2[[3[[4[]]]][3[4[]]]]]]]][1[2[3[4[[][]]]]]]]}{[[1[[2[3[[4[[]]][4[]]]]][[2[[3[4[]]]]][2[3[4[]]]]]]]]]]}{1,2,3,4}};
		\node[label=east:{$= \CT[S]_2$}] (h) at (4,-4.6) {\cotree[1.2]{[[1[2[3[4[[][]]]]]][1[[[2[[3[[4[]]]][3[4[]]]]]]]]]}{[[1[[[2[[3[4[]]]]][2[3[4[]]]]][2[3[[4[[]]][4[]]]]]]]]]]}{1,2,3,4}};
		\draw[thick,gray] (a.south) -- (c.north);
		\draw[thick,gray] (a.south) -- (e.north);
		\draw[thick,gray] (b.south) -- (d.north);
		\draw[thick,gray] (b.south) -- (f.north);
		\draw[thick,gray] (c.south) -- (g.north);
		\draw[thick,gray] (d.south) -- (g.north);
		\draw[thick,gray] (e.south) -- (h.north);
		\draw[thick,gray] (f.south) -- (h.north);
	\end{tikzpicture}
	}
	\caption{Rotations on all $(3,3)$-cotrees larger than~$\CT[S]_1$ or~$\CT[S]_2$ and smaller than~$\CT_1$ or~$\CT_2$. This shows that~$\CT[S]_1$ and~$\CT[S]_2$ have no join, and~$\CT_1$ and~$\CT_2$ have no meet, so that the rotation graph on binary $(3,3)$-cotrees does not define a lattice.}
	\label{fig:constrainahedronNotLattice}
\end{figure}
\end{remark}


\subsection{Vertex, facet, and Minkowski sum descriptions}
\label{subsec:vertexFacetDescriptionsConstrainahedra}

Our next three statements, illustrated in \cref{fig:coordinatesCotrees,fig:inequalitiesCotrees}, provide the vertex, facet, and Minkowski sum descriptions of the $(m,n)$-constrai\-nahedron~$\Constrainahedron$.
The proofs are elementary computations from \mbox{\cref{def:associahedron,def:graphicalZonotope,def:constrainahedron}}.

\begin{proposition}
\label{prop:verticesConstrainahedron}
For any~$i \in [m+n]$, the $i$-th coordinate of the vertex of the $(m,n)$-constrainahedron $\Constrainahedron$ corresponding to a binary $(m,n)$-cotree~$(L, R, \mu)$ is given by
\begin{itemize}
\item if~$i \le m$, the product of the numbers of leaves in the left and right subtrees of~$\node{n}$, plus the number of nodes of~$R$ below~$\node{n}$, where~$\node{n}$ is the node of~$L$ labeled by~$i$ in inorder.
\item if~$i \ge m+1$, the product of the numbers of leaves in the left and right subtrees of~$\node{n}$, plus the number of nodes of~$L$ below~$\node{n}$, where~$\node{n}$ is the node of~$R$ labeled by~$i-m$ in inorder.
\end{itemize}
In particular, the sum of the coordinates is always~$\binom{m+1}{2} + \binom{n+1}{2} + mn = \binom{m+n+1}{2}$.
\end{proposition}

\begin{proposition}
\label{prop:facetsConstrainahedron}
Let~$\CT \eqdef (L, R, \mu)$ be a $(m,n)$-cotree of rank~$m+n-2$.
Let~${A \eqdef A_1 \cup \dots \cup A_k}$ where $A_1, \dots A_k$ are the inorder labels of the nodes of~$L$ located in the bottom part~$\mu_1$, and let ${B \eqdef B_1 \cup \dots \cup B_\ell}$ where~$B_1, \dots, B_\ell$ are the inorder labels shifted by~$m$ of the nodes of~$R$ located in the bottom part~$\mu_1$.
Then the facet of the $(m,n)$-constrainahedron~$\Constrainahedron$ corresponding to~$\CT$ is defined by the inequality
\[
\dotprod{\b{x}}{\one_{A \cup B}} \ge \sum_{i \in [k]} \binom{|A_i|+1}{2} + \sum_{j \in [\ell]} \binom{|B_j|+1}{2} + |A| \cdot |B|.
\]
Moreover, this inequality is a facet defining inequality of the permutahedron~$\Perm[m+n]$ if and only if~$k \le 1$ and~$\ell \le 1$, \ie if both $L$ and $R$ have at most two nodes.
\end{proposition}

\begin{proposition}
\label{prop:MinkowskiConstrainahedron}
The $(m,n)$-constrainahedron~$\Constrainahedron$ is the Minkowski sum of the faces \linebreak ${\simplex_I \eqdef \conv\set{\b{e}_i}{i \in I}}$ of the standard simplex~$\simplex_{[m+n]}$ corresponding to all subsets~$I \subseteq [m+n]$ such that~$|I \cap [m]| \le 1$ and $|I \cap [n]^{+m}| \le 1$,  or~$I$ is a subinterval of~$[m]$ or of~$[n]^{+m}$.
\end{proposition}

\begin{example}
\cref{fig:coordinatesCotrees} illustrates some vertex coordinates of $\Constrainahedron[3][3]$ computed by \cref{prop:verticesConstrainahedron} and \cref{fig:inequalitiesCotrees} illustrates some facet inequalities of $\Constrainahedron[3][3]$ computed by \cref{prop:facetsConstrainahedron}. Note that all vertices of $\Constrainahedron[3][3]$ have coordinate sum~$21$. Note that for any pair $(i, j) \in \{(1, 2), (1, 3), (2, 2), (2, 3), (2, 4), (3, 2), (3, 4), (4, 4)\}$, we have~$\CT_i$ smaller than~$\CT[S]_j$ in deletion order, so that the vertex corresponding to~$\CT_i$ is contained in the facet corresponding to~$\CT[S]_j$.

\begin{figure}[h]
	\centerline{
	\begin{tabular}{c@{}c@{}c@{}c}
		$\CT_1$
		&
		$\CT_2$
		&
		$\CT_3$
		&
		$\CT_4$
		\\[-.3cm]
		\cotree[1.2]{[[[1[[[2[[][]]]]]][1[2[]]]][1[2[]]]]}{[[[1[[[2[[]]][2[]]][[2[]][2[]]]]]]]}{1,2}
		&
		\cotree[1.2]{[[1[[[2[[][]]]]]][[1[2[]]][1[2[]]]]]}{[[[1[[[2[[]]][2[]]][[2[]][2[]]]]]]]}{1,2}
		&
		\cotree[1.2]{[[1[[2[[3[[4[[][]]]]]]]]][1[2[[3[4[]]][3[4[]]]]]]]}{[[1[[2[[3[[4[[]]][4[]]]]]][2[3[[4[]][4[]]]]]]]]}{1,2,3,4}
		&
		\cotree[1.2]{[[1[[[2[[][]]]]]][1[2[[][]]]]]}{[[1[[[2[[]]][2[]]][[2[]][2[]]]]]]}{1,2}
		\\[.1cm]
		$(1, 5, 6, 2, 5, 2)$
		&
		$(1, 7, 4, 2, 5, 2)$
		&
		$(1, 7, 3, 2, 6, 2)$
		&
		$(1, 7, 1, 3, 6, 3)$
	\end{tabular}
	}
	\caption{Vertices of $\Constrainahedron[3][3]$ corresponding to four binary $(3,3)$-cotrees.}
	\label{fig:coordinatesCotrees}
\end{figure}
\begin{figure}[h]
	\centerline{
	\begin{tabular}{c@{}c@{}c@{}c}
		$\CT[S]_1$
		&
		$\CT[S]_2$
		&
		$\CT[S]_3$
		&
		$\CT[S]_4$
		\\[-.3cm]
		\cotree[1.2]{[[1[[]]][1[]][1[]][1[]]]}{[[1[[][][][]]]]}{1}
		&
		\cotree[1.2]{[[1[[][]]][1[]][1[]]]}{[[1[[][]]][1[[][]]]]}{1}
		&
		\cotree[1.2]{[[1[[][]]][1[]][1[]]]}{[[1[[][][][]]]]}{1}
		&
		\cotree[1.2]{[[1[[][]]][1[[][]]]]}{[[1[[][][][]]]]}{1}
		\\[.1cm]
		$x_3 + x_4 + x_5 \ge 6$
		&
		$x_1 + x_4 + x_6 \ge 5$
		&
		$x_1 + x_4 + x_5$
		&
		$x_1 + x_3 + x_4$
		\\
		&
		&
		$+ x_6 \ge 10$
		&
		$+ x_5 + x_6 \ge 14$
	\end{tabular}
	}
	\caption{Facet defining inequalities of $\Constrainahedron[3][3]$ corresponding to four rank $4$ $(3,3)$-cotrees.}
	\label{fig:inequalitiesCotrees}
\end{figure}
\end{example}


\subsection{Numerology}
\label{subsec:numerologyConstrainahedra}

We now present enumerative results on the number of vertices, faces and facets of the $(m,n)$-constrainahedron~$\Constrainahedron$.
The first few values of these numbers are collected in \cref{table:verticesConstrainahedraBiassociahedra,table:facetsConstrainahedraBiassociahedra,table:facesConstrainahedra} in \cref{subsec:tablesConstrainahedraBiassociahedra}.
We start with vertices.
See~\cref{table:verticesConstrainahedraBiassociahedra}.

\begin{proposition}
\label{prop:numberVerticesConstrainahedra}
The number of vertices of the $(m,n)$-constrainahedron~$\Constrainahedron$ (equivalently of binary $(m,n)$-cotrees) is given by
\[
[x^{m+1} \, y^{n+1}] \, \sum_{i = 0}^{\min(m,n)} 2 \, \tCGF^{(i)}(x) \, \tCGF^{(i)}(y) + \tCGF^{(i)}(x) \, \tCGF^{(i+1)}(y) + \tCGF^{(i+1)}(x) \, \tCGF^{(i)}(y),
\]
where $\tCGF^{(i)}(x)$ is defined for~$i \ge 0$ by
\[
\tCGF^{(0)}(x) = x
\qquad\text{and}\qquad
\tCGF^{(i)}(x) = \tCGF^{(i-1)}(\CGF(x)) - \tCGF^{(i-1)}(x),
\]
where
\[
\CGF(x) = \frac{1-\sqrt{1-4x}}{2}
\]
is the Catalan generating function (see \cref{prop:CatalanSchroderGF}).
\end{proposition}

\begin{proof}
According to \cref{prop:cotreeDeletion,prop:faceLatticeConstrainahedron}, we need to count the binary $(m,n)$-cotrees.
We group them according to their type, which can be of the form~$(\ell r)^i$, $(r\ell)^i$, $(\ell r)^i\ell$ or~$(r\ell)^ir$.
We then need to construct the two binary trees~$L$ and~$R$ with compatible partitions of their nodes into~$i$ (or~$i+1$) parts.
We construct a partitioned binary tree with~$i+1$ parts by
\begin{itemize}
\item choosing a binary tree~$T$ for the first part (thus the apparition of~$\CGF$),
\item grafting at each leaf of~$T$ a partitioned binary tree with $i-1$ parts (thus the substitution of the~$y$ variable in~$\tCGF^{(i)}$), such that not all leaves of~$T$ are replaced by an empty binary tree (thus the subtraction of~$\tCGF^{(i-1)}$ in the definition of~$\tCGF^{(i)}$).
\qedhere
\end{itemize}
\end{proof}

We now consider the number of facets of the $(m,n)$-constrainahedron~$\Constrainahedron$.
See~\cref{table:facetsConstrainahedraBiassociahedra}.

\begin{proposition}
\label{prop:numberFacetsConstrainahedra}
The number of facets of the $(m,n)$-constrainahedron~$\Constrainahedron$ is
\[
 (2^m-1)(2^n-1) + \binom{m+1}{2} + \binom{n+1}{2} - 1.
\]
\end{proposition}

\begin{proof}
According to \cref{prop:cotreeDeletion,prop:faceLatticeConstrainahedron}, the possible types for the $(m,n)$-cotrees corresponding to facets of the $(m,n)$-constrainahedron~$\Constrainahedron$ are:
\begin{itemize}
\item type~$\ell r$ (resp.~type~$r\ell$): then both~$L$ and~$R$ have a single node, thus a single choice,
\item type~$b\ell$ (resp.~type~$rb$): then $L$ (resp.~$R$) is a non-trivial corolla while~$R$ (resp.~$L$) has a single node, thus~$\binom{m+1}{2}-1$ choices (resp.~$\binom{n+1}{2}-1$ choices),
\item type~$\ell b$ (resp.~type~$br$): then~$L$ (resp.~$R$) is any Schr\"oder tree of height~$2$ while~$R$ (resp.~$L$) has a single node, thus~$2^m-2$ choices (resp.~$2^n-2$ choices),
\item type~$bb$: then both~$L$ and~$R$ are Schr\"oder trees of height~$2$, thus~$(2^m-2)(2^n-2)$ choices.
\qedhere
\end{itemize}
\end{proof}

Finally, adapting the approach of \cref{prop:numberVerticesConstrainahedra}, we can count all faces of the $(m,n)$-constraina\-hedron~$\Constrainahedron$ according to their dimension.

\begin{proposition}
\label{prop:numberFacesConstrainahedra}
Let~$CT(m,n,p)$ denote the number of $p$-dimensional faces of the $(m,n)$-constraina\-hedron~$\Constrainahedron$, or equivalently the number of $(m,n)$-cotrees of rank~$p$.
Then the generating function~$\CTGF(x,y,z) \eqdef \sum_{m,n,p} CT(m,n,p) \, x^m \, y^n \, z^p$ is given by
\[
\CTGF(x,y,z) = \sum_w \SGF^w_u(x,z) \, \SGF^w_d(y,z)
\]
where
\begin{itemize}
\item $w$ runs over all words on the alphabet~$\{\ell,r,b\}$ with no two consecutive~$\ell$ nor two consecutive~$r$ and such that~$1 \le |w|_\ell + |w|_b \le m$ and~$1 \le |w|_r + |w|_b \le n$,
\item for a letter~$s \in \{\ell,r\}$, the generating function~$\SGF^w_s(y,z)$ is defined by~$\SGF^\varepsilon_s(y,z) \eqdef y$ and
\[
\SGF^{w}_s(y,z) \eqdef
\begin{cases}
\SGF^{w'}_s(\SGF(y,z), z) - \SGF^{w'}_s(y,z) & \text{if } w = sw', \\
\SGF^{w'}_s \big( \frac{y}{1-yz}, z \big) - \SGF^{w'}_s(y,z) & \text{if } w = bw', \\
\SGF^{w'}_s(y,z) & \text{if } w = tw' \text{ with } t \notin \{s,b\}, \\
\end{cases}
\]
where
\[
\SGF(y,z) = \frac{1+yz-\sqrt{1-4y-2yz+y^2z^2}}{2(z+1)}
\]
is the Schr\"oder generating function (see \cref{prop:CatalanSchroderGF}).
\end{itemize}
\end{proposition}

\begin{proof}
According to \cref{prop:cotreeDeletion,prop:faceLatticeConstrainahedron}, we need to count the $(m,n)$-cotrees of rank~$p$.
We group them according to their type, which can be any word~$w$ on the alphabet~$\{\ell,r,b\}$ with no two consecutive~$\ell$ nor two consecutive~$r$ and such that~$1 \le |w|_\ell + |w|_b \le m$ and~$1 \le |w|_r + |w|_b \le n$.
We then need to construct the two trees~$L$ and~$R$ with compatible partitions of their nodes.
If~$w = \varepsilon$ is the empty word, then both~$L$ and~$R$ are empty trees with a single leaf.
Otherwise, $w = tw'$ with~$t \in \{\ell,r,b\}$, and we construct~$L$ (resp.~$R$) by considering a tree~$L'$ (resp.~$R'$) for~$w'$~and
\begin{itemize}
\item if~$t = \ell$ (resp.~$t = r$), grafting at each leaf of~$L'$ (resp.~$R'$) a Schr\"oder tree (thus the substitution of the~$y$ variable in~$\SGF^{w'}_s$ by~$\SGF(y,z)$), such that not all leaves of~$L'$ (resp.~$R'$) are replaced by an empty trees (thus the subtraction of~$\SGF^{w'}_s$),
\item if~$t = b$, grafting at each leaf of~$L'$ (resp.~$R'$) either an empty tree or a tree with a single node (thus the substitution of the~$y$ variable in~$\SGF^{w'}_s$ by~$\smash{\frac{y}{1-yz}}$), such that not all leaves of~$L'$ (resp.~$R'$) are replaced by an empty tree (thus the subtraction of~$\SGF^{w'}_s$).
\qedhere
\end{itemize}
\end{proof}

For instance, the $f$-vectors of all constrainahedra~$\Constrainahedron$ with~$m + n \le 5$ are displayed in \cref{fig:multiplihedra2,fig:multiplihedra3,fig:multiplihedra4} (all these constrainahedra are multiplihedra since~${\Perm[1] = \Asso[1]}$ and ${\Perm[2] = \Asso[2]}$).
The $f$-vector of the $(3,3)$-constrainahedron~$\Constrainahedron[3][3]$ is 
\[
f(\Constrainahedron[3][3]) = (1, 606, 1550, 1384, 498, 60, 1).
\]


\section{Biassociahedra}
\label{sec:biassociahedra}

In this section, we study the family of $(m,n)$-biassociahedra, obtained as the shuffle of an $m$-anti-associahedron~$\Ossa[m]$ with an $n$-associahedron~$\Asso[n]$.
The combinatorics of the biassociahedron was already studied in~\cite{Markl, SaneblidzeUmble-matrads, MerkulovWillwacher}.
We recall the combinatorial model of bitrees (\cref{subsec:bitrees}), describe the face lattice, fan and oriented skeleton of the $(m,n)$-biassociahedron in terms of these bitrees (\cref{subsec:biassociahedra}), provide explicit vertex and facet descriptions of the $(m,n)$-biassociahedron (\cref{subsec:vertexFacetDescriptionsBiassociahedra}), and present enumerative results on the number of vertices, faces and facets of the $(m,n)$-biassociahedron (\cref{subsec:numerologyBiassociahedra}).


\subsection{Bitrees}
\label{subsec:bitrees}

We start by recalling the bitrees of~\cite{Markl}, illustrated in \cref{fig:bitrees}.
Intuitively, a bitree is a pair of Schr\"oder trees, the first growing up and the second growing down, drawn side to side, together with the information of the relative positions of their nodes.
Examples are illustrated in \cref{fig:bitrees}.

We say that a tree is (growing) \defn{up} (resp.~\defn{down}) when we see it as a poset oriented from its root to its leaves (resp.~from its leaves to its roots), and we draw it accordingly so that the orientation goes from bottom to top.

\begin{definition}[\cite{Markl}]
\label{def:bitree}
A \defn{$(m,n)$-bitree} is a triple~$\BT \eqdef (U, D, \mu)$, where $U$ is an up Schr\"oder $m$-tree, $D$ is a down Schr\"oder $n$-tree, and $\mu$ is an ordered partition of the nodes of~$U$ and~$D$~such~that
\begin{itemize}
\item the part of~$\mu$ containing a node~$\node{n}$ of~$U$ (resp.~$D$) distinct from the root is before (resp.~after) or equal to the part of~$\mu$ containing the parent of~$\node{n}$,
\item no two consecutive parts of~$\mu$ are both contained in~$U$ or both contained in~$D$,
\item there is no oriented path in~$U$ (resp.~in~$D$) joining two nodes in a part of~$\mu$ which meets both~$U$ and~$D$.
\end{itemize}
We say that a part of~$\mu$ is of type~$u$, $d$ or~$b$ when it contains nodes from~$U$, $D$ or both~$U$ and~$D$, and we call \defn{type} of the bitree the word given by the sequence of types of the parts of~$\mu$.
We denote by~$\f{BT}_{m,n}$ the set of $(m,n)$-bitrees.
\end{definition}

To represent a $(m,n)$-bitree~$\BT \eqdef (U, D, \mu)$, we draw the two trees~$U$ and~$D$ side by side in opposite directions ($U$ grows up while $D$ grows down), and we mark the separations between the parts of~$\mu$ by (red) horizontal lines.
Note that~$\mu$ is read from bottom to top.
Examples are illustrated in \cref{fig:bitrees}.

\begin{figure}
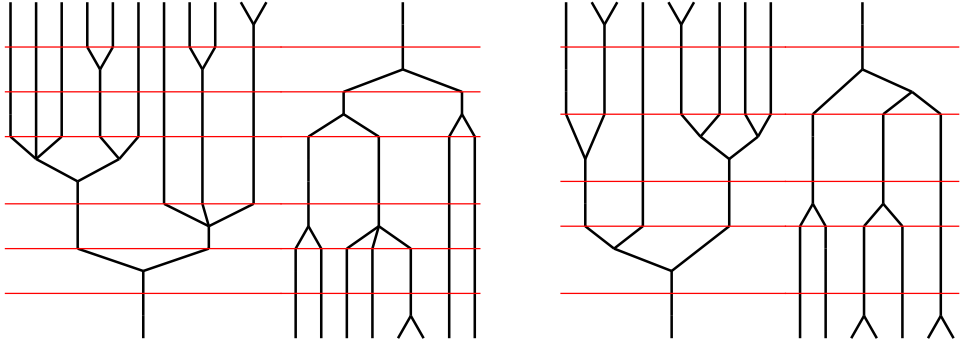

	\centerline{
		\bitree[1.2]{[[1[[2[3[[[4[[5[6[]]]]][4[5[6[]]]][4[5[6[]]]]][[4[5[[6[]][6[]]]]][4[5[6[]]]]]]]][2[[3[4[5[6[]]]]][3[4[5[[6[]][6[]]]]]][3[4[5[6[[][]]]]]]]]]]]}{[[6[[5[[4[[[3[[2[[1[]]]][2[1[]]]]]]]][4[3[[2[1[]]][2[1[]]][2[1[[][]]]]]]]]][5[[4[3[2[1[]]]]][4[3[2[1[]]]]]]]]]]}{1,2,3,4,5,6}
		\hspace*{-.5cm}
		\bitree[1.2]{[[1[[[2[[3[[4[[[5[]]]]][4[5[[][]]]]]]]][2[3[4[5[]]]]]][2[3[[[4[5[[][]]]][4[5[]]]][[4[5[]]][4[5[]]]]]]]]]]}{[[5[[4[[[3[[2[1[]]][2[[[1[]]]]]]]]]][[4[3[[2[1[[][]]]][2[1[]]]]]][4[3[2[1[[][]]]]]]]]]]}{1,2,3,4,5}
	}
	\caption{A $(10,7)$-bitree of type~$dubudbu$ (left), a binary $(8,6)$-bitree of type~$dududu$ (right).}
	\label{fig:bitrees}
\end{figure}

We now define the bitree deletion poset.
\cref{def:bitreeDeletion} provides a direct description in terms of bitrees, while \cref{def:bitreePreposet} provides an alternative simpler but indirect description in terms of preposets.
To illustrate the following definition, \cref{fig:bitreeDeletions} represents a sequence of deletions in $(6,5)$-bitrees.

\begin{figure}
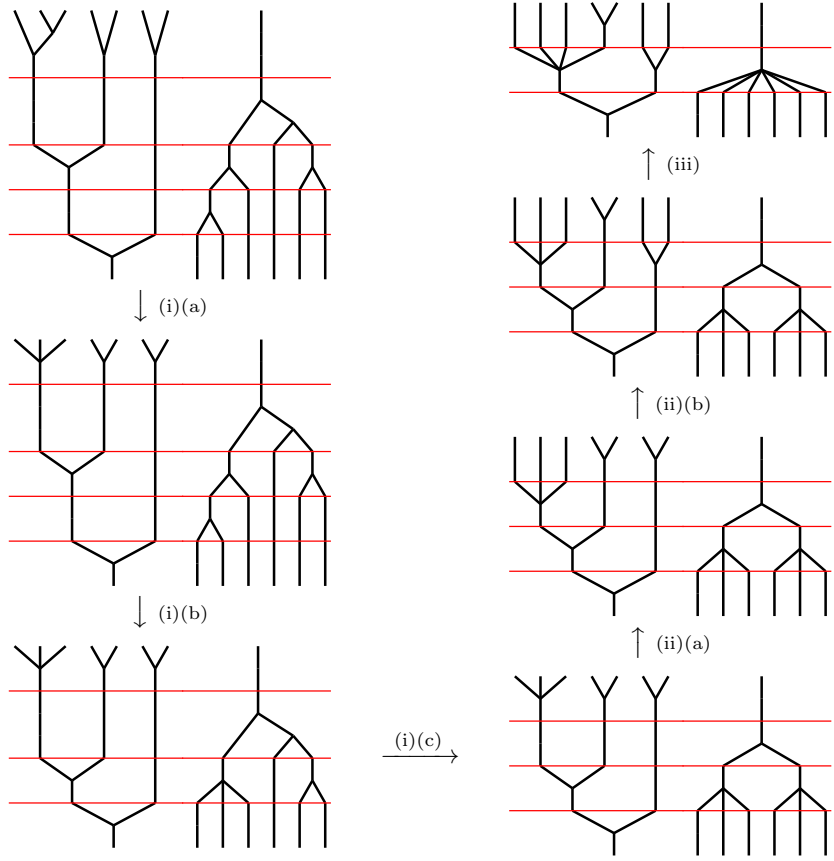

	\centerline{
		\begin{tabular}{c}
		\bitree[1.2]{[[1[[2[[3[[[4[[][[][]]]]]]][3[4[[][]]]]]]]][1[2[3[4[[][]]]]]]]}{[[[4[[3[[2[[1[[]]][1[]]]][2[1[]]]]][[3[2[1[]]]][3[[2[1[]]][2[1[]]]]]]]]]]}{1,2,3,4} \\
		$\big\downarrow$ {\scriptsize (i)(a)} \\[-.2cm]
		\bitree[1.2]{[[1[[2[[3[[[4[[][][]]]]]][3[4[[][]]]]]]]][1[2[3[4[[][]]]]]]]}{[[4[[3[[2[[1[[]]][1[]]]][2[1[]]]]][[3[2[1[]]]][3[[2[1[]]][2[1[]]]]]]]]]}{1,2,3,4} \\
		$\big\downarrow$ {\scriptsize (i)(b)} \\[-.2cm]
		\bitree[1.2]{[[1[[2[[[3[[][][]]]]]][2[3[[][]]]]]][1[2[3[[][]]]]]]}{[[3[[2[[1[[]]][1[]][1[]]]][[2[1[]]][2[[1[]][1[]]]]]]]]}{1,2,3}
		\end{tabular}
		\hspace{-.3cm}\raisebox{-4.5cm}{$\xrightarrow{\text{ (i)(c) }}$}\hspace{-.3cm}
		\begin{tabular}{c}
		\bitree[1.2]{[[1[[2[]][2[]][2[]][2[[][]]]]][1[[2[]][2[]]]]]}{[[2[[1[[]]][1[]][1[]][1[]][1[]][1[]]]]]}{1,2} \\
		$\big\uparrow$ {\scriptsize (iii)} \\[-.2cm]
		\bitree[1.2]{[[1[[2[[3[]][3[]][3[]]]][2[3[[][]]]]]][1[2[[3[]][3[]]]]]]}{[[3[[2[[1[[]]][1[]][1[]]]][2[[1[]][1[]][1[]]]]]]]}{1,2,3} \\		
		$\big\uparrow$ {\scriptsize (ii)(b)} \\[-.2cm]
		\bitree[1.2]{[[1[[2[[3[]][3[]][3[]]]][2[3[[][]]]]]][1[2[3[[][]]]]]]}{[[3[[2[[1[[]]][1[]][1[]]]][2[[1[]][1[]][1[]]]]]]]}{1,2,3} \\		
		$\big\uparrow$ {\scriptsize (ii)(a)} \\[-.2cm]
		\bitree[1.2]{[[1[[2[[3[[][][]]]]][2[3[[][]]]]]][1[2[3[[][]]]]]]}{[[3[[2[[1[[]]][1[]][1[]]]][2[[1[]][1[]][1[]]]]]]]}{1,2,3}
		\end{tabular}
	}
	\caption{Deletions in $(6,5)$-bitrees.}
	\label{fig:bitreeDeletions}
\end{figure}

\begin{definition}
\label{def:bitreeDeletion}
Let~$\BT \eqdef (U, D, \mu)$ and $\BT' \eqdef (U', D', \mu')$ be two $(m,n)$-bitrees.
We say that~$\BT'$ is obtained by a \defn{deletion} in~$\BT$ in either of the following three cases:
\begin{enumerate}[(i)]
\item \textbf{Node deletion:} $U'$ (resp.~$D'$) is obtained by deleting a node~$\node{n}$ with parent~$\node{p}$ in~$U$ (resp.~$D$) in the following situations:
	\begin{enumerate}[(a)]
	\item $\node{n}$ and $\node{p}$ belong to the same part of~$\mu$, then~$\mu'$ is obtained by deleting~$\node{n}$ from~$\mu$,
	\item  the part of~$\node{n}$ is of type~$u$ (resp.~$d$), the part of~$\node{p}$ is of type~$b$, and the parts of~$\node{n}$ and~$\node{p}$ are consecutive, then~$\mu'$ is obtained by deleting~$\node{n}$ from~$\mu$,
	\item the part of~$\node{n}$ is of type~$b$, the part of~$\node{p}$ is of type~$u$ (resp.~$d$), and the parts of~$\node{n}$ and~$\node{p}$ are consecutive, then~$\mu'$ is obtained from~$\mu$ by moving~$\node{p}$ to the part of~$\node{n}$ and deleting~$\node{n}$.
	\end{enumerate}
\item \textbf{Nodes move:} $U' = U$, $D' = D$, and~$\mu'$ is obtained from~$\mu$ by 
	\begin{enumerate}[(a)]
	\item either creating, in between two consecutive parts~$\mu_i$ of type~$u$ (resp.~$d$) and~$\mu_{i+1}$ of type~$d$ (resp.~$u$), a new part containing a node of~$\mu_i$ whose children (resp.~parent) are not in~$\mu_i$ and a node of~$\mu_{i+1}$ whose children (resp.~parent) are not in~$\mu_{i+1}$ (and removing these nodes from their original parts in~$\mu$),
	\item or moving a node~$\node{n}$ of~$U$ from its part~$\mu_i$ to the previous (or next) part~$\mu_{i \pm 1}$ in~$\mu$, provided that the part $\mu_i$ is of type~$u$, that the part~$\mu_{i \pm 1}$ is of type~$b$, and that the parent (or children) of~$\node{n}$ does not belong to~$\mu_i \cup \mu_{i \pm 1}$ (and same for~$D$ exchanging $u$/$d$, previous/next and parent/children),
	\end{enumerate}
\item \textbf{Twin parts merge:} $\mu'$ is obtained by merging two consecutive parts of~$\mu$ of type~$b$, and~$U'$ (resp.~$D'$) is obtained by deleting any node~$\node{n}$ in~$U$ (resp.~$D$) such that both~$\node{n}$ and its parent belong to these parts.
\end{enumerate}
\end{definition}

\begin{proposition}
\label{prop:bitreeDeletion}
For all integers~$m, n \ge 0$, the set~$\f{BT}_{m,n}$ is stable by deletion, and the deletion graph is the Hasse diagram of a poset ranked by~$\rank(U, D, \mu) = m + n - |U| - |D| + \beta(\mu)$, where~$\beta(\mu)$ is the sum of~$|\mu_i|-1$ over all parts~$\mu_i$ of~$\mu$ with~$\mu_i \cap U \ne \varnothing \ne \mu_i \cap D$.
In particular a $(m,n)$-bitree~$\BT \eqdef (U, D, \mu)$ has
\begin{itemize}
\item rank~$0$ if and only if both~$U$ and~$D$ are binary trees, and no part of~$\mu$ meets both~$U$ and~$D$,
\item rank~$m+n-2$ if and only if~$\mu$ has two parts, and each part of~$\mu$ either meets both~$U$ and~$D$ or contains a single node,
\item rank~$m+n-1$ if and only if~$\mu$ has a single part (hence, both~$U$ and~$D$ have a single node).
\end{itemize}
\end{proposition}

\begin{proof}
Consider a deletion transforming~$\BT \eqdef (U, D, \mu)$ to~$\BT' \eqdef (L', R', \mu')$.
Then~$\BT'$ is clearly a $(m,n)$-bitree since~$U'$ and~$D'$ are still Schr\"oder trees, and the partition~$\mu'$ fulfills the conditions of \cref{def:bitree}.
For the rank, we distinguish three cases corresponding to that of \cref{def:bitreeDeletion}:
\begin{enumerate}[(i)]
\item \textbf{Node deletion:} $|U'| + |D'| = |U| + |D| - 1$ while~$\beta(\mu') = \beta(\mu')$.
\item \textbf{Nodes move:} $|U'| = |U|$, $|D'| = |D|$, while~$\beta(\mu') = \beta(\mu) + 1$.
\item \textbf{Twin parts merge:} if~$\delta$ denotes the number of nodes~$\node{n}$ of~$U$ and $D$ such that both~$\node{n}$ and its parent belong to the merged parts of~$\mu$, then $|U'| + |D'| = |U| + |D| - \delta$ and~$\beta(\mu') = \beta(\mu) - \delta + 1$.
\end{enumerate}
In all three situations, we get~$\rank(\BT') = \rank(\BT)+1$.
The end of the statement immediately follows.
\end{proof}

\begin{definition}
\label{def:bitreeDeletionPoset}
The \defn{$(m,n)$-bitree deletion poset} is the poset on $\f{BT}_{m,n}$ where a $(m,n)$-bitree is covered by all $(m,n)$-bitrees that can be obtained by a deletion.
\end{definition}

The $(m,n)$-bitree deletion poset can alternatively be defined using preposets.

\begin{definition}
\label{def:bitreePreposet}
A $(m,n)$-bitree~$\BT \eqdef (U, D, \mu)$ defines a preposet~$\preccurlyeq_{\BT}$ on~$[m+n]$ that can be read as follows.
Label~$U$ by~$[m]$ in inorder and $D$ by~$[n]^{+m}$ in inorder (shifted by~$m$).
Then, for any~$i,j \in [m+n]$, we have~$i \preccurlyeq_{\BT} j$ if the part of~$\mu$ containing~$i$ is before the part of~$\mu$ containing~$j$, or if there is a (possibly empty) path from the node containing~$i$ to the node containing~$j$ in the tree~$U$ oriented towards its leaves or in the tree~$D$ oriented towards its root.
\end{definition}

\begin{proposition}
\label{prop:characterizationBitreePreposets}
The preposets~$\preccurlyeq_{\BT}$ for~$\BT \in \f{BT}_{m,n}$ are precisely the preposets~$\preccurlyeq$ on~$[m+n]$ in which any~$1 \le i < k \le m+n$ are comparable (\ie~$i \preccurlyeq k$ or $i \succcurlyeq k$ or both) if and only if
\begin{itemize}
\item either~$i \le m < k$,
\item or~$m < i$ (resp.~$k \le m$) and at least one of the following holds:
	\begin{itemize}
	\item there exists no~$i < j < k$ such that $i \prec j \succ k$ (resp.~$i \succ j \prec k$),
	\item there exists~$j \in [m]$ (resp.~$j \in [n]^{+m}$) such that $i \preccurlyeq j \preccurlyeq k$ or $i \succcurlyeq j \succcurlyeq k$.
	\end{itemize}
\end{itemize}
\end{proposition}

\begin{proof}
Any preposet~$\preccurlyeq_{\BT}$ clearly satisfies these conditions.
Conversely, given a preposet~$\preccurlyeq$ on~${[m+n]}$ satisfying these conditions, consider 
\begin{itemize}
\item the preposet~$\preccurlyeq_u$ on~$[m]$ defined by~$i \preccurlyeq_u k$ if and only if $i \preccurlyeq k$ and there is no~$i < j < k$ such that~$i \succ j \prec k$,
\item the preposet~$\preccurlyeq_d$ on~$[n]$ defined by~$i \preccurlyeq_d k$ if and only if ${i + m \preccurlyeq k+m}$ and there is no~${i < j < k}$ such that~$i + m \prec j + m \succ k + m$.
\end{itemize}
The preposet~$\preccurlyeq_u$ (resp.~$\preccurlyeq_d$) is clearly the preposet~$\preccurlyeq_U$ (resp.~$\preccurlyeq_D$) of an up Schr\"oder $m$-tree~$U$ (resp.~a down Schr\"oder $n$-tree~$D$).
We then obtain the partition~$\mu$ by considering the relations~$i \preccurlyeq k$ with~$i \le m < k$.
Details are left to the reader.
\end{proof}

\begin{proposition}
\label{prop:bitreeDeletionPosetOnPreposets}
In the bitree deletion poset, $\BT$ is smaller than~$\BT'$ if and only if~$\preccurlyeq_{\BT}$ refines~$\preccurlyeq_{\BT'}$.
\end{proposition}

\begin{proof}
An immediate case analysis shows that deletions in a bitree~$\BT$ defined in \cref{def:bitreeDeletion} precisely translate all possible refinements in the corresponding preposet~$\preccurlyeq_{\BT}$.
\end{proof}

Finally, we define the rotations in bitrees, which correspond to rank~$1$ bitrees.
To illustrate the following definition, \cref{fig:bitreeRotations} represents a sequence of right rotations in binary $(3,2)$-bitrees.

\begin{figure}
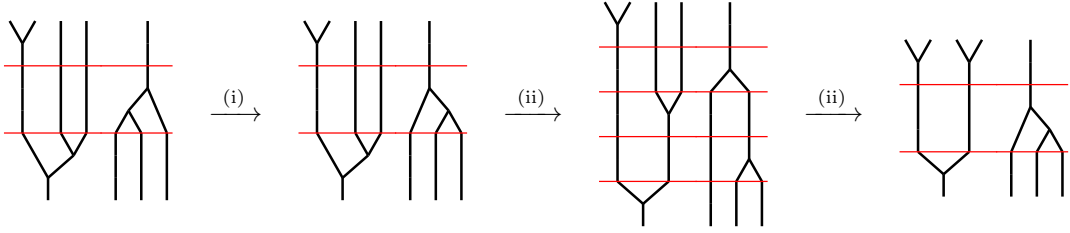

	\centerline{
		\bitree[1.2]{[[1[[[2[[][]]]]]][[1[2[]]][1[2[]]]]]}{[[2[[[1[[[]]]][1[]]][1[]]]]]}{1,2}
		\hspace{-.3cm}\raisebox{-1.5cm}{$\xrightarrow{\text{ (i) }}$}\hspace{-.3cm}
		\bitree[1.2]{[[1[[[2[[][]]]]]][[1[2[]]][1[2[]]]]]}{[[2[[1[[[]]]][[1[]][1[]]]]]]}{1,2}
		\hspace{-.3cm}\raisebox{-1.5cm}{$\xrightarrow{\text{ (ii) }}$}\hspace{-.3cm}
		\raisebox{.25cm}{\bitree[1.2]{[[1[[2[[3[[4[[][]]]]]]]]][1[2[[3[4[]]][3[4[]]]]]]]}{[[4[[3[[2[1[[]]]]]][3[2[[1[]][1[]]]]]]]]}{1,2,3,4}}
		\hspace{-.3cm}\raisebox{-1.5cm}{$\xrightarrow{\text{ (ii) }}$}\hspace{-.3cm}
		\raisebox{-.25cm}{\bitree[1.2]{[[1[[[2[[][]]]]]][1[2[[][]]]]]}{[[2[[1[[]]][[1[]][1[]]]]]]}{1,2}}
	}
	\caption{Right rotations in binary $(3,2)$-bitrees.}
	\label{fig:bitreeRotations}
\end{figure}

\begin{definition}
\label{def:rotationsBitrees}
We call \defn{binary $(m,n)$-bitrees} the rank~$0$ $(m,n)$-bitrees, \ie where both~$U$ and~$D$ are binary trees, and no part of~$\mu$ meets both~$U$ and~$D$.
We say that two binary $(m,n)$-bitrees~$\BT \eqdef (U, D, \mu)$ and $\BT' \eqdef (U', D', \mu')$ are connected by a \defn{right rotation} if:
\begin{enumerate}[(i)]
\item \textbf{Edge rotation:} $U'$ (resp.~$D'$) is obtained from~$U$ (resp.~$D$) by the right rotation of an edge whose endpoints belong to the same part of~$\mu$,
\item \textbf{Twin parts:} $U' = U$, $D' = D$, and $\mu'$ is obtained from~$\mu$ by creating, in between two consecutive parts~$\mu_i$ of type~$u$ and~$\mu_{i+1}$ of type~$d$, first a new part containing a node of~$\mu_{i+1}$ whose children are not in~$\mu_{i+1}$, and second a new part containing a node of~$\mu_i$ whose children are not in~$\mu_i$ (and removing these nodes from their original parts in~$\mu$, and merging consecutive parts of the same type~$u$ or~$d$ if any).
\end{enumerate}
\end{definition}

\begin{remark}
\label{rem:algebraicInterpretationBiassociahedra}
The algebraic interpretation of the binary $(m,n)$-bitrees involves both a magmatic product $\ast$ and a magmatic coproduct $\Delta$ on a set $X$.
The nodes in the left part of a bitree are associated with the coproduct $\Delta$, while the nodes in the right part are associated with the product~$\ast$.
One starts at the bottom with a $1 \times (n+1)$-matrix of elements of $X$ (with $1$ column and $n+1$ rows).
Intermediate steps will go through rectangular $p \times q$-matrices of elements of $X$ with increasing $1 \leq p \leq m+1$ and decreasing $1 \leq q \leq n+1$, until one reaches a $(m+1) \times 1$-matrix of elements of $X$ at the top.
Going up through a node in the left part of the bitree means applying $\Delta$ to each element in a column of the matrix, replacing this column by two columns and increasing $p$ by $1$.
Similarly, going up through a node in the right part of the bitree means applying $\ast$ to corresponding elements in two consecutive rows of the matrix, replacing these two rows by a single row and decreasing $q$ by $1$.
In short, a left node stands for $\Delta$ duplicating a column, and a right node for $\ast$ merging two consecutive rows.
\end{remark}


\subsection{Anti-associahedra $\shuffleDP$ Associahedra}
\label{subsec:biassociahedra}

We now consider shuffles of anti-associahedra with asso\-ciahedra.
We call \defn{anti-associahedron} the polytope~$\Ossa \eqdef (n+1) \, \one - \Asso$.
It has the same combinatorics (but a different embedding) as the associahedron~$\Asso$.

\begin{definition}
\label{def:biassociahedron}
The \defn{$(m,n)$-biassociahedron} is the polytope~$\Biassociahedron = \Ossa[m] \shuffleDP \Asso[n]$.
\end{definition}

Note that since~$\Perm[1] = \Asso[1]$ and $\Perm[2] = \Asso[2]$, the first $(m,n)$-biassociahedron which is neither an associahedron, nor a $(m,n)$-multiplihedron, is the $(3,3)$-biassociahedron \linebreak $\Biassociahedron[3][3]$, which is a $5$-dimensional polytope.
There is thus no reasonable example to be drawn in this section.

\begin{proposition}
\label{prop:faceLatticeBiassociahedron}
The face lattice of the $(m,n)$-biassociahedron~$\Biassociahedron$ is isomorphic to the $(m,n)$-bitree deletion poset (augmented with a minimal element).
\end{proposition}

\begin{proof}
This follows from \cref{prop:shuffleFacePreposets} (see also \cref{rem:biPreposets}), since~$(m,n)$-bitrees are just a specialization of bipreposets.
\end{proof}

\begin{remark}
\label{rem:simpleBiassociahedron}
In contrast to the associahedron~$\Asso$, the biassociahedron~$\Biassociahedron$ is simple if and only if~$m = 0$, or~$n = 0$, or $\max(m,n) \le 2$.
\end{remark}

\begin{proposition}
\label{prop:fanBiassociahedron}
The normal fan of the $(m,n)$-biassociahedron~$\Biassociahedron$ is the fan containing one cone $\polytope{C}(\BT) \eqdef \set{\b{x} \in \R^{m+n}}{x_i \le x_j \text{ if } i \preccurlyeq_{\BT} j}$ for each~${\BT \in \f{BT}_{m,n}}$.
\end{proposition}

\begin{proof}
Immediate from \cref{prop:faceLatticeBiassociahedron,def:bitreePreposet}.
\end{proof}

\begin{proposition}
\label{prop:graphBiassociahedron}
When oriented in the direction~${\b{\omega} \eqdef (n,\dots,1) - (1,\dots,n) = \sum_{i \in [n]} (n+1-2i) \, \b{e}_i}$, the graph of the $(m,n)$-biassociahedron~$\Biassociahedron$ is isomorphic to the right rotation graph on binary $(m,n)$-bitrees.
\end{proposition}

\begin{proof}
It follows from \cref{prop:faceLatticeBiassociahedron} that the vertices of~$\Biassociahedron$ correspond to the binary $(m,n)$-bitrees.
It is easy to check that the edges of~$\Biassociahedron$ oriented by~$\b{\omega}$ correspond to right rotations on binary $(m,n)$-bitrees.
\end{proof}

\begin{remark}
\label{rem:noBiTamari}
In contrast to \cref{prop:graphMultiplihedron}, note that the right rotation graph on binary $(m,n)$-bitrees is not the Hasse diagram of a lattice when~$m \ge 3$ and~$n \ge 3$.
See \cref{fig:biassociahedronNotLattice} for examples of a pair of binary $(3,3)$-bitrees with no join and a pair of binary $(3,3)$-bitrees with no meet.
\begin{figure}
	\centerline{
	\begin{tikzpicture}
		\node[label=west:{$\BT_1 =$}] (a) at (-4,4.3) {\bitree[1.2]{[[1[[2[[3[[[]]]][3[]]]]]][[1[2[3[]]]][1[2[3[]]]]]]}{[[[3[[2[1[[[]]]]]]][3[2[1[]]]]][3[2[[1[]][1[]]]]]]}{1,2,3}};
		\node[label=east:{$= \BT_2$}] (b) at (4,4.3) {\bitree[1.2]{[[[1[2[3[]]]][1[2[3[]]]]][1[[2[[3[[[]]]][3[]]]]]]]}{[[3[2[[1[]][1[]]]]][[3[[2[1[[[]]]]]]][3[2[1[]]]]]]}{1,2,3}};
		\node (c) at (-6,0) {\bitree[1.2]{[[1[[2[[3[[]]][3[]]]]]][[1[2[3[]]]][1[2[3[]]]]]]}{[[3[[2[[1[[[]]]][1[]]]]]][3[2[[1[]][1[]]]]]]}{1,2,3}};
		\node (d) at (2,0) {\bitree[1.2]{[[1[[2[[3[[[]]]][3[]]]]]][1[2[[3[]][3[]]]]]]}{[[3[[2[[1[[]]][1[]]]]]][[3[2[1[]]]][3[2[1[]]]]]]}{1,2,3}};
		\node (e) at (-2,0) {\bitree[1.2]{[[1[[2[[3[[[]]]][3[]]]]]][1[2[[3[]][3[]]]]]]}{[[[3[2[1[]]]][3[2[1[]]]]][3[[2[[1[[]]][1[]]]]]]]}{1,2,3}};
		\node (f) at (6,0) {\bitree[1.2]{[[[1[2[3[]]]][1[2[3[]]]]][1[[2[[3[[]]][3[]]]]]]]}{[[3[[2[[1[[[]]]][1[]]]]]][3[2[[1[]][1[]]]]]]}{1,2,3}};
		\node[label=west:{$\BT[S]_1 =$}] (g) at (-4,-4.6) {\bitree[1.2]{[[1[[2[3[[4[[5[[]]][5[]]]]]]]]][1[2[[3[4[5[]]]][3[4[5[]]]]]]]]}{[[5[[4[3[[2[[1[]][1[]]]]]]]]][5[4[[3[2[1[[]]]]][3[2[1[]]]]]]]]}{1,2,3,4,5}};
		\node[label=east:{$= \BT[S]_2$}] (h) at (4,-4.6) {\bitree[1.2]{[[1[2[[3[4[5[]]]][3[4[5[]]]]]]][1[[2[3[[4[[5[[]]][5[]]]]]]]]]]}{[[5[4[[3[2[1[[]]]]][3[2[1[]]]]]]][5[[4[3[[2[[1[]][1[]]]]]]]]]]}{1,2,3,4,5}};
		\draw[thick,gray] (a.south) -- (c.north);
		\draw[thick,gray] (a.south) -- (e.north);
		\draw[thick,gray] (b.south) -- (d.north);
		\draw[thick,gray] (b.south) -- (f.north);
		\draw[thick,gray] (c.south) -- (g.north);
		\draw[thick,gray] (d.south) -- (g.north);
		\draw[thick,gray] (e.south) -- (h.north);
		\draw[thick,gray] (f.south) -- (h.north);
	\end{tikzpicture}
	}
	\caption{Rotations on all $(3,3)$-bitrees larger than~$\BT[S]_1$ or~$\BT[S]_2$ and smaller than~$\BT_1$ or~$\BT_2$. This shows that~$\BT[S]_1$ and~$\BT[S]_2$ have no join, and~$\BT_1$ and~$\BT_2$ have no meet, so that the rotation graph on binary $(3,3)$-bitrees does not define a lattice.}
	\label{fig:biassociahedronNotLattice}
\end{figure}
\end{remark}


\subsection{Vertex and facet descriptions}
\label{subsec:vertexFacetDescriptionsBiassociahedra}

Our next two statements, illustrated in \cref{fig:coordinatesBitrees,fig:inequalitiesBitrees}, provide the vertex and facet descriptions of the $(m,n)$-biasso\-ciahedron~$\Biassociahedron$.
The proofs are elementary computations from \mbox{\cref{def:associahedron,def:graphicalZonotope,def:biassociahedron}}.

\begin{proposition}
\label{prop:verticesBiassociahedron}
For any~$i \in [m+n]$, the $i$-th coordinate of the vertex of the $(m,n)$-biassociahedron $\Biassociahedron$ corresponding to a binary $(m,n)$-bitree~$(U, D, \mu)$ is given by
\begin{itemize}
\item if~$i \le m$, then $m+1$ minus the product of the numbers of leaves in the left and right subtrees of~$\node{n}$, plus the number of nodes of~$D$ below~$\node{n}$, where~$\node{n}$ is the node of~$U$ labeled by~$i$ in inorder.
\item if~$i \ge m+1$, the product of the numbers of leaves in the left and right subtrees of~$\node{n}$, plus the number of nodes of~$U$ below~$\node{n}$, where~$\node{n}$ is the node of~$D$ labeled by~$i-m$ in inorder.
\end{itemize}
In particular, the sum of the coordinates is always~$\binom{m+1}{2} + \binom{n+1}{2} + mn = \binom{m+n+1}{2}$.
\end{proposition}

\begin{proposition}
\label{prop:facetsBiassociahedron}
Let~$\BT \eqdef (U, D, \mu)$ be a $(m,n)$-bitree of rank~$m+n-2$.
Let~${A \eqdef A_1 \cup \dots \cup A_k}$ where $A_1, \dots A_k$ are the inorder labels of the nodes of~$U$ located in the top part~$\mu_2$, and let ${B \eqdef B_1 \cup \dots \cup B_\ell}$ where~$B_1, \dots, B_\ell$ are the inorder labels shifted by~$m$ of the nodes of~$D$ located in the bottom part~$\mu_1$.
Then the facet of the $(m,n)$-biassociahedron~$\Biassociahedron$ corresponding to~$\BT$ is defined by the inequality
\[
\dotprod{\b{x}}{\one_{([m] \ssm A) \cup B}} \ge \binom{m+1}{2} - |A| \cdot (m+1) + \sum_{i \in [k]} \binom{|A_i|+1}{2} + (m-|A|) \cdot |B| + \sum_{j \in [\ell]} \binom{|B_j|+1}{2}.
\]
Moreover, this inequality is a facet defining inequality of the permutahedron~$\Perm[m+n]$ if and only if~$k \le 1$ and~$\ell \le 1$, \ie if both $U$ and $D$ have at most two nodes.
\end{proposition}

Note that, in contrast to \cref{prop:MinkowskiMultiplihedron,prop:MinkowskiConstrainahedron}, we do not provide an expression of the $(m,n)$-biassociahedron~$\Biassociahedron$ as a signed Minkowski sum of faces of the standard simplex~$\simplex_{[m+n]}$. Such an expression is possible (since~$\Biassociahedron$ is a deformed permutahedron by \cref{prop:shuffleDeformedPermutahedra}), but combinatorially less attractive than that of~$\Multiplihedron$ or~$\Constrainahedron$ (as it requires to express the faces of the opposite standard simplex as signed Minkowski sums of faces of the standard simplex).
See \cite{Lange} for further discussion.

\begin{example}
\cref{fig:coordinatesBitrees} illustrates some vertex coordinates of $\Biassociahedron[3][3]$ computed by \cref{prop:verticesBiassociahedron} and \cref{fig:inequalitiesBitrees} illustrates some facet inequalities of $\Biassociahedron[3][3]$ computed by \cref{prop:facetsBiassociahedron}. Note that all vertices of $\Biassociahedron[3][3]$ have coordinate sum~$21$. Note that for any pair $(i, j) \in \{(1, 2), (1, 3), (2, 2), (2, 3), (2, 4), (3, 2), (3, 3), (3, 4), (4, 3), (4, 4)\}$, we have~$\BT_i$ smaller than~$\BT[S]_j$ in deletion order, so that the vertex corresponding to~$\BT_i$ is contained in the facet corresponding to~$\BT[S]_j$.

\begin{figure}
	\centerline{
	\begin{tabular}{c@{}c@{}c@{}c}
		$\BT_1$
		&
		$\BT_2$
		&
		$\BT_3$
		&
		$\BT_4$
		\\[-.3cm]
		\bitree[1.2]{[[[1[[[2[[][]]]]]][1[2[]]]][1[2[]]]]}{[[2[[[1[[[]]]][1[]]][[1[]][1[]]]]]]}{1,2}
		&
		\bitree[1.2]{[[1[[[2[[][]]]]]][[1[2[]]][1[2[]]]]]}{[[2[[[1[[[]]]][1[]]][[1[]][1[]]]]]]}{1,2}
		&
		\bitree[1.2]{[[1[[2[[3[[4[[][]]]]]]]]][1[2[[3[4[]]][3[4[]]]]]]]}{[[4[[3[[2[[1[[]]][1[]]]]]][3[2[[1[]][1[]]]]]]]]}{1,2,3,4}
		&
		\bitree[1.2]{[[1[[[2[[][]]]]]][1[2[[][]]]]]}{[[2[[[1[[]]][1[]]][[1[]][1[]]]]]]}{1,2}
		\\[.1cm]
		$(6, 2, 1, 3, 6, 3)$
		&
		$(6, 0, 3, 3, 6, 3)$
		&
		$(6, 0, 5, 2, 6, 2)$
		&
		$(6, 0, 6, 2, 5, 2)$
	\end{tabular}
	}
	\caption{Vertices of $\Biassociahedron[3][3]$ corresponding to four binary $(3,3)$-bitrees.}
	\label{fig:coordinatesBitrees}
\end{figure}

\begin{figure}
	\centerline{
	\begin{tabular}{c@{}c@{}c@{}c}
		$\BT[S]_1$
		&
		$\BT[S]_2$
		&
		$\BT[S]_3$
		&
		$\BT[S]_4$
		\\[-.3cm]
		\bitree[1.2]{[[1[[]]][1[]][1[]][1[]]]}{[[1[[]]][1[]][1[]][1[]]]}{1}
		&
		\bitree[1.2]{[[1[[][]]][1[]][1[]]]}{[[1[[][]]][1[[][]]]]}{1}
		&
		\bitree[1.2]{[[1[[][]]][1[]][1[]]]}{[[1[[][][][]]]]}{1}
		&
		\bitree[1.2]{[[1[[][]]][1[[][]]]]}{[[1[[]]][1[]][1[]][1[]]]}{1}
		\\[.1cm]
		$x_1 + x_2 + x_3 \ge 6$
		&
		$x_2 + x_3 + x_4$
		&
		$x_2 + x_3 + x_4$
		&
		$x_2 \ge 0$
		\\
		&
		$+ x_6 \ge 9$
		&
		$+ x_5 + x_6 \ge 15$
		&
	\end{tabular}
	}
	\caption{Facet defining inequalities of $\Biassociahedron[3][3]$ corresponding to four rank $4$ $(3,3)$-bitrees.}
	\label{fig:inequalitiesBitrees}
\end{figure}
\end{example}


\subsection{Numerology}
\label{subsec:numerologyBiassociahedra}

We now present enumerative results on the number of vertices, faces and facets of the $(m,n)$-biassociahedron~$\Biassociahedron$.
The first few values of these numbers are collected in \cref{table:verticesConstrainahedraBiassociahedra,table:facetsConstrainahedraBiassociahedra,table:facesBiassociahedra} in \cref{subsec:tablesConstrainahedraBiassociahedra}.
We start with vertices.
See~\cref{table:verticesConstrainahedraBiassociahedra}.

\begin{proposition}
\label{prop:numberVerticesBiassociahedra}
The number of vertices of the $(m,n)$-biassociahedron~$\Biassociahedron$ (equivalently of binary $(m,n)$-bitrees) is given by
\[
[x^{m+1} \, y^{n+1}] \, \sum_{i = 0}^{\min(m,n)} 2 \, \tCGF^{(i)}(x) \, \tCGF^{(i)}(y) + \tCGF^{(i)}(x) \, \tCGF^{(i+1)}(y) + \tCGF^{(i+1)}(x) \, \tCGF^{(i)}(y),
\]
where $\tCGF^{(i)}(x)$ is defined for~$i \ge 0$ by
\[
\tCGF^{(0)}(x) = x
\qquad\text{and}\qquad
\tCGF^{(i)}(x) = \tCGF^{(i-1)}(\CGF(x)) - \tCGF^{(i-1)}(x),
\]
where
\[
\CGF(x) = \frac{1-\sqrt{1-4x}}{2}
\]
is the Catalan generating function (see \cref{prop:CatalanSchroderGF}).
\end{proposition}

\begin{proof}
According to \cref{prop:bitreeDeletion,prop:faceLatticeBiassociahedron}, we need to count the binary $(m,n)$-bitrees.
We group them according to their type, which can be of the form~$(ud)^i$, $(du)^i$, $(ud)^iu$ or~$(du)^id$.
We then need to construct the two binary trees~$U$ and~$D$ with compatible partitions of their nodes into~$i$ (or~$i+1$) parts.
We construct a partitioned binary tree with~$i+1$ parts by
\begin{itemize}
\item choosing a binary tree~$T$ for the first part (thus the apparition of~$\CGF$),
\item grafting at each leaf of~$T$ a partitioned binary tree with $i-1$ parts (thus the substitution of the~$y$ variable in~$\tCGF^{(i)}$), such that not all leaves of~$T$ are replaced by an empty binary tree (thus the subtraction of~$\tCGF^{(i-1)}$ in the definition of~$\tCGF^{(i)}$).
\qedhere
\end{itemize}
\end{proof}

We now consider the number of facets of the $(m,n)$-biassociahedron~$\Biassociahedron$.
See~\cref{table:facetsConstrainahedraBiassociahedra}.

\begin{proposition}
\label{prop:numberFacetsBiassociahedra}
The number of facets of the $(m,n)$-biassociahedron~$\Biassociahedron$ is
\[
 (2^m-1)(2^n-1) + \binom{m+1}{2} + \binom{n+1}{2} - 1.
\]
\end{proposition}

\begin{proof}
According to \cref{prop:bitreeDeletion,prop:faceLatticeBiassociahedron}, the possible types for the $(m,n)$-bitrees corresponding to facets of the $(m,n)$-biassociahedron~$\Biassociahedron$ are:
\begin{itemize}
\item type~$ud$ (resp.~type~$du$): then both~$U$ and~$D$ have a single node, thus a single choice,
\item type~$bu$ (resp.~type~$db$): then $U$ (resp.~$D$) is a non-trivial corolla while~$D$ (resp.~$U$) has a single node, thus~$\binom{m+1}{2}-1$ choices (resp.~$\binom{n+1}{2}-1$ choices),
\item type~$ub$ (resp.~type~$bd$): then~$U$ (resp.~$D$) is any Schr\"oder tree of height~$2$ while~$D$ (resp.~$U$) has a single node, thus~$2^m-2$ choices (resp.~$2^n-2$ choices),
\item type~$bb$: then both~$U$ and~$D$ are Schr\"oder trees of height~$2$, thus~$(2^m-2)(2^n-2)$ choices.
\qedhere
\end{itemize}
\end{proof}

Finally, adapting the approach of \cref{prop:numberVerticesBiassociahedra}, we can count all faces of the $(m,n)$-biassocia\-hedron~$\Biassociahedron$ according to their dimension.

\begin{proposition}
\label{prop:numberFacesBiassociahedra}
Let~$BT(m,n,p)$ denote the number of $p$-dimensional faces of the $(m,n)$-biassocia\-hedron~$\Biassociahedron$, or equivalently the number of $(m,n)$-bitrees of rank~$p$.
Then the generating function~$\BTGF(x,y,z) \eqdef \sum_{m,n,p} BT(m,n,p) \, x^m \, y^n \, z^p$ is given by
\[
\BTGF(x,y,z) = \sum_w \SGF^{\rev(w)}_u(x,z) \, \SGF^w_d(y,z)
\]
where
\begin{itemize}
\item $w$ runs over all words on the alphabet~$\{u,d,b\}$ with no two consecutive~$u$ nor two consecutive~$d$ and such that~$1 \le |w|_u + |w|_b \le m$ and~$1 \le |w|_d + |w|_b \le n$,
\item $\rev(w) \eqdef w_k \dots w_1$ denotes the reverse of the word~$w = w_1 \dots w_k$, 
\item for a letter~$s \in \{u,d\}$, the generating function~$\SGF^w_s(y,z)$ is defined by~$\SGF^\varepsilon_s(y,z) \eqdef y$ and
\[
\SGF^{w}_s(y,z) \eqdef
\begin{cases}
\SGF^{w'}_s(\SGF(y,z), z) - \SGF^{w'}_s(y,z) & \text{if } w = sw', \\
\SGF^{w'}_s \big( \frac{y}{1-yz}, z \big) - \SGF^{w'}_s(y,z) & \text{if } w = bw', \\
\SGF^{w'}_s(y,z) & \text{if } w = tw' \text{ with } t \notin \{s,b\}, \\
\end{cases}
\]
where
\[
\SGF(y,z) = \frac{1+yz-\sqrt{1-4y-2yz+y^2z^2}}{2(z+1)}
\]
is the Schr\"oder generating function (see \cref{prop:CatalanSchroderGF}).
\end{itemize}
\end{proposition}

\begin{proof}
According to \cref{prop:bitreeDeletion,prop:faceLatticeBiassociahedron}, we need to count the $(m,n)$-bitrees of rank~$p$.
We group them according to their type, which can be any word~$w$ on the alphabet~$\{u,d,b\}$ with no two consecutive~$u$ nor two consecutive~$d$ and such that~$1 \le |w|_u + |w|_b \le m$ and~$1 \le |w|_d + |w|_b \le n$.
We then need to construct the two trees~$U$ and~$D$ with compatible partitions of their nodes.
If~$w = \varepsilon$ is the empty word, then both~$U$ and~$D$ are empty trees with a single leaf.
Otherwise, $w = tw'$ with~$t \in \{u,d,b\}$, and we construct~$U$ (resp.~$D$) by considering a tree~$U'$ (resp.~$D'$) for~$w'$~and
\begin{itemize}
\item if~$t = u$ (resp.~$t = d$), grafting at each leaf of~$U'$ (resp.~$D'$) a Schr\"oder tree (thus the substitution of the~$y$ variable in~$\SGF^{w'}_s$ by~$\SGF(y,z)$), such that not all leaves of~$U'$ (resp.~$D'$) are replaced by an empty trees (thus the subtraction of~$\SGF^{w'}_s$),
\item if~$t = b$, grafting at each leaf of~$U'$ (resp.~$D'$) either an empty tree or a tree with a single node (thus the substitution of the~$y$ variable in~$\SGF^{w'}_s$ by~$\smash{\frac{y}{1-yz}}$), such that not all leaves of~$U'$ (resp.~$D'$) are replaced by an empty tree (thus the subtraction of~$\SGF^{w'}_s$).
\qedhere
\end{itemize}
\end{proof}

For instance, the $f$-vectors of all biassociahedra~$\Biassociahedron$ with~$m + n \le 5$ are displayed in \cref{fig:multiplihedra2,fig:multiplihedra3,fig:multiplihedra4} (all these biassociahedra are multiplihedra since~${\Perm[1] = \Asso[1]}$ and ${\Perm[2] = \Asso[2]}$).
The $f$-vector of the $(3,3)$-biassociahedron~$\Biassociahedron[3][3]$ is 
\[
f(\Biassociahedron[3][3]) = (1, 606, 1549, 1382, 497, 60, 1).
\]
Note that it slightly differs from the $f$-vector of the $(3,3)$-constrainahedron~$\Constrainahedron[3][3]$ which is 
\[
f(\Constrainahedron[3][3]) = (1, 606, 1550, 1384, 498, 60, 1),
\]
given in \cref{subsec:numerologyConstrainahedra}.


\addtocontents{toc}{ \vspace{.2cm} }
\bibliographystyle{alpha}
\bibliography{shuffleDeformedPermutahedra}
\label{sec:biblio}


\clearpage
\appendix

\section{Tables}
\label{sec:tables}

\enlargethispage{.5cm}
All references like~\OEIS{A000142} are entries of the Online Encyclopedia of Integer Sequences~\cite{OEIS}.


\subsection{Zonotopes} \;~
\label{subsec:tablesZonotopes}

\begin{table}[h]
	\centerline{\begin{tabular}{r|rrrrrrrrrr|l}
		$m \backslash n$ & 0 & 1 & 2 & 3 & 4 & 5 & 6 & 7 & 8 & 9 & \\
		\hline
		0 & . & 1 & 2 & 4 & 8 & 16 & 32 & 64 & 128 & 256 & \OEIS{A000079} \\
		1 & 1 & 2 & 6 & 18 & 54 & 162 & 486 & 1458 & 4374 & & \OEIS{A025192} \\
		2 & 2 & 6 & 24 & 96 & 384 & 1536 & 6144 & 24576 & & & \OEIS{A002023} \\
		3 & 6 & 24 & 120 & 600 & 3000 & 15000 & 75000 & & & & \OEIS{A235702} \\
		4 & 24 & 120 & 720 & 4320 & 25920 & 155520 & & & & & ? \\
		5 & 120 & 720 & 5040 & 35280 & 246960 & & & & & & \\
		6 & 720 & 5040 & 40320 & 322560 & & & & & & & \\
		7 & 5040 & 40320 & 362880 & & & & & & & & \\
		8 & 40320 & 362880 & & & & & & & & & \\
		9 & 362880 & & & & & & & & & & \\
		\hline
		& \OEIS{A000142} & \OEIS{A000142} & \OEIS{A000142} & \OEIS{A001563} & \OEIS{A002775} & \OEIS{A091363} & \OEIS{A091364} & ? & & & 
	\end{tabular}}
	\caption{Number of vertices of~$\Zono[K_m] \shuffleDP \Zono[P_n] = \Perm[m] \shuffleDP \Para[n]$.}
	\label{table:verticesPermCube}
\end{table}

\begin{table}[h]
	\centerline{\begin{tabular}{r|rrrrrrrrrr|l}
		$m \backslash n$ & 0 & 1 & 2 & 3 & 4 & 5 & 6 & 7 & 8 & 9 & \\
		\hline
		0 & . & 1 & 2 & 4 & 6 & 8 & 10 & 12 & 14 & 16 & \OEIS{A000027} \\
		1 & 1 & 2 & 6 & 12 & 20 & 30 & 42 & 56 & 72 & & \OEIS{A002378} \\
		2 & 2 & 6 & 14 & 28 & 52 & 94 & 170 & 312 & & & \OEIS{A290699} \\
		3 & 6 & 14 & 30 & 60 & 116 & 222 & 426 & & & & \OEIS{A308580} \\
		4 & 14 & 30 & 62 & 124 & 244 & 478 & & & & & ? \\
		5 & 30 & 62 & 126 & 252 & 500 & & & & & & \\
		6 & 62 & 126 & 254 & 508 & & & & & & & \\
		7 & 126 & 254 & 510 & & & & & & & & \\
		8 & 254 & 510 & & & & & & & & & \\
		9 & 510 & & & & & & & & & & \\
		\hline
		& \OEIS{A000918} & \OEIS{A000918} & \OEIS{A000918} & \OEIS{A028399} & \OEIS{A173034} & ? & & & & &
	\end{tabular}}
	\caption{Number of facets of~$\Zono[K_m] \shuffleDP \Zono[P_n] = \Perm[m] \shuffleDP \Para[n]$.}
	\label{table:facetsPermCube}
\end{table}

\begin{table}[h]
	\centerline{\begin{tabular}{r|rrrrrrrrrr|l}
		$m \backslash n$ & 0 & 1 & 2 & 3 & 4 & 5 & 6 & 7 & 8 & 9 & \\
		\hline
		0 & . & 1 & 1 & 1 & 1 & 1 & 1 & 1 & 1 & 1 & \OEIS{A000012} \\ 
		1 & 1 & 2 & 4 & 8 & 16 & 32 & 64 & 128 & 256 & & \OEIS{A000079} \\
		2 & 2 & 6 & 18 & 54 & 162 & 486 & 1458 & 4374 & & & \OEIS{A008776}, \OEIS{A025192} \\
		3 & 6 & 24 & 96 & 384 & 1536 & 6144 & 24576 & & & & \OEIS{A002023} \\
		4 & 24 & 120 & 600 & 3000 & 15000 & 75000 & & & & & \OEIS{A235702} \\
		5 & 120 & 720 & 4320 & 25920 & 155520 & & & & & & ? \\
		6 & 720 & 5040 & 35280 & 246960 & & & & & & & \\
		7 & 5040 & 40320 & 322560 & & & & & & & & \\
		8 & 40320 & 362880 & & & & & & & & \\
		9 & 362880 & & & & & & & & & & \\
		\hline
		& \OEIS{A000142} & \OEIS{A000142} & \OEIS{A001563} & \OEIS{A002775} & \OEIS{A091363} & \OEIS{A091364} & ? & & & & \\
	\end{tabular}}
	\caption{Number of vertices of~$\Zono[K_m] \shuffleDP \Zono[E_n] = \Perm[m] \shuffleDP \Point[n]$.}
	\label{table:verticesPermPoint}
\end{table}

\begin{table}[h]
	\centerline{\begin{tabular}{r|rrrrrrrrrr|l}
		$m \backslash n$ & 0 & 1 & 2 & 3 & 4 & 5 & 6 & 7 & 8 & 9 & \\
		\hline
		0 & . & 1 & 1 & 1 & 1 & 1 & 1 & 1 & 1 & 1 & \OEIS{A000012} \\
		1 & 1 & 2 & 4 & 6 & 8 & 10 & 12 & 14 & 16 & & \OEIS{A005843} \\
		2 & 2 & 6 & 12 & 22 & 40 & 74 & 140 & 270 & & & \OEIS{A131520} \\
		3 & 6 & 14 & 28 & 54 & 104 & 202 & 396 & & & & ? \\
		4 & 14 & 30 & 60 & 118 & 232 & 458 & & & & & \\
		5 & 30 & 62 & 124 & 246 & 488 & & & & & & \\
		6 & 62 & 126 & 252 & 502 & & & & & & & \\
		7 & 126 & 254 & 508 & & & & & & & & \\
		8 & 254 & 510 & & & & & & & & & \\
		9 & 510 & & & & & & & & & & \\
		\hline
		& \OEIS{A000918} & \OEIS{A000918} & \OEIS{A028399} & \OEIS{A246168} & ? & & & & & & \\
	\end{tabular}}
	\caption{Number of facets of~$\Zono[K_m] \shuffleDP \Zono[E_n] = \Perm[m] \shuffleDP \Point[n]$.}
	\label{table:facetsPermPoint}
\end{table}

\begin{table}[h]
	\centerline{\begin{tabular}{r|rrrrrrrrrr|l}
		$m \backslash n$ & 0 & 1 & 2 & 3 & 4 & 5 & 6 & 7 & 8 & 9 & \\
		\hline
		0 & . & 1 & 1 & 1 & 1 & 1 & 1 & 1 & 1 & 1 & \OEIS{A000012} \\
		1 & 1 & 2 & 4 & 8 & 16 & 32 & 64 & 128 & 256 & & \OEIS{A000079} \\
		2 & 1 & 4 & 14 & 46 & 146 & 454 & 1394 & 4246 & & & \OEIS{A027649} \\
		3 & 1 & 8 & 46 & 230 & 1066 & 4718 & 20266 & & & & \OEIS{A027650} \\
		4 & 1 & 16 & 146 & 1066 & 6902 & 41506 & & & & & \OEIS{A027651} \\
		5 & 1 & 32 & 454 & 4718 & 41506 & & & & & & \OEIS{A283811} \\
		6 & 1 & 64 & 1394 & 20266 & & & & & & & \OEIS{A283812} \\
		7 & 1 & 128 & 4246 & & & & & & & & \OEIS{A283813} \\
		8 & 1 & 256 & & & & & & & & & \OEIS{A284032} \\
		9 & 1 & & & & & & & & & & \OEIS{A284033} \\
		\hline
		& \OEIS{A000012} & \OEIS{A000079} & \OEIS{A027649} & \OEIS{A027650} & \OEIS{A027651} & \OEIS{A283811} & \OEIS{A283812} & \OEIS{A283813} & \OEIS{A284032} & \OEIS{A284033} & \\
	\end{tabular}}
	\caption{Number of vertices of~$\Zono[E_m] \shuffleDP \Zono[E_n] = \Point[m] \shuffleDP \Point[n]$.}
	\label{table:verticesPointPoint}
\end{table}

\begin{table}[h]
	\centerline{\begin{tabular}{r|rrrrrrrrrr|l}
		$m \backslash n$ & 0 & 1 & 2 & 3 & 4 & 5 & 6 & 7 & 8 & 9 & \\
		\hline
		0 & . & 1 & 1 & 1 & 1 & 1 & 1 & 1 & 1 & 1 & \OEIS{A000012} \\
		1 & 1 & 2 & 4 & 6 & 8 & 10 & 12 & 14 & 16 & & \OEIS{A005843} \\
		2 & 2 & 4 & 12 & 22 & 40 & 74 & 140 & 270 & & & \OEIS{A131520} \\
		3 & 4 & 6 & 22 & 48 & 98 & 196 & 390 & & & & ? \\
		4 & 6 & 8 & 40 & 98 & 212 & 438 & & & & & \\
		5 & 8 & 10 & 74 & 196 & 438 & & & & & & \\
		6 & 10 & 12 & 140 & 390 & & & & & & & \\
		7 & 12 & 14 & 270 & & & & & & & & \\
		8 & 14 & 16 & & & & & & & & & \\
		9 & 16 & & & & & & & & & & \\
		\hline
		& \OEIS{A005843} & \OEIS{A005843} & \OEIS{A131520} & ? & & & & & & & \\
	\end{tabular}}
	\caption{Number of facets of~$\Zono[E_m] \shuffleDP \Zono[E_n] = \Point[m] \shuffleDP \Point[n]$.}
	\label{table:facetsPointPoint}
\end{table}


\clearpage
\subsection{Multiplihedra} ~
\label{subsec:tablesMultiplihedra}

\begin{table}[h]
	\centerline{\begin{tabular}{r|rrrrrrrrrr|l}
		$m \backslash n$ & 0 & 1 & 2 & 3 & 4 & 5 & 6 & 7 & 8 & 9 & \\
		\hline
		0 & . & 1 & 2 & 5 & 14 & 42 & 132 & 429 & 1430 & 4862 & \OEIS{A000108} \\
		1 & 1 & 2 & 6 & 21 & 80 & 322 & 1348 & 5814 & 25674 & & \OEIS{A121988} \\
		2 & 2 & 6 & 24 & 108 & 520 & 2620 & 13648 & 72956 & & & $2 \cdot \OEIS{A158826}$ \\
		3 & 6 & 24 & 120 & 660 & 3840 & 23220 & 144504 & & & & ? \\
		4 & 24 & 120 & 720 & 4680 & 31920 & 225120 & & & & & \\
		5 & 120 & 720 & 5040 & 37800 & 295680 & & & & & & \\
		6 & 720 & 5040 & 40320 & 342720 & & & & & & & \\
		7 & 5040 & 40320 & 362880 & & & & & & & & \\
		8 & 40320 & 362880 & & & & & & & & & \\
		9 & 362880 & & & & & & & & & & \\
		\hline
		& \OEIS{A000142} & \OEIS{A000142} & \OEIS{A000142} & \OEIS{A084253} & ? & & & & & & $m! \cdot \OEIS{A158825}$
	\end{tabular}}
	\caption{Number of vertices of the multiplihedra~$\Multiplihedron \eqdef \Perm[m] \shuffleDP \Asso[n]$. See~$\OEIS{A158825}$.}
	\label{table:verticesMultiplihedra}
\end{table}

\begin{table}[h]
	\centerline{\begin{tabular}{r|rrrrrrrrrr|l}
		$m \backslash n$ & 0 & 1 & 2 & 3 & 4 & 5 & 6 & 7 & 8 & 9 & \\
		\hline
		0 & . & 1 & 2 & 5 & 9 & 14 & 20 & 27 & 35 & 44 & \OEIS{A000096} \\
		1 & 1 & 2 & 6 & 13 & 25 & 46 & 84 & 155 & 291 & & \OEIS{A335439} \\
		2 & 2 & 6 & 14 & 29 & 57 & 110 & 212 & 411 & & & ? \\
		3 & 6 & 14 & 30 & 61 & 121 & 238 & 468 & & & & \\
		4 & 14 & 30 & 62 & 125 & 249 & 494 & & & & & \\
		5 & 30 & 62 & 126 & 253 & 505 & & & & & & \\
		6 & 62 & 126 & 254 & 509 & & & & & & & \\
		7 & 126 & 254 & 510 & & & & & & & & \\
		8 & 254 & 510 & & & & & & & & & \\
		9 & 510 & & & & & & & & & & \\
		\hline
		& \OEIS{A000918} & \OEIS{A000918} & \OEIS{A000918} & \OEIS{A036563} & \OEIS{A048490} & ? & & & & &
	\end{tabular}}
	\caption{Number of facets of the multiplihedra~$\Multiplihedron \eqdef \Perm[m] \shuffleDP \Asso[n]$.}
	\label{table:facetsMultiplihedra}
\end{table}

\begin{table}[h]
	\centerline{\begin{tabular}{r|rrrrrrrrrr|l}
		$m \backslash n$ & 0 & 1 & 2 & 3 & 4 & 5 & 6 & 7 & 8 & 9 & \\
		\hline
		0 & . & 1 & 3 & 11 & 45 & 197 & 903 & 4279 & 20793 & 103049 & \OEIS{A001003} \\
		1 & 1 & 3 & 13 & 67 & 381 & 2311 & 14681 & 96583 & 653049 & & ? \\
		2 & 3 & 13 & 75 & 497 & 3583 & 27393 & 218871 & 1810373 & & & \\
		3 & 13 & 75 & 541 & 4375 & 38073 & 349423 & 3341753 & & & & \\
		4 & 75 & 541 & 4683 & 44681 & 454855 & 4859697 & & & & & \\
		5 & 541 & 4683 & 47293 & 519847 & 6055401 & & & & & & \\
		6 & 4683 & 47293 & 545835 & 6790697 & & & & & & & \\
		7 & 47293 & 545835 & 7087261 & & & & & & & & \\
		8 & 545835 & 7087261 & & & & & & & & & \\
		9 & 7087261 & & & & & & & & & & \\
		\hline
		& \OEIS{A000670} & \OEIS{A000670} & \OEIS{A000670} & ? & & & & & & &
	\end{tabular}}
	\caption{Total number of faces of the multiplihedra~$\Multiplihedron \eqdef \Perm[m] \shuffleDP \Asso[n]$. The empty face is not counted, but the polytope itself is.}
	\label{table:facesMultiplihedra}
\end{table}

\clearpage
\subsection{Constrainahedra and biassociahedra} ~
\label{subsec:tablesConstrainahedraBiassociahedra}

\begin{table}[h]
	\centerline{\begin{tabular}{r|rrrrrrrrrr|l}
		$m \backslash n$ & 0 & 1 & 2 & 3 & 4 & 5 & 6 & 7 & 8 & 9 & \\
		\hline
		0 & . & 1 & 2 & 5 & 14 & 42 & 132 & 429 & 1430 & 4862 & \OEIS{A000108} \\ 
		1 & 1 & 2 & 6 & 21 & 80 & 322 & 1348 & 5814 & 25674 & & \OEIS{A121988} \\
		2 & 2 & 6 & 24 & 108 & 520 & 2620 & 13648 & 72956 & & & $2\cdot$\OEIS{A158826} \\
		3 & 5 & 21 & 108 & 606 & 3580 & 21910 & 137680 & & & & ? \\
		4 & 14 & 80 & 520 & 3580 & 25520 & 186420 & & & & & \\
		5 & 42 & 322 & 2620 & 21910 & 186420 & & & & & & \\
		6 & 132 & 1348 & 13648 & 137680 & & & & & & & \\
		7 & 429 & 5814 & 72956 & & & & & & & & \\
		8 & 1430 & 25674 & & & & & & & & & \\
		9 & 4862 & & & & & & & & & & \\
		\hline
		& \OEIS{A000108} & \OEIS{A121988} & $2\cdot$\OEIS{A158826} & ? & & & & & & &
	\end{tabular}}
	\caption{Number of vertices of the constrainahedra~$\Constrainahedron \eqdef \Asso[m] \shuffleDP \Asso[n]$ and of the biassociahedra~$\Biassociahedron \eqdef \Ossa[m] \shuffleDP \Asso[n]$.}
	\vspace{-.3cm}
	\label{table:verticesConstrainahedraBiassociahedra}
\end{table}

\begin{table}[h]
	\centerline{\begin{tabular}{r|rrrrrrrrrr|l}
		$m \backslash n$ & 0 & 1 & 2 & 3 & 4 & 5 & 6 & 7 & 8 & 9 & \\
		\hline
		0 & . & 0 & 2 & 5 & 9 & 14 & 20 & 27 & 35 & 44 & \OEIS{A000096} \\
		1 & 0 & 2 & 6 & 13 & 25 & 46 & 84 & 155 & 291 & & \OEIS{A335439} \\
		2 & 2 & 6 & 14 & 29 & 57 & 110 & 212 & 411 & & & ? \\
		3 & 5 & 13 & 29 & 60 & 120 & 237 & 467 & & & \\
		4 & 9 & 25 & 57 & 120 & 244 & 489 & & & & \\
		5 & 14 & 46 & 110 & 237 & 489 & & & & & \\
		6 & 20 & 84 & 212 & 467 & & & & & & \\
		7 & 27 & 155 & 411 & & & & & & & \\
		8 & 35 & 291 & & & & & & & & \\
		9 & 44 & & & & & & & & & \\
		\hline
		& \OEIS{A000096} & \OEIS{A335439} & ? 
	\end{tabular}}
	\caption{Number of facets of the constrainahedra~$\Constrainahedron \eqdef \Asso[m] \shuffleDP \Asso[n]$ and of the biassociahedra~$\Biassociahedron \eqdef \Ossa[m] \shuffleDP \Asso[n]$.}
	\vspace{-.3cm}
	\label{table:facetsConstrainahedraBiassociahedra}
\end{table}

\begin{table}[h]
	\centerline{\begin{tabular}{r|rrrrrrrrrr|l}
		$m \backslash n$ & 0 & 1 & 2 & 3 & 4 & 5 & 6 & 7 & 8 & 9 & \\
		\hline
		0 & . & 1 & 3 & 11 & 45 & 197 & 903 & 4279 & 20793 & 103049 & \OEIS{A001003} \\
		1 & 1 & 3 & 13 & 67 & 381 & 2311 & 14681 & 96583 & 653049 & & ? \\
		2 & 3 & 13 & 75 & 497 & 3583 & 27393 & 218871 & 1810373 & & & \\
		3 & 11 & 67 & 497 & 4099 & 36205 & 336107 & 3243085 & & & & \\
		4 & 45 & 381 & 3583 & 36205 & 384819 & 4251605 & & & & & \\
		5 & 197 & 2311 & 27393 & 336107 & 4251605 & & & & & & \\
		6 & 903 & 14681 & 218871 & 3243085 & & & & & & & \\
		7 & 4279 & 96583 & 1810373 & & & & & & & & \\
		8 & 20793 & 653049 & & & & & & & & & \\
		9 & 103049 & & & & & & & & & & \\
		\hline
		& \OEIS{A001003} & ? & & & & & & & & &
	\end{tabular}}
	\caption{Total number of faces of the constrainahedra~$\Constrainahedron \eqdef \Asso[m] \shuffleDP \Asso[n]$. The empty face is not counted, but the polytope itself is.}
	\label{table:facesConstrainahedra}
\end{table}

\clearpage
\begin{table}[h]
	\centerline{\begin{tabular}{r|rrrrrrrrrr|l}
		$m \backslash n$ & 0 & 1 & 2 & 3 & 4 & 5 & 6 & 7 & 8 & 9 & \\
		\hline
		0 & . & 1 & 3 & 11 & 45 & 197 & 903 & 4279 & 20793 & 103049 & \OEIS{A001003} \\
		1 & 1 & 3 & 13 & 67 & 381 & 2311 & 14681 & 96583 & 653049 & & ? \\
		2 & 3 & 13 & 75 & 497 & 3583 & 27393 & 218871 & 1810373 & & & \\
		3 & 11 & 67 & 497 & 4095 & 36137 & 335287 & 3234433 & & & & \\
		4 & 45 & 381 & 3583 & 36137 & 383375 & 4229985 & & & & & \\
		5 & 197 & 2311 & 27393 & 335287 & 4229985 & & & & & & \\
		6 & 903 & 14681 & 218871 & 3234433 & & & & & & & \\
		7 & 4279 & 96583 & 1810373 & & & & & & & & \\
		8 & 20793 & 653049 & & & & & & & & & \\
		9 & 103049 & & & & & & & & & & \\
		\hline
		& \OEIS{A001003} & ? & & & & & & & & &
	\end{tabular}}
	\caption{Total number of faces of the biassociahedra~$\Biassociahedron \eqdef \Ossa[m] \shuffleDP \Asso[n]$. The empty face is not counted, but the polytope itself is.}
	\label{table:facesBiassociahedra}
\end{table}

\end{document}